\documentclass{compositio}

\usepackage{amsmath,amsfonts,yhmath,hyperref}
\usepackage{amsthm}
\usepackage{amssymb}
\usepackage{mathrsfs}
\usepackage[latin9]{inputenc}
\usepackage[nottoc,notlof,notlot]{tocbibind}
\usepackage{mathtools}
\usepackage{enumitem}
\usepackage{tikz-cd}
\usepackage{multirow}

\usepackage{float}

\usepackage{pgf,tikz}

\usetikzlibrary{arrows}
\usetikzlibrary[patterns]

\definecolor{qqqqff}{rgb}{0.,0.,1.}
\definecolor{cqcqcq}{rgb}{0.7529411764705882,0.7529411764705882,0.7529411764705882}
\definecolor{ttqqqq}{rgb}{0.2,0.,0.}
\definecolor{qqqqff}{rgb}{0.,0.,1.}
\definecolor{xdxdff}{rgb}{0.49019607843137253,0.49019607843137253,1.}
\definecolor{zzttqq}{rgb}{0.6,0.2,0.}
\definecolor{cqcqcq}{rgb}{0.7529411764705882,0.7529411764705882,0.7529411764705882}

\definecolor{yqyqyq}{rgb}{0.5019607843137255,0.5019607843137255,0.5019607843137255}
\definecolor{uuuuuu}{rgb}{0.26666666666666666,0.26666666666666666,0.26666666666666666}
\definecolor{xdxdff}{rgb}{0.49019607843137253,0.49019607843137253,1.}
\definecolor{qqqqff}{rgb}{0.,0.,1.}

 \font\ncsc=cmcsc10

\newcommand{\PP}{\mathbb{P}}
\newcommand{\NN}{\mathbb{N}}
\newcommand{\ZZ}{\mathbb{Z}}
\newcommand{\RR}{\mathbb{R}}
\newcommand{\CC}{\mathbb{C}}

\newcommand{\QQ}{\mathbb{Q}}

\newcommand{\TT}{\mathbb{T}}

\newcommand{\Dfk}{\mathfrak{D}}
\newcommand{\mfk}{\mathfrak{m}}

\renewcommand{\div}{\mathrm{div} }

\newcommand{\E}{\CC E}
\renewcommand{\P}{\mathcal{P}}
\renewcommand{\L}{\mathcal{L}}
\newcommand{\N}{\mathcal{N}}
\newcommand{\M}{\mathcal{M}}

\newcommand{\X}{\mathcal{X}}
\newcommand{\T}{\mathcal{T}}

\newcommand{\Q}{\mathcal{Q}}
\newcommand{\V}{\mathcal{V}}
\renewcommand{\H}{\mathcal{H}}
\renewcommand{\O}{\mathcal{O}}
\newcommand{\F}{\mathcal{F}}

\newcommand{\CCC}{\mathscr{C}}

\newcommand{\val}{\mathrm{val}}

\newcommand{\bino}[2]{\begin{pmatrix}
#1 \\
#2 \\
\end{pmatrix}}

\newtheorem{theo}{Theorem}[section]
\newtheorem*{theom}{Theorem}
\newtheorem{prop}[theo]{Proposition}
\newtheorem{coro}[theo]{Corollary}
\newtheorem{lem}[theo]{Lemma}

\theoremstyle{definition}
\newtheorem{defi}[theo]{Definition}
\newtheorem{prob}[theo]{Problem}

\theoremstyle{remark}
\newtheorem{remark}[theo]{Remark}
\newenvironment{rem}[1]{
    \begin{remark}#1}{
    \xqed{\blacklozenge}\end{remark}
}

\theoremstyle{remark}
\newtheorem{example}[theo]{Example}
\newenvironment{expl}[1]{
    \begin{example}#1}{
    \xqed{\lozenge}\end{example}
}

\newcommand{\xqed}[1]{
    \leavevmode\unskip\penalty9999 \hbox{}\nobreak\hfill
    \quad\hbox{\ensuremath{#1}}}

\keywords{Enumerative geometry, tropical refined invariants\\ \textit{Data Statement:} I do not have any data to point.\\
Thomas Blomme,
{\ncsc Universit\'e de Gen\`eve, 5-7 rue du Conseil G\'en\'eral,
1205 Gen\`eve,
Switzerland} \\
\textit{Email :} thomas.blomme@unige.ch}
 
\begin{document}
 
%\renewcommand{\qedsymbol}{\filledbox}
%Good resources for looking up how to do stuff:
%Binary operators: http://www.access2science.com/latex/Binary.html
%General help: http://en.wikibooks.org/wiki/LaTeX/Mathematics
%Or just google stuff
 
\title{Floor diagrams and enumerative invariants of line bundles over an elliptic curve}
\author{Thomas Blomme}

\begin{abstract}
We use the tropical geometry approach to compute absolute and relative enumerative invariants of complex surfaces which are $\CC P^1$-bundles over an elliptic curve. We also show that the tropical multiplicity used to count curves can be refined by the standard Block-G\"ottsche refined multiplicity to give tropical refined invariants. We then give a concrete algorithm using floor diagrams to compute these invariants along with the associated interpretation as operators acting on some Fock space. The floor diagram algorithm allows one to prove the piecewise polynomiality of the relative invariants, and the quasi-modularity of their generating series.
\end{abstract}

\maketitle

\tableofcontents

\section{Introduction}

\subsection{Overview}

The goal of this paper is to study classical enumerative invariants of some ruled surfaces over an elliptic curve $\E$, generalizing the case of the trivial $\CC P^1$-bundle $\E\times\CC P^1$ studied by J. B\"ohm, C. Goldner and H. Markwig in \cite{bohm2020counts}. The principal tool is a correspondence theorem that relates the classical enumerative invariants to tropical counts of curves in suitable affine integer manifolds. Once in the scope of tropical geometry, it is then possible to use combinatorial techniques to relate the tropical counts to counts of floor diagrams and compute generating series using a Fock space approach.

\medskip	

\textbf{Enumerative geometry.} The varieties we are interested in are ruled surfaces over an elliptic curve. The latter are classified in \cite{hartshorne2013algebraic}. According to the classification, all but two are of the following form: take an elliptic curve $\E$ and a line bundle $\mathcal{L}$ over it, and then consider the projectivization $\PP(\mathcal{L}\oplus\mathcal{O})$ to get a complex surface that is a $\CC P^1$-bundle over $\E$. Such a surface is denoted by $\CC F$. If the degree $\delta$ of the line bundle is non-zero, it can assumed to be positive and there is only one surface up to isomorphism which is thus denoted by $\CC F_\delta$. If the degree is $0$, surfaces are classified by $\alpha\in\E\simeq\mathrm{Pic}\E$ up to inversion and the corresponding surface is denoted by $\CC F_{0,\alpha}$.

\begin{rem}
Theorems 2.12 and 2.15 in Surfaces chapter from \cite{hartshorne2013algebraic} classify $\CC P^1$-bundles over an elliptic curve, which are by definition obtained by projectivizing a rank $2$ vector bundle. In this paper, we only consider the surfaces associated to the decomposable ones. There are exactly two other surfaces up to isomorphism obtained from undecomposable one, and are not handled in this paper.
\end{rem}

Constructed as a projectivization $\PP(\L\oplus\O)$, the surfaces $\CC F$ contain several curves: the distinguished $0$-section $E_0$ image of $e\in\E\mapsto [0:1]$ and the $\infty$-section $E_\infty$ image of $z\in\E\mapsto[l:0]$. It also contains the fibers of $\CC F\to\E$. The second homology group of $\CC F$ generated by the class of a fiber $F$ and the zero section $E_0$. We can then count the number of genus $g$ curves in some homology class $d_1 E_0+d_2 F$ passing through a suitable number of points. Such a count happens not to depend on the choice of the points as long as it is generic, nor the choice of the elliptic curve. It only depends on the choice of $\L$ through its degree $\delta$, as long as $\L$ is generically chosen if this degree is $0$. The paper aims at providing a concrete way to compute these invariants. Similarly, it is also possible to count curves with a fixed tangency profile with the divisors $E_0$ and $E_\infty$. One thus obtains instead so-called \textit{relative} invariants.

\medskip

\textbf{Correspondence theorems and tropical geometry.} In the case of toric surfaces, a major breakthrough in the computation of enumerative invariants was realized in the beginning of the 21st century with the use of \textit{Tropical Geometry}. In \cite{mikhalkin2005enumerative}, G. Mikhalkin proved a correspondence theorem that relates counts of classical curves passing through some specified points, and counts of tropical curves in $\RR^2$ satisfying the same kind of conditions. Using different techniques, namely log-geometry, T. Nishinou and B. Siebert \cite{nishinou2006toric} also proved a similar correspondence theorem. Both approaches got generalized to various settings with new different kinds of constraints, for instance with the insertion of $\psi$-constraints. See \cite{shustin2002patchworking}, \cite{tyomkin2012tropical},\cite{tyomkin2017enumeration}, \cite{mandel2020descendant} or \cite{cavalieri2021counting} for generalizations or other proofs of the correspondence theorem.

Unfortunately, the tropical geometry approach is often restricted due to the possible appearance of so-called \textit{superabundant tropical curves} in some situations. For such curves, the dimension of the space in which the tropical counterparts to the classical curves vary does not match the dimension of the associated complex space. This complicates the correspondence. For instance, in \cite{nishinou2020realization}, T. Nishinou gives a correspondence theorem for curves in abelian surfaces, which are all superabundant. In our case, superabundant curves can occur for the spaces $\TT F_{0,\alpha}$ coming from a torsion element $\alpha$. In the case of the trivial bundle handled in \cite{bohm2020counts}, this is avoided due to the fact that the superabundance comes from curves that realize sections of the line bundle, for which the realization theorem is easily proved.

\medskip

\textbf{Tropical geometry and concrete computations.} The use of a suitable \textit{correspondence theorem} allows one to reduce a complex count to a tropical count, whose computation is still waiting to be carried out. Computing the tropical invariants can be a rather difficult combinatorial task. In the case of invariants in toric surfaces, the tropical Caporaso-Harris formula, which is proved in the tropical setting in \cite{gathmann2007caporaso}, allows one to carry out the computation for the case of the projective plane $\CC P^2$. It also computes enumerative invariants relative to a line. The idea of the tropical proof of the formula is to choose the points constraints in a degenerated way, so that the tropical curves split into simple pieces that are easier to handle. An iterated application of the Caporaso-Harris formula leads the tropical curve that we are looking for to break into small pieces and adopt a so-called \textit{floor decomposition}. This amounts to choose the points through which the tropical curves pass in a stretched position. The fact that the curve break allows to project the problem onto the direction in which the points are stretched and reduce the 2-dimensional enumerative problem to a 1-dimensional enumerative problem.

\medskip

\textbf{Concrete computations and floor diagrams.} Following the iteration of the Caporaso-Harris formula, it only suffices to handle tropical curves passing through a stretched position. To this end, E. Brugall\'e and G. Mikhalkin introduced \textit{floor diagrams} in \cite{brugalle2007enumeration} and \cite{brugalle2008floor}. These were later used for instance by G. Mikhalkin and S. Fomin in \cite{mikhalkin2010labeled}. They enable concrete, and to some extent efficient, computations of tropical invariants. The use of floor diagrams has since been generalized to other situations, for instance by F. Block and L. G\"ottsche \cite{block2016fock} to refined enumerative invariants, and by R. Cavalieri, P. Johnson, H. Markwig and D. Ranganathan \cite{cavalieri2018graphical}\cite{cavalieri2021counting} to compute descendant Gromov-Witten invariants of Hirzebruch surfaces.

\smallskip

Presented this way, floor diagrams may appear to be a mainly toric object, taking its meaning in tropical geometry. Yet, they alo appear in non-toric situations, see \cite{brugalle2015floor}, where the heuristic of floor diagrams in more general situations is explained.

\medskip

\textbf{Floor diagrams and Fock space.} The use of floor diagrams essentially reduces a $2$-dimensional enumerative problem to a $1$-dimensional one. It reduces a problem on tropical curves to a problem on tropical covers. Tropical covers are maps between graphs studied in \cite{caporaso2014gonality} and \cite{cavalieri2010tropical}. It also emphasizes some deep relation between planar geometry, and operators on a Fock space. This approach was pioneered by Y. Cooper and R. Pandharipande in \cite{cooper2017fock}, where they relate Severi degrees of $\CC P^1\times\CC P^1$ and $\CC P^2$ to matrix elements of exponential of some operator on a Fock space. This approach was generalized in \cite{block2016fock} to a wider family of toric surfaces, and in \cite{cavalieri2021counting} for Hirzebruch surfaces with the insertion of descendant conditions.

\smallskip

The dimension reduction coming from the use of floor diagrams essentially comes from the existence of a projection. In the case of toric surfaces, although it is possible to project in several directions, each time one obtains a problem on a line. When we do enumeration in for instance the trivial bundle $\TT E\times\TT P^1$, there are now two directions in which one can projects: the vertical direction projects onto a line, and the horizontal direction projects onto a circle. In \cite{bohm2020counts}, the authors use the projection to $\TT E$ to relate enumerative invariants of $\TT E\times\TT P^1$ to some counts over an elliptic curve, relating to some other problems they studied in \cite{boehm2018tropical}. Such projection over a circle also relates to a Fock space interpretation using a \textit{trace formula} which was already noticed in \cite{cooper2017fock} to compute some invariants of $\E\times\CC P^1$.

\medskip

\textbf{Tropical geometry and refined invariants.} The correspondence theorem \cite{mikhalkin2005enumerative} for toric surfaces expresses the multiplicity of a tropical curve as a product over the vertices of the tropical curve. In \cite{block2016refined}, F. Block and L. G\"ottsche proposed to replace the vertices multiplicity $m$ by their quantum analog $[m]_q=\frac{q^{m/2}-q^{-m/2}}{q^{1/2}-q^{-1/2}}$, which is a symmetric Laurent polynomial in the variable $q$. This new multiplicity is called \textit{refined}, since it specializes to the usual complex multiplicity when $q$ is set equal to $1$. In \cite{itenberg2013block}, I. Itenberg and G. Mikhalkin proved that for the tropical enumerative problem of counting curves of fixed genus in a toric surface passing through a suitable number of points, the refined multiplicity from Block-G\"ottsche also provides an invariant. Furthermore, the setting in which F. Block and L. G\"ottsche extended the results from \cite{cooper2017fock} in \cite{block2016fock} is already adapted to their refined multiplicities, as they consider a deformed version of the Fock space and Heisenberg algebra from \cite{cooper2017fock}. This way, their computations yield the refined invariants versions of Severi degrees considered in \cite{block2016refined} and \cite{itenberg2013block} rather than the usual ones. The meaning of these tropical refined invariants remains an intense area of studies \cite{gottsche2014refined}, \cite{bousseau2019tropical}, \cite{mikhalkin2017quantum}, \cite{blomme2021refinedreal} that have already been generalized to different settings \cite{blomme2021refinedtrop}, \cite{gottsche2019refined}, \cite{blechman2019refined}, \cite{schroeter2018refined}, but remain quite mysterious.

\subsection{Results}

We now develop the various results of the paper, and their role in the previous considerations.

\medskip

\textbf{Complex setting.} We consider enumeration of curves in $\CC F$, which is either $\CC F_\delta$ or some $\CC F_{0,\alpha}$ for a generic choice of $\alpha$. A curve in $\CC F_\delta$ is said to be of bidegree $(d_1,d_2)$ if it realizes the homology class $d_1 E_0+d_2 F$, where $E_0$ is the class of the zero section, and $F$ the class of a fiber. The number of genus $g$ irreducible curves of bidegree $(d_1,d_2)$ passing through $\delta d_1+2d_2+g-1$ points in general position is denoted by $N_{g,(d_1,d_2)}^{\delta,\mathrm{cpx}}$. The count including reducible curves is denoted with a $\bullet$. This count happens not no depend on the choice of the point configuration as long as it is generic, nor the choice of the elliptic curve, and only depends on the surface through the degree $\delta$. See Section \ref{section complex setting} for details.

\begin{rem}
The genus of a curve $\CC C$ is defined to be $g=1-\chi(\CC C)$. In particular, the genus of a reducible curves with $r$ components $\CC C_i$ is $1-r+\sum_1^r g(\CC C_i)$.
\end{rem}

A curve of bidegree $(d_1,d_2)$ has $\delta d_1+d_2$ intersection points with the $0$-section $E_0$, and $d_2$ intersection points with the $\infty$-section $E_\infty$. Thus, one can further look for genus $g$ curves of bidegree $(d_1,d_2)$ that have a specified intersection profile with $E_0$ and $E_\infty$. Such a data is given by two partitions $\mu_0\vdash \delta d_1+d_2$ and $\mu_\infty\vdash d_2$. A curve is said to have tangency profile $(\mu_0,\mu_\infty)$ if it has $\mu_{0i}$ (resp. $\mu_{\infty i}$) points of tangency index $i$ with $E_0$ (resp. $E_\infty$).

\smallskip

We want to impose additional conditions of two types: having a tangency of order $i$ at a specific point of $E_0$ (resp. $E_\infty$), having a tangency of order $i$ somewhere on $E_0$ (resp. $E_\infty$). To this extent, let $\mu_0,\nu_0,\mu_\infty,\nu_\infty$ be four partitions such that $\mu_0+\nu_0\vdash\delta d_1+d_2$ and $\mu_\infty+\nu_\infty\vdash d_2$ ($\vdash n$ meaning partition $n$, see end of Introduction for notations about partitions.), one denotes by $N_{g,(d_1,d_2)}^{\delta,\mathrm{cpx}}(\mu_0,\mu_\infty,\nu_0,\nu_\infty)$ the number of irreducible curves passing through $|\nu_0|+|\nu_\infty|+g-1$ points in generic position such that:
\begin{itemize}[label=-]
\item have $\mu_{0i}$ (resp. $\mu_{\infty i}$) intersection points of tangency index $i$ at specified points of $E_0$ (resp. $E_\infty$),
\item have $\nu_{0i}$ (resp. $\nu_{\infty i}$) more intersection points of tangency index $i$ at unspecified points of $E_0$ (resp. $E_\infty$). 
\end{itemize}
Similarly, the number with a $\bullet$ denotes the number including the reducible curves. This is exactly the analog of the enumerative problem defining the invariants $N_{d,g}(\alpha,\beta)$ involved in the Caporaso-Harris formula from \cite{gathmann2007caporaso}. They also do not depend on the choice of the constraints as long as they are generic, nor the choice of the elliptic curve, and they only depend on $\CC F$ through the degree $\delta$.

\medskip

\textbf{Tropical setting.} In the tropical world, one defines tropical counts of genus $g$ tropical curves of bidegree $(d_1,d_2)$ in the tropical surface $\TT F$, which is the total space of some tropical line bundle over a tropical elliptic curve $\TT E$. Briefly, they are constructed as follows, see section \ref{section tropical line bundle elliptic curve} for more details.
\begin{itemize}[label=-]
\item In the case of the trivial line bundle, this total space is just $\TT E\times\RR\subset \TT E\times\TT P^1$, where $\TT E$ is a tropical elliptic curve. Topologically, a tropical elliptic curve is just a circle. Thus, $\TT E\times\RR$ is an open cylinder, and $\TT E\times\TT P^1$ is a compact cylinder. One other way to obtain it is to take the strip $[0;l]\times \RR\subset\RR^2$ and to identify the points $(0,y)$ and $(l,y)$.
\item For a non-trivial degree $0$ line bundle, we still take the strip $[0;l]\times \RR\subset\RR^2$ but identify the two sides by a translation of size $\alpha$. Only the value of $\alpha$ mod $l$ is important, and we get a cylinder $\TT F_{0,\alpha}$.
\item If the line bundle is of non-zero degree, the total space $\TT F_\delta$ is still a cylinder but with a non-trivial integer structure. It can be constructed from the strip $[0;l+\varepsilon[\times\RR$ identifying the small strips $[0;\varepsilon[\times\RR$ and $[l;l+\varepsilon[\times\RR$ by a shear transformation. In other words, still viewing it as some quotient of the strip $[0;l]\times\RR$ identifying both sides, a line that crosses the right boundary component of the strip comes back from the left boundary component with a different slope. Only the degree $\delta$ matters and we get cylinders $\TT F_\delta$.
\end{itemize} 

\medskip

\textbf{Correspondence theorem.} We provide a theorem that gives a correspondence between the count of complex curves in $\CC F$ which are solution to the previous enumerative problem, and the count of tropical curves inside $\TT F$ solution to the analog tropical problem, counted with their usual multiplicity as a product over the vertices. If $h:\Gamma\to\TT F$ is a trivalent parametrized tropical curve (see Section \ref{section tropical line bundle elliptic curve} for definitions), its multiplicity is
$$m^\CC_\Gamma=\prod_V m_V,$$
where the product is over the trivalent vertices $m_V=|\det(a_V,b_V)|$, with $a_V$ and $b_V$ are two out of the three outgoing slopes of $h$. Contrarily to the toric case, there is no unique affine chart for $\TT F$, so slopes are not really defined, but fortunately, the determinant is invariant under a change of chart, so that the multiplicity is well-defined.

\begin{theom}[\ref{theorem correspondence}]
The count of genus $g$ parametrized tropical curve of bidegree $(d_1,d_2)$ in $\TT F$ passing through $\delta d_1+2d_2+g-1$ points in general position with multiplicity $m^\CC_\Gamma$ does not depend on the choice of the points and satisfies
$$N_{g,(d_1,d_2)}^{\delta,\mathrm{cpx},\bullet}=N_{g,(d_1,d_2)}^{\delta,\mathrm{trop},\bullet} \text{ and }N_{g,(d_1,d_2)}^{\delta,\mathrm{cpx}}=N_{g,(d_1,d_2)}^{\delta,\mathrm{trop}}.$$
The count of genus $g$ parametrized tropical curve of bidegree $(d_1,d_2)$ in $\TT F$ and tangency profile $(\mu_0+\nu_0,\mu_\infty+\nu_\infty)$, prescribed tangencies at $|\mu_0|$ points on $E_0$ and $|\mu_\infty|$ points on $E_\infty$, passing through $|\nu_0|+|\nu_\infty|+g-1$ points in general position with multiplicity $\frac{1}{I^{\mu_0+\mu_\infty}} m^\CC_\Gamma$ (where $I^\mu=\prod_i i^{\mu_i}$) does not depend on the choice of the constraints and satisfies
$$N_{g,(d_1,d_2)}^{\delta,\mathrm{cpx},\bullet}(\mu_0,\mu_\infty,\nu_0,\nu_\infty)=N_{g,(d_1,d_2)}^{\delta,\mathrm{trop},\bullet}(\mu_0,\mu_\infty,\nu_0,\nu_\infty),$$
$$N_{g,(d_1,d_2)}^{\delta,\mathrm{cpx}}(\mu_0,\mu_\infty,\nu_0,\nu_\infty)=N_{g,(d_1,d_2)}^{\delta,\mathrm{trop}}(\mu_0,\mu_\infty,\nu_0,\nu_\infty).$$
\end{theom}

\medskip

\textbf{Tropical refined enumeration.} Then, we show that replacing the complex multiplicity with the refined multiplicity from Block-G\"ottsche yields a tropical invariant. Following the definition of the complex multiplicity preceding the correspondence theorem, the refined multiplicity is defined as
$$m_\Gamma^q=\prod_V [m_V]_q=\prod_V \frac{q^{m_V/2}-q^{-m_V/2}}{q^{1/2}-q^{-1/2}}\in\ZZ[q^{\pm 1/2}].$$

\begin{theom}[\ref{theorem refined invariance}]
The count of genus $g$ parametrized tropical curve of bidegree $(d_1,d_2)$ in $\TT F$ passing through $\delta d_1+2d_2+g-1$ points in general position with multiplicity $m_\Gamma^q$ does not depend on the choice of the points as long as it is generic.
\end{theom}

\begin{rem}
The analog statement for the relative refined counts, \textit{i.e.} the refined count of curves with the tangency conditions to $E_0$ and $E_\infty$ is also true. The invariance with respect to the choice of the tropical elliptic curve $\TT E$ follows from a scaling argument, and the independance on the choice of the cylinders if $\delta=0$ follows from the computation using floor diagrams.
\end{rem}

\medskip

\textbf{Floor diagram algorithm.} Following \cite{brugalle2007enumeration}, we introduce suitable floor diagrams and their multiplicities to provide an algorithm to compute the invariants involved in the above theorems, both classical and refined.  As in the toric surface case, the floor diagram algorithm can also be seen as the iterated version of some Caporaso-Harris type formula, which is stated in Theorem \ref{theorem caporaso harris formula}.

\medskip

\textbf{Generating series and behavior of the invariants.} The floor diagram algorithm enables a concrete computation of the invariants. Although this could be the sought purpose, it is often more interesting to study the behavior of  and the regularity of the families of invariants, varying some of their arguments, such as the genus or the degree. For instance, consider double Hurwitz numbers, which are cover of $\CC P^1$ satisying specific ramification data. There is a way to use tropical geometry to compute them. It was observed in \cite{goulden2005towards}, and later studied in \cite{shadrin2008chamber}, that double Hurwitz numbers with a prescribed number of ramification points are piecewise polynomial functions. This was recovered using tropical geometry in \cite{cavalieri2010tropical}. Such behavior was also proved for double Gromov-Witten invariants of Hirzebruch surfaces in \cite{ardila2017double}. We here prove a similar statement: we use the floor diagrams to prove a polynomiality statement for the relative Gromov-Witten invariants of $\CC F$. We consider partitions $\nu_0$ and $\nu_\infty$ of fixed length $l_0$ and $l_\infty$, but of arbitrary size, provided that they satisfy $\|\nu_0\|-\|\nu_\infty\|=\delta d_1$ for some fixed $d_1$. Let $d_2=\|\nu_\infty\|$. We look for curves of genus $g$ of fixed bidegree $(d_1,d_2)$ and tangency profile $(\nu_0,\nu_\infty)$ passing through $l_0+l_\infty+g-1$ points in general position. We consider a partition $\nu$ of fixed length $l$ as a tuple $(\nu_1,\dots,\nu_l)$.

\begin{theom}[\ref{theorem polynomiality}]
For any fixed $g$ and $d_1$, the function $(\nu_0,\nu_\infty)\mapsto N^\mathrm{trop}_{g,(d_1,d_2)}(0,0,\nu_0,\nu_\infty)$ is a piecewise polynomial function in the variables $\nu_{0,1},\dots,\nu_{0,l_0},\nu_{\infty, 1},\dots,\nu_{\infty, l_\infty-1}$.
\end{theom}

Recently, E. Brugall\'e and A. J. Puentes \cite{brugalle2020polynomiality} proved a polynomiality statement for the coefficients of a fixed codegree in the families of refined invariance of toric surfaces. Such a behavior could also be studied for the hereby introduced refined invariants using the floor diagrams.

In the case of the $\CC F_{0,\alpha}$, instead of varying the intersection number with the zero-section $d_2$, it is also possible to vary the intersection number with a fiber $d_1$. The floor diagrams introduced in this paper allow one to recover the quasi-modularity statement from \cite{bohm2020counts}. Similarly to the polynomiality, the quasi-modularity is also a desirable property because it allows for a control of the asymptotic of the coefficients. Quasi-modular functions are a generalization of the modular forms, which are functions having nice transformations properties with respect to the action of the modular group $PSL_2(\ZZ)$. Following \cite{kaneko1995generalized}, quasi-modular forms express as a polynomial in the Eisenstein series $G_2$, $G_4$ and $G_6$, where $G_{2k}(y)=\sum_1^\infty \left(\sum_{d|n}d^k\right)y^n$ up to an affine transformation. They are stable by the operator $D=y\frac{\mathrm{d}}{\mathrm{d}y}$. The following theorem also generalizes the statement from \cite{bohm2020counts} because it also encompasses the case of relative invariants.

\begin{theom}\ref{theorem quasimodularity}
For fixed $g$, $d_2$, $\mu_0$, $\mu_\infty$, $\nu_0$ and $\nu_\infty$, the generating series
$$F_{g,d_2}(\mu_0,\mu_\infty,\nu_0,\nu_\infty)(y)=\sum_{d_1} N_{g,(d_1,d_2)}^\mathrm{trop}(\mu_0,\mu_\infty,\nu_0,\nu_\infty)y^{d_1},$$
are quasi-modular forms.
\end{theom}

\medskip

\textbf{Fock space approach.} Finally, we adapt the considerations of \cite{block2016fock} to our setting using our floor diagrams to relate the enumerative invariants to some matrix coefficients of some operators in a Heisenberg algebra acting on a Fock space, which is the content of Theorem \ref{theorem coefficient fock operator}. This allows one to express the generating series of the invariants considered in the paper.

\subsection{Further directions}

We now give several directions in which the result of the papers could be extended.

%\textbf{Degree $0$ line bundles} The considerations of the correspondence theorem do not work as easily for degree $0$ line bundles over the elliptic curve. This is due to the appearance of superabundant curves: curves that do not intersect the $0$ section nor the $\infty$-section. In the case of the trivial line bundle $\E\times\CC P^1$, the realization problem is easily handled. However, one can find families of degree $0$ line bundles that do not admit any sections nor multisections, but such that their tropicalization has some. This complicates the correspondence since for instance, some reducible tropical curves might be realizable but not by reducible curves.

\textbf{Descendant invariants.} One further direction would be to add $\psi$-constraints, as in \cite{cavalieri2021counting}. It should be possible to adapt the setting, at least for degree $0$ line bundles, to compute descendant Gromov-Witten invariants of line bundles over an elliptic curve, and express their generating series in term of other operators on the Fock space. This corresponds to counting tropical curves with vertices which are not trivalent anymore. However, this computation remains formal since it only reduces the computation of descendant invariants to descendant invariants with a unique point condition, which remain to be computed. Moreover, the appearance of superabundant curves needs to be handled carefully.

\textbf{Real refined enumeration.} In some situations, namely in \cite{mikhalkin2017quantum} and \cite{blomme2021refinedreal}, the tropical refined invariants are connected to some refined real enumerative invariants. More precisely, we do a signed count of real rational curves with prescribed tangency conditions on the toric boundary of the considered variety, refined by the value of a suitable \textit{quantum index}. Unfortunately, except the fibers, the surfaces $\CC F$ do not contain any rational curves. However, some recent results of I. Itenberg and E. Shustin proved that a real refined invariance result for genus $1$ curves in some toric surfaces. One could look into a possible generalization of the result to the case of $\CC F$.

\subsection{Plan of the paper}

The paper is organized as follows. 
\begin{itemize}[label=-]
\item In the second section, we introduce the surfaces $\CC F$ and present enumerative problems in it, leading to the definition of the enumerative invariants.
\item In the third section, we translate the first section in the tropical setting, introducing the tropical spaces $\TT F$. We also define tropical curves inside $\TT F$.
\item In the fourth section, we prove the main abstract results of the paper by giving the correspondence theorem to relate the tropical setting to the complex setting, and the tropical refined invariance statement.
\item In the fifth section, we get to the computational point of view by introducing floor diagrams that provide an algorithm to compute the invariants introduced in the preceding sections, and give some consequences such as a Caporaso-Harris type formula.
\item The sixth section is devoted to some examples and consequences of the use of floor diagrams.
\item Last, we provide an alternate point of view on floor diagrams using operators on some Fock space .
\end{itemize}

\textit{Acknowledgements} The author thanks Hannah Markwig for suggesting the problem and numerous discussions during the week at Oberwolfach concerning floor diagrams, and Johannes Rau for making with her some nice images of tropical curves inside a cylinder. The author also thanks the other organizers of this conference R. Cavalieri and D. Ranganathan. The author thanks the anonymous referee for helpful remarks, Erwan Brugall\'e and Ilya Tyomkin for discussions about the dimension of the moduli space of curves.

\medskip

\textit{Notations} If $\mu=(\mu_1,\mu_2,\dots)$ is a partition of an integer $d$, meaning $\sum i\mu_i=d$, we write $\mu\vdash d$. We set
$$|\mu|=\sum \mu_i,\ \|\mu\|=\sum_i i\mu_i,\ I^\mu=\prod_i i^{\mu_i}.$$
If $\nu$ is a second partition, we say that $\nu\leqslant\mu$ if $\nu_i\leqslant\mu_i$ for each $i$, and we set
$$\bino{\mu}{\nu}=\prod_i \bino{\mu_i}{\nu_i}.$$

\section{Complex setting}
\label{section complex setting}

\subsection{Ruled surfaces over an elliptic curve}

We consider ruled surfaces (\textit{i.e.} $\CC P^1$-bundles) over an elliptic curve $\E$. They are constructed as the projectivization $\PP(\V)$ of a rank $2$ vector bundle $\V\to\E$. These ruled surfaces are classified in \cite{hartshorne2013algebraic}. We have the following alternative:
\begin{itemize}[label=-]
\item either $\V$ does split as the sum of two line bundles, in which case the ruled surface is of the form $\PP(\L\oplus\O)$ for some line bundle $\L$ on the elliptic curve $\E$,
\item or $\V$ does not split, and it leads to two distinct ruled surfaces up to isomorphism \cite{hartshorne2013algebraic}.
\end{itemize}

We refer to \cite{griffiths2014principles} for an introduction to line bundles and divisors on curves. Moreover, as elliptic curves are one-dimensional abelian varieties, we also refer to the section of  \cite{griffiths2014principles} for the construction of line bundles over abelian varieties as quotient of the trivial bundle on $(\CC^*)^n$ (here just $\CC^*$), writing the abelian variety as a quotient of $(\CC^*)^n$. Here, $\E$ can be written of the form $\CC^*/\langle\lambda\rangle$. Assume that $\E=\CC^*/\langle\lambda\rangle$. Let $\alpha\in\CC^*$ and $\delta\in\ZZ$. Then the quotient of $\CC^*\times\CC$ (resp. $\CC^*\times\CC^*$ or $\CC^*\times\CC P^1$) by the action generated by
$$(z,w)\longmapsto (\lambda z,\alpha z^\delta w),$$
is a line bundle (resp. $\CC^*$-bundle or $\CC P^1$-bundle) over $\E=\CC^*/\langle\lambda\rangle$. It is of degree $\delta$. Moreover, up to isomorphism, every line bundle is of this form.

\medskip

We specifically care about the ruled surfaces of the form $\PP(\L\oplus\O)$. The line bundle $\L$ can be assumed to be of degree $\delta\geqslant 0$. The ruled surfaces are then classified by the degree $\delta$ if $\delta>0$, or by $\L\in\mathrm{Pic}(\E)$ up to inversion if $\delta=0$. This is due to the fact that $\E$ acts by translation on itself, and the induced action is transitive on the set of degree $\delta$ line bundles if $\delta\neq 0$, and trivial on the set of degree $0$ line bundles. In the following, we assume that $\L$ is given of non-negative degree. If $\delta>0$, we denote the ruled surface by $\CC F_\delta$, and if $\delta=0$, we denote the surface associated to an element $\alpha\in\mathrm{Pic}(\E)\simeq\E$ by $\CC F_{0,\alpha}$. When we do not care about which surface, we just write $\CC F$.

\begin{rem}
By abuse of notation, if $\delta=0$, $\alpha\in\CC^*$ also denotes the line bundle in $\mathrm{Pic}(\E)\simeq\E$. Indeed, up to a change of basis of the form $(z,w z^k)$, we see that $\alpha$ only matters up to multiplication by $\lambda$, meaning $\alpha\in\CC^*/\langle\lambda\rangle=\E$.
\end{rem}

The second homology group of $\CC F$ is generated by the class $F$ of a fiber, and $E_0$ of the $0$-section $z\in\E\mapsto [0: 1]$. Notice that we also have an $\infty$-section $E_\infty$ given by $z\in\E\mapsto [l: 0]$. Their classes satisfy
$$E_\infty=E_0-\delta F.$$

\medskip

The Chern class $c_1(\CC F)\in H^2(\CC F,\ZZ)$ satisfies the following:
\begin{itemize}[label=-]
\item $c_1(\CC F)\cdot F=2$ since the tangent bundle of $\PP(\L\oplus \O)$ restricted to a fiber is the sum of the trivial bundle and tangent bundle of $\CC P^1$, which is isomorphic to $\O(2)$.
\item $c_1(\CC F)\cdot E_0=\delta$ since the tangent bundle of $\CC F$ restricted to $E_0$ is the sum of the tangent bundle of $\E$, which is trivial, and the line bundle $\L$, which has Chern class $\delta$.
\item Using the relation between $E_\infty$, $E_0$ and $F$, we get that $c_1(\CC F)\cdot E_\infty=-\delta$, which is unsurprisingly the Chern class of $\L^{-1}$, for which the $0$-section and $\infty$-section are switched.
\end{itemize}

\subsection{Curves in the total space and enumerative problems}

\subsubsection{Dimension of the moduli space of curves.} We now consider curves inside the ruled surfaces $\CC F$.

\begin{defi}
We say that a curve $\varphi:\CC C\to\CC F$, where $\CC C$ is a genus $g$ Riemann surface, is of \textit{bidegree} $(d_1,d_2)$ if it realizes the class $d_1 E_0+d_2 F$.
\end{defi}

\begin{rem}
Recall that the genus of a curve $\CC C$ is defined by the relation $\chi(\CC C)=2-2g$. If the curve is irreducible, its first Betti number is $2g$, but this is no longer the case if it is reducible.
\end{rem}

\begin{expl}
The sections of the line bundle $\L$ realize the class $E_0$ and are therefore curves of bidegree $(1,0)$. Meromorphic sections with $a$ poles realize the class $E_0+ aF$.
\end{expl}

For a class to be realizable, except $E_\infty$ itself, it has to intersect positively with $E_\infty$ and $F$, implying that $d_1,d_2\geqslant 0$. Conversely, as both $E_0$ and $F$ are realized by complex curves, the cone of realizable class is defined by $d_1,d_2\geqslant 0$.

\medskip

Let $\M(\CC F,(d_1,d_2))$ be the moduli space of curves of bidegree $(d_1,d_2)$ inside $\CC F$. If the Picard group of $\CC F$ was of rank $0$, it would just be $\PP(H^0(\CC F,\L_{(d_1,d_2)}))$, where $\L_{(d_1,d_2)}$ is a line bundle on $\CC F$ having Chern class $(d_1,d_2)$. as the Picard group is not of rank $0$ but of rank $1$, $\M(\CC F,(d_1,d_2))$ is a projective bundle over $\mathrm{Pic}(\E)$. Its dimension is given by the adjunction formula and is
$$\dim\M(\CC F,(d_1,d_2))=\frac{(\delta d_1+2d_2)(d_1+1)}{2}.$$
We now give the Proposition 15 from \cite{kleiman2013gottsche}, adapted to our setting: it is valid for any projective surface up to replacing the linear system by curves in a fixed class. It allows us to compute the dimension of the space of curves inside $\CC F$.

\begin{prop}\cite{kleiman2013gottsche}
Let $W$ be an irreducible subvariety of $\M(\CC F,(d_1,d_2))$, $\CCC_W\to W$ the tautological family of curves, $\widetilde{\CCC_W}\to\CCC_W$ the normalization, $f_W :\widetilde{\CCC_W}\to\CC F$ the natural morphism, and $0\in W$ a general closed point. Assume that the curve $\CCC_0$ is reduced.
\begin{enumerate}[label=(\roman*)]
\item There exists a natural embedding $T_0 W\hookrightarrow H^0(\widetilde{\CCC_0},\N_{f_0}/\N^\mathrm{tor}_{f_0})$, where $\N$ denotes the normal sheaf and $T_0 W$ the tangent space to $W$.
\item If $c_1(\CC F)\cdot\CCC\geqslant 1$ for any irreducible component $\CCC\subset\CCC_0$, then
$$\dim W\leqslant h^0(\widetilde{\CCC_0},\N_{f_0}/\N^\mathrm{tor}_{f_0})\leqslant c_1(\CC F)\cdot(d_1,d_2)+p_g(\CCC_0)-1.$$
\item If we have equality in (ii) and $c_1(\CC F)\cdot\CCC\geqslant 2$ for an irreducible component $\CCC$ of $\CCC_0$, then $\CCC$ is immersed.
\item If (ii) is an equality and $c_1(\CC F)\cdot\CCC\geqslant 2$ for any irreducible component $\CCC$ of $\CCC_0$, then $\N_{f_0}$ is invertible and the map from (i) is an isomorphism.
\end{enumerate}
\end{prop}

The proposition asserts that provided $\delta d_1+2d_2\geqslant 1$, which is always satisfied unless $d_1=d_2=0$, any irreducible component of the space of reduced genus $g$ curves of bidegree $(d_1,d_2)$ inside $\M(\CC F,(d_1,d_2))$ is of dimension at most $\delta d_1+2d_2+g-1$. Furthermore, if we consider only irreducible curves and we assume $\delta d_1+2d_2\geqslant 2$, which is always satisfied unless $\delta=0$ and $d_2=0$, or $\delta=1$ and $(d_1,d_2)=(1,0)$, curves are immersed. If $\delta=1$, curves of bidegree $(1,0)$ are just sections of the bundle, so the dimension count is given by Riemann-Roch and the fact that the curves are immersed is obvious. If $\delta=0$, Beside $E_0$ and $E_\infty$, it follows from Proposition \ref{theorem complex menelaus} that curves of bidegree $(d_1,0)$ can only occur if the line bundle $\L$ is torsion. If it is chosen generically, there are none and the assumption is always satisfied.

Furthermore, we have the strata of nodal curves of genus $g$ and bidegree $(d_1,d_2)$. By the adjunction formula, they have $\frac{(\delta d_1+2d_2)(d_1-1)}{2}+1-g$ nodes. Thus, each node imposing a codimension $1$ constraint, the dimension of the strata is at least
$$\dim\M(\CC F,(d_1,d_2))-\left(\frac{(\delta d_1+2d_2)(d_1-1)}{2}+1-g\right)=\delta d_1+2d_2+g-1.$$
In particular, there are components in the space of reduced genus $g$ curves of the expected dimension. There might be components of smaller dimension (although it is possible to show that there are in fact none) but they do not matter for the enumerative problem that we consider.

\begin{rem}
There might also be curves which are not nodal and can be deformed in a space of the expected dimension, but the computation that we do near the tropical limit proves that there are none either.
\end{rem}

\medskip

\subsubsection{Classical invariants.} We now have the following enumerative problem.

\begin{prob}
How many (irreducible) genus $g$ curves of bidegree $(d_1,d_2)$ passing through $\delta d_1 +2d_2 +g-1$ points in generic position are there ?
\end{prob}

If the constraints are chosen generically, as the number of constraints is chosen equal to the dimension, we expect a finite number of curves.

\begin{prop}
The number of genus $g$ curves of bidegree $(d_1,d_2)$ passing through a generic point configuration $\P$ of $\delta d_1 +2d_2 +g-1$ points is finite. It does not depend on the choice of $\P$ as long as it is generic.
\end{prop}

\begin{proof}
Taking into account all choices of $\P$, we get a family of subspaces $W_\P$ of $\M(\CC F,(d_1,d_2))$ indexed by $\P\in\CC F^{\delta d_1+2d_2+g-1}$: $W_\P$ is the subspace of curves passing through $\P$. Globally, the $W_\P$ cover $\M(\CC F,(d_1,d_2))$, and we thus have transversality with the subspace of reduced genus $g$ curves. Sard's lemma ensures that we also have transversality with a generic member of $(W_\P)$. As the dimensions have been chosen to be complementary, we get a finite number of curves passing through a generic $\P$.

The number of intersection points (\textit{i.e.} curves passing through $\P$ does not depend on the choice of $\P$ since it corresponds to the intersection number inside $\M(\CC F,(d_1,d_2))$ between the space of genus $g$ reduced curves and $W_\P$.
\end{proof}

Moreover, if the point configuration is chosen generically, no non-reduced curve can pass through them since its image varies in a space of strictly smaller dimension, and thus cannot meet the constraints by genericity. So the reduced count contains in fact all the curves. The obtained invariant is momentarily denoted by $N_{g,(d_1,d_2)}^{\CC F,\mathrm{cpx}}$.

\begin{prop}
If $\delta\neq 0$, the number $N_{g,(d_1,d_2)}^{\CC F_\delta,\mathrm{cpx}}$ does not depend on the choice of the elliptic curve $\E$. Furthermore, if $\delta=0$, the number does not depend on the choice of $\alpha\in\mathrm{Pic}(\E)$ as long as it is chosen generically, nor the choice of $\E$. Thus, the numbers only depend on $\delta$, and are denoted $N_{g,(d_1,d_2)}^{\delta,\mathrm{cpx}}$.
\end{prop}

\begin{proof}
The argument from Proposition 15 in \cite{kleiman2013gottsche} also apply for any characteristic $0$ field. In particular, they apply for families by considering the field of (converging) Puiseux series $\CC\{\{t\}\}$, and then specializing to the fibers. It suffices to take families of elliptic curves of the form $\CC^*/\langle \lambda(t) t^l\rangle$, which can either be seen as families of complex elliptic curves, or as curves over the field of Puiseux series $\CC\{\{t\}\}$.

The surfaces $\CC F_{0,\alpha}$ require a little more care since there might be curves of bidegree $(d_1,0)$, but this class does not satisfy the assumptions of Proposition 15 in \cite{kleiman2013gottsche}. Curves different from $E_0$ and $E_\infty$ occur in such a class only if $\alpha$ is a $d_1$-torsion element in $\E\simeq\mathrm{Pic}(\E)$. This follows from Proposition \ref{theorem complex menelaus} below as we have $D(\CC C)=0$ for such a curve. Moreover, such curves may appear as irreducible components of curves of other bidegrees. If $(d_1,d_2)$ is fixed, the set of $\alpha\in\CC^*$ for which the classes $(k,0)$, $k\leqslant d$ is discrete: it corresponds to the preimage in $\CC^*$ of torsion elements up to $d_1$ in $\mathrm{Pic}(\E)$. Thus, we can assume that $\alpha$ is not one of them. Moreover, this is still true for $\alpha'$ in a small neighbourhood of $\alpha$.

Finally, we can put the surfaces $\CC F_{0,\alpha}$ in a family by considering the construction of $\CC F_{0,\alpha}$ with $\alpha$ as a parameter: we quotient $\CC^*\times\CC P^1\times\CC^*$ by the following action,
$$(z,w,\alpha)\longmapsto (\lambda z,\alpha w,\alpha).$$
The quotient has a well-defined map to $\CC^*$ by mapping $(z,w,\alpha)$ to $\alpha$, and it is stable by the action. The preimage of $\alpha\in\CC^*$ is $\CC F_{0,\alpha}$. Therefore, the intersection number $N_{g,(d_1,d_2}^{\CC F_{0,\alpha},\mathrm{cpx}}$ remains the same for $\alpha'$ in a small neighbourhood of $\alpha$, and thus for generic $\alpha$.
\end{proof}

\begin{rem}
The use of families of the form $\CC^*/\langle\lambda t^l\rangle$ is interesting since it allows us to pass to the tropical limit, where a correspondence statement allows us to compute the invariants.
\end{rem}

\begin{rem}
These invariants also appear in \cite{grafnitz2022proper} as log Gromov-Witten invariants through degeneration of surfaces.
\end{rem}

In total, we get a family of invariants $N_{g,(d_1,d_2)}^{\delta,\mathrm{cpx}}$ for $\delta\geqslant 0$ by counting irreducible curves. Counting reducible curves as well, we denote the invariant by $N_{g,(d_1,d_2)}^{\delta,\mathrm{cpx},\bullet}$.

\begin{rem}
The assumption that $\alpha\in\CC^*$ is generic is necessary to get invariance for the following reason. If the parameter $\alpha$ becomes torsion, some irreducible curves might become reducible with some irreducible component being of bidegree $(k,0)$. These components fail to deform for generic $\alpha$, and thus the invariance also fails. The invariance could be recovered by carefully counting the curves having irreducible components of bidegree $(k,0)$.
\end{rem}

\subsubsection{Relative invariants.} Let $\varphi:\CC C\to\CC F$ be a curve of bidegree $(d_1,d_2)$. Intersecting $\varphi(\CC C)$ with the two distinguished divisors $E_0$ and $E_\infty$, we get respectively $\delta d_1+d_2$ and $d_2$ intersection points counted with multiplicity. The number of intersection points of each tangency index give two partitions $\mu_0\vdash \delta d_1+d_2$ and $\mu_\infty\vdash d_2$. We then say that $\CC C$ has tangency profile $(\mu_0,\mu_\infty)$. Furthermore, as $E_0$ and $E_\infty$ are both identified with the elliptic curve $\E$, we get two divisors
$$\sum_{i}\sum_{1\leqslant j\leqslant\mu_{0i}}ip_{0ij} \text{ and }\sum_{i}\sum_{1\leqslant j\leqslant\mu_{\infty i}}ip_{\infty ij},$$
for some distinct points $p_{0ij},p_{\infty ij}\in\E$. In particular, for each curve $\varphi:\CC C\to\CC F$, we get a divisor $D(\CC C)=\sum_{ij} i(p_{0ij}-p_{\infty ij})$ on $\E$. In fact, we have the following Proposition.

\begin{prop}\label{theorem complex menelaus}
If $D_0$ is some divisor on $\E$ such that $\L\simeq\L_{D_0}$, and $\varphi:\CC C\to\CC F$ is some curve of bidegree $(d_1,d_2)$, with $d_1\neq 0$, then $D(\CC C)$ is equivalent to $d_1 D_0$ on $\E$.
\end{prop}

\begin{proof}
%This merely comes from the relation between $\mathrm{Pic}\E$ and $\mathrm{Pic}(\CC F)$:
%$$\mathrm{Pic}(\CC F)=\ZZ\oplus\pi^*\mathrm{Pic}(\E),$$
%where $\ZZ$ is generated by the class of the section $E_0$, and $\pi^*\mathrm{Pic}\E$ by the fibers $\pi^{-1}(p)$ for any $p\in\E$. Following \cite{hartshorne2013algebraic}, through this isomorphism, a divisor $D$ on $\CC F$ is mapped to $(D\cdot F,\L_{D'})$, where $D'=D-(D\cdot F)E_0$ is a divisor that satisfies
%$$\L_{D'}\simeq\pi^*\pi_*\L_{D'},$$
%and $\pi_* D'$ is a line bundle on $\E$. The curve $\varphi:\CC C\to\CC F$ defines a divisor $D$ in $\CC F$, mapped to $(d_1,\L')$ by the isomorphism. the result comes from the expression of $\L'$. 

If $d_1=1$, it is exactly the fact that the divisor of a meromorphic section of $\L$ satisfies $\L\simeq\L_D$. Another way to view it is as follows: let $s$ and $s'$ be two meromorphic sections of the line bundle. Any linear combination $\alpha s+\beta s'$ of $s$ and $s'$ is still a meromorphic section, and its divisor depends only on $(\alpha,\beta)$ up to multiplication by a non-zero constant. Thus, taking their divisor, we get a map from $\CC P^1$ to the Jacobian of the curve, that associates to $[\alpha:\beta]$ the divisor of the section $\alpha s+\beta s'$. As maps from the projective line to a complex torus are constant, we get that $s$ and $s'$ have equivalent divisors.

If we now take any curve inside the total space $\CC F$, it can be considered as a $d_1$-multisection, meaning there are $d_1$ points in the fiber over each point in the base. If $s$ is a meromorphic section of the line bundle, we can still consider linear combinations between the multisection and $s$, except these are now $d_1$-multisection. Varying the coefficients and taking the divisor, we similarly get a map from $\CC P^1$ to the Jacobian of the curve. As these maps are constant, we get that the divisor of the multisection is equivalent to the divisor of $s$ viewed as a $d_1$-multisection, which is $d_1 D$.
\end{proof}

\begin{rem}
In the case of the trivial bundle $\E\times\CC P^1$ sections are meromorphic functions. The statement then amounts to say that the divisor of a meromorphic function is principal.
\end{rem}

Imposing a tangency of order $i$ at a fixed point of $E_0$ (resp. $E_\infty$) is a codimension $i$ condition, while imposing a tangency of order $i$ somewhere on $E_0$ (resp. $E_\infty$) is codimension $i-1$. For the second enumerative problem, let $\mu_0+\nu_0\vdash \delta d_1+d_2$ and $\mu_\infty+\nu_\infty\vdash d_2$ be partitions. The number of solutions to the following enumerative problem counts curves with fixed tangency profile with the divisor $E_0+E_\infty$.

\begin{prob}
How many genus $g$ curves of bidegree $(d_1,d_2)$, have $\mu_{0i}$ (resp. $\mu_{\infty i}$) intersection points of tangency order $i$ at fixed points of $E_0$ (resp. $E_\infty$) and $\nu_{0i}$ (resp. $\nu_{\infty i}$) more intersection of tangency order $i$ with $E_0$ (resp. $E_\infty$) at non-fixed points, and pass through $|\nu_0|+|\nu_\infty|+g-1$ points in generic position in $\CC F$ ?
\end{prob}

As in the case of point constraints, the number of such curves does not depend either on the choice of the points as long as it is generic, the choice of the elliptic curve, and the choice of generic $\alpha\in\mathrm{Pic}(\E)$ if $\delta=0$. The number of these curves is denoted by $N^{\delta,\mathrm{cpx},\bullet}_{g,(d_1,d_2)}(\mu_0,\mu_\infty,\nu_0,\nu_\infty)$ for the reducible count, and $N^{\delta,\mathrm{cpx}}_{g,(d_1,d_2)}(\mu_0,\mu_\infty,\nu_0,\nu_\infty)$ for the irreducible count.

\section{Tropical line bundle over an elliptic curve}
\label{section tropical line bundle elliptic curve}

We now adapt the previous section to the tropical setting, defining the ruled surfaces over a tropical elliptic curve $\TT E$ which are obtained from tropical line bundles. The reader is assumed to be familiar with some notions of tropical geometry. We refer to \cite{brugalle2014bit} for a introduction to tropical geometry, and to sections $3$ and $4$ of \cite{mikhalkin2008tropical} for more details on divisors and line bundles over tropical curves.

\subsection{Total space of a tropical line bundle}

\subsubsection{Tropical elliptic curve.} We take the following definition of a tropical elliptic curve.

\begin{defi}
A tropical elliptic curve $\TT E$ is a quotient $\RR/l\ZZ$, where $l$ is some positive real number.
\end{defi}

As in the complex setting, tropical elliptic curves are also abelian groups. Topologically, a tropical elliptic curve is homeomorphic to $S^1$, just as any complex elliptic curve is homeomorphic to a torus $(S^1)^2$. Moreover, as complex elliptic curves are classified by a complex number $\tau$ (up to a $SL_2(\ZZ)$ action), tropical elliptic curves are classified by their length $l$.

\medskip

\subsubsection{The tropical cylinders.} We adopt a construction mimicking the construction of line bundles over a complex torus done in \cite{griffiths2014principles}. We consider the $\ZZ$-action on $\RR^2$ generated by the following affine diffeomorphism:
$$\varphi:(x,y)\longmapsto(x+l,y+\delta x-\alpha),$$
where $\delta\in\ZZ$ and $\alpha\in\RR$. The diffeomorphism is an affine map whose linear part lies in $GL_2(\ZZ)$. Thus, the quotient space inherits a lattice structure from the standard lattice structure of $\RR^2$. As the action onto the first coordinate is just a translation, projecting onto the first coordinate, we get a map
$$\RR^2/\langle\varphi\rangle\longrightarrow\RR/l\ZZ.$$
This maps makes the quotient space $\RR^2/\langle\varphi\rangle$ into a $\RR$-bundle over the tropical elliptic curve $\TT E=\RR/l\ZZ$.

\begin{rem}
Changing $\RR^2$ with $\RR\times\TT$ (where $\TT=\RR\cup\{-\infty\}$) or $\RR\times\TT P^1$ (where $\TT=\RR\cup\{-\infty,+\infty\}$) with the same action, we would get a line bundle or $\TT P^1$-bundle over the elliptic curve $\RR/l\ZZ$. This corresponds to the construction of \cite{griffiths2014principles} leading to a $\CC^*$, $\CC$ or $\CC P^1$-bundle according to the fiber that we choose.
\end{rem}

Up to reversing the vertical direction inside $\RR^2$, we can assume that $\delta\geqslant 0$. Then,
\begin{itemize}[label=-]
\item If $\delta>0$, using a horizontal translation inside $\TT E$ and $\RR^2$, one can assume that $\alpha=0$. The total space is denoted by $\TT F_\delta$, and we write $\varphi_\delta:(x,y)\mapsto(x+l,y+\delta x)$ when necessary.
\item If $\delta=0$, using a basis of the form $(e_1+ke_2,e_1)$ for some $k\in\ZZ$ instead of the canonical basis of $\RR^2$, the parameter $\alpha$ of the translation $\varphi$ becomes $\alpha-kl$. Thus, one can assume that $0\leqslant\alpha<l$. The total space is denoted by $\TT F_{0,\alpha}$, and we write $\varphi_{0,\alpha}:(x,y)\mapsto(x+l,y-\alpha)$ when necessary. The parameter $\alpha$ can thus be seen as an element of $\TT E$.
\end{itemize}

Topologically, the spaces that we obtain are cylinders. They differ by the lattice structure that they inherit from the standard lattice structure of $\RR^2$.

\begin{rem}
One other way to obtain the cylinders would be the following. Consider a cover of $\TT E=\RR/l\ZZ$ by open set $(U_i)$, and a cocycle of affine functions on $U_i\cap U_j$, thus defining a tropical line bundle. Then glue the trivial bundles $U_i\times\RR$ (resp. $U_i\times\TT$, $U_i\times\TT P^1$) together over the $U_i\cap U_j$ using the cocycle. This constructs the total space of the $\RR$-bundle (resp. line bundle, $\TT P^1$-bundle).
\end{rem}

\begin{rem}
The cylinders $\TT F$ are not compact. They yet correspond to the tropicalization of the surfaces $\CC F_\delta$ just like the usual tropicalization of $\CC P^2$ is $\RR^2$ with the data of the unbounded directions corrsponding to its toric divisors. In our case, the two directions are the top and bottom directions of the cylinders, corresponding to the divisors $E_0$ and $E_\infty$ of $\CC F_\delta$.
\end{rem}

Concretely, taking a fundamental domain, the cylinders $\TT F$ can be seen as follows.
\begin{itemize}[label=-]
\item If $\delta>0$, the cylinder $\TT F_\delta$ is the quotient of the strip $[0;l]\times\RR$ by the identification of both sides: $(0,y)\sim(l,y)$, but the lattice structure has monodromy $\left(\begin{smallmatrix} 1 & 0 \\ -\delta & 1 \\ \end{smallmatrix}\right)$: a straight line having slope $(1,p)$ crossing the right side of the strip comes back from the left side with slope $(1,p-\delta)$.
\item If $\delta=0$, the cylinder $\TT F_{0,\alpha}$ is obtained identifying both sides of the strip $[0;l]\times\RR$ by a translation: $(0,y)\sim(l,y-\alpha)$. The slope of a line crossing the boundary does not change.
\end{itemize}
We refer to the next section for pictures of the cylinders with curves inside it, expliciting the lattice structure.

\subsection{Tropical curves in $\TT F$}

\subsubsection{Abstract tropical curves.} An \textit{irreducible abstract tropical curve} is a connected metric graph with unbounded edges called \textit{ends} that have infinite length. The number of neighbours of a vertex of $\Gamma$ is called its \textit{valence}. Edges adjacent to vertices of valence $1$ are required to have infinite length, and we remove the $1$-valent vertices from $\Gamma$. The other edges are required to have finite length and are called \textit{bounded edges}. The \textit{genus} of an abstract tropical curve is defined by $g=1-\chi(\Gamma)$. If $\Gamma$ is connected, it is equal to its first Betti number.

\begin{rem}
Due to the enumerative problem that we study, the tropical curves considered in the paper do not have genus at their vertices and we do not include that in the definition. Such a generalization would be needed if one tried to consider for instance descendant invariants, as in \cite{cavalieri2021counting}.
\end{rem}

\subsubsection{Parametrized tropical curves.} In all what follows, $\TT F$ is assumed to be either some $\TT F_\delta$, or some $\TT F_{0,\alpha}$.

\begin{defi}
A parametrized tropical curve inside $\TT F$ is a map $h:\Gamma\to\TT F$ where $\Gamma$ is an abstract tropical curve, and
\begin{itemize}[label=-]
\item $h$ is affine with integer slope on the edges of $\Gamma$,
\item at each vertex, one has the balancing condition: the sum of the outgoing slopes of $h$ is $0$.
\end{itemize}
\end{defi}

Here, integer slope means that the slope of $h$ belongs to the tangent lattice. Concretely, drawing the curve inside $[0;l]\times\RR$, it means that the derivative lies in $\ZZ^2$. The weight of an edge $e$ is the integer length of the slope of $h$ on $e$, which does not depend on the chosen chart. As the vertical direction $(0,1)$ is unchanged by the gluing maps, we can speak about the vertical direction.

\medskip

\subsubsection{Bidegree of tropical curves inside $\TT F$.}

\begin{defi}
We say that a parametrized curve $h:\Gamma\to\TT F$ is of \textit{bidegree} $(d_1,d_2)$ if it has $d_2$ unbounded edges (counted with weight) in the direction $(0,1)$, and $d_1$ intersection points (counted with multiplicity equal to the horizontal coordinate of their slope) with a vertical line.
\end{defi}

Let $h:\Gamma\to\TT F$ and $h':\Gamma'\to\TT F$ be two parametrized tropical curves of respective bidegrees $(d_1,d_2)$ and $(d'_1,d'_2)$. Assume they are transverse to each other. For any intersection point $p$ of $\Gamma$ and $\Gamma'$, the intersection index $|\det(\partial_p h,\partial_p h')|$ is well-defined since it does not depend on the choice of the integer affine chart of $\TT F$. Thus, we can speak of the intersection index between $\Gamma$ and $\Gamma'$ at $p$, and define their intersection number by adding the contributions of the various intersection points. Using the balancing condition, similarly to the tropical B\'ezout theorem in the plane, we get that the intersection number only depends on the bidegrees of the two curves, and is equal to $\delta d_1 d'_1+d_1 d'_2+d_2 d'_1$. Therefore, the bidegree of a curve is defined by its intersection number with the $\infty$-section and the class of the fiber. The balancing condition ensures that it does not depend on the chosen vertical fiber.

\begin{rem}
Alternatively, one could consider the classes realized by the parametrized tropical curve $h:\Gamma\to\TT F$ inside the tropical homology group $H_{1,1}(\TT F,\ZZ)$. To do this one should consider the compact space that is a $\TT P^1$-bundle over $\TT E$ rather than a $\RR$-bundle. The homology group is then generated by the class $F$ of a fiber, which is a vertical line, and some section, for instance the $0$-section, represented by the bottom boundary component of the cylinder. A curve is of bidegree $(d_1,d_2)$ if it realizes the class $d_1 E_0+d_2 F$.
\end{rem}

\begin{lem}
Counted with multiplicity, a curve of bidegree $(d_1,d_2)$ has $\delta d_1+d_2$ (counted with weights) unbounded ends in the direction $(0,-1)$.
\end{lem}

\begin{proof}
It is just the intersection number of the curve with the class of the $0$-section. Alternatively, use the cutting procedure from Section \ref{section cutting procedure} to get a tropical curve inside $\RR^2$. The balancing condition ensures that the sum of the outgoing slopes of the unbounded ends add up to $0$. there are by assumption $d_2$ ends in the direction $(0,11)$. The ends on the right have slope $(u_i,v_i)$, with $\sum u_i= d_1$, and each one of them is associated to an end on the left with slope $(-u_i,\delta u_i-v_i)$. Thus, all add up to $(0,\delta \sum u_i)=(0,\delta d_1)$. The result follows.
\end{proof}

We now define the tangency profile of a parametrized tropical curve of bidegree $(d_1,d_2)$.

\begin{defi}
Let $\|\mu_0\|= \delta d_1+d_2$ and $\|\mu_\infty\|= d_2$ be partitions. A parametrized tropical curve of bidegree $(d_1,d_2)$ is said to have ramification profile $(\mu_0,\mu_\infty)$ if it has $\mu_{0i}$ ends of slope $(0,-i)$, and $\mu_{\infty i}$ ends of slope $(0,i)$.
\end{defi}

\subsubsection{Examples of tropical curves in cylinders.} We now give several examples and illustrations of tropical curves living in the spaces $\TT F_{0,\alpha}$, $\TT F_1$ and $\TT F_2$.

\begin{expl}
\begin{itemize}[label=-]
\item The tropical graph of a section of the line bundle realizes the class $E_0$. It has exactly one intersection point with every fiber, $\delta$ unbounded ends going down, and none going up. By tropical graph, we mean the usual graph with vertical ends added at each corner point so that the balancing condition is satisfied.
\item The tropical graph of a meromorphic section of the line bundle realizes the class $E_0+d_2 F$, where $d_2$ is the number of poles of the section.
\end{itemize}
\end{expl}

\begin{expl}
On Figure \ref{figure example ExP1} one can see three examples of tropical curves inside the total space of a degree $0$ line bundle. The pictures are drawn inside the strip $[0;l]\times\RR$ both whose sides are identified.
\begin{itemize}[label=-]
\item The examples $(a)$ and $(b)$ take place in the trivial bundle $\TT E\times\RR$ and depict two genus $1$ curves of respective bidegrees $(2,2)$ and $(1,2)$.
\item On $(c)$ is drawn a genus $1$ and bidegree $(1,1)$ curve in some $\TT F_{0,\alpha}$ for some non-zero $\alpha$. The identification between the two sides is done with some vertical translation.
\end{itemize}
Both examples $(b)$ and $(c)$ can be seen as tropical graphs of meromorphic sections of the respective line bundles. Notice that here, the slope does not change when crossing the boundary and coming back from the other side.
\end{expl}

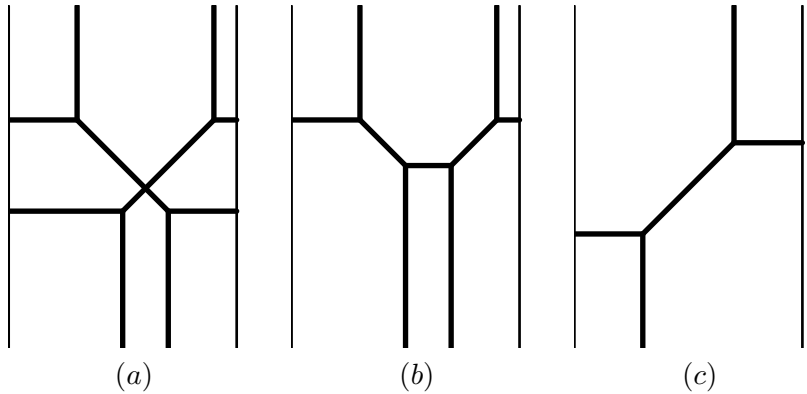
\begin{figure}
\begin{center}
\begin{tabular}{ccc}
\begin{tikzpicture}[line cap=round,line join=round,>=triangle 45,x=0.3cm,y=0.3cm]
\clip(0,0) rectangle (11,15);
\draw [line width=1pt] (0,0)-- (0,15);
\draw [line width=1pt] (10,0)-- (10,15);

\draw [line width=2pt] (0,10)-- (3,10);
\draw [line width=2pt] (3,10)-- (7,6);
\draw [line width=2pt] (7,6)-- (10,6);
\draw [line width=2pt] (0,6)-- (5,6);
\draw [line width=2pt] (5,6)-- (9,10);
\draw [line width=2pt] (9,10)-- (10,10);

\draw [line width=2pt] (9,10)-- (9,15);
\draw [line width=2pt] (3,10)-- (3,15);
\draw [line width=2pt] (7,6)-- (7,0);
\draw [line width=2pt] (5,6)-- (5,0);

\begin{scriptsize}

\end{scriptsize}
\end{tikzpicture}
&
\begin{tikzpicture}[line cap=round,line join=round,>=triangle 45,x=0.3cm,y=0.3cm]
\clip(0,0) rectangle (11,15);
\draw [line width=1pt] (0,0)-- (0,15);
\draw [line width=1pt] (10,0)-- (10,15);

\draw [line width=2pt] (0,10)-- (3,10);
\draw [line width=2pt] (3,10)-- (5,8);
\draw [line width=2pt] (5,8)-- (7,8);
\draw [line width=2pt] (7,8)-- (9,10);
\draw [line width=2pt] (9,10)-- (10,10);

\draw [line width=2pt] (9,10)-- (9,15);
\draw [line width=2pt] (3,10)-- (3,15);
\draw [line width=2pt] (7,8)-- (7,0);
\draw [line width=2pt] (5,8)-- (5,0);

\begin{scriptsize}

\end{scriptsize}
\end{tikzpicture}
&
\begin{tikzpicture}[line cap=round,line join=round,>=triangle 45,x=0.3cm,y=0.3cm]
\clip(0,0) rectangle (11,15);
\draw [line width=1pt] (0,0)-- (0,15);
\draw [line width=1pt] (10,0)-- (10,15);

\draw [line width=2pt] (0,5)-- (3,5);
\draw [line width=2pt] (3,5)-- (7,9);
\draw [line width=2pt] (7,9)-- (10,9);

\draw [line width=2pt] (7,9)-- (7,15);
\draw [line width=2pt] (3,5)-- (3,0);

\begin{scriptsize}

\end{scriptsize}
\end{tikzpicture} \\
$(a)$ & $(b)$ & $(c)$ \\
\end{tabular}
\caption{\label{figure example ExP1}Examples of tropical curves inside $\TT E\times\RR$ ($(a)$ and $(b)$), and inside $\TT F_{0,\alpha}$ for some non zero $\alpha$ ($(c)$).}
\end{center}
\end{figure}

\begin{expl}
On Figure \ref{figure section curve TF1} we draw two sections of a degree $1$ bundle. We obtain thus curves in $\TT F_1$. On $(a)$, the identification between the two sides of the strip is done with a non-zero vertical translation. The slope changes when the boundary is crossed. This comes from the fact that the total space is constructed by gluing the left and right small strips between the whole and dotted lines with an affine map which is not a mere translation: through the gluing, the dashed line on the right is sent to a small part on the left horizontal part of the curve.

On $(b)$, it is the same curve but drawn in a chart such that the identification between the two sides of the strip is done without translation (but still with a change of slope when crossing it). As a consequence, the unique unbounded end and the unique vertex live on the boundary of the chart. One can check that the balancing condition is indeed satisfied.
\end{expl}

\begin{figure}[h]
\begin{center}
\begin{tabular}{cc}
\begin{tikzpicture}[line cap=round,line join=round,>=triangle 45,x=0.3cm,y=0.3cm]
\clip(0,0) rectangle (11,15);
\draw [line width=1pt] (0,0)-- (0,15);
\draw [line width=1pt] (10,0)-- (10,15);
\draw [line width=1pt, dotted] (11,0)-- (11,15);
\draw [line width=1pt, dotted] (1,0)-- (1,15);

\draw [line width=2pt] (0,5)-- (5,5);
\draw [line width=2pt] (5,5)-- (10,10);
\draw [line width=2pt, dashed] (10,10)-- (11,11);
\draw [line width=2pt] (5,5)-- (5,0);

\begin{scriptsize}

\end{scriptsize}
\end{tikzpicture}
&
\begin{tikzpicture}[line cap=round,line join=round,>=triangle 45,x=0.3cm,y=0.3cm]
\clip(-1,0) rectangle (11,15);
\draw [line width=1pt] (0,0)-- (0,15);
\draw [line width=1pt] (10,0)-- (10,15);

\draw [line width=2pt] (0,5)-- (10,5);

\draw [line width=2pt, dashed] (0,5)-- (-1,6);
\draw [line width=2pt, dashed] (10,5)-- (11,6);

\draw [line width=2pt] (0,0)-- (0,5);
\draw [line width=2pt] (10,0)-- (10,5);

\begin{scriptsize}

\end{scriptsize}
\end{tikzpicture}\\
$(a)$ & $(b)$ \\
\end{tabular}
\caption{\label{figure section curve TF1} Examples of sections inside $\TT F_1$. On $(a)$, the chart has not been chosen so that the identification between the two sides of the strip is identity. On $(b)$ it is the case.}
\end{center}
\end{figure}
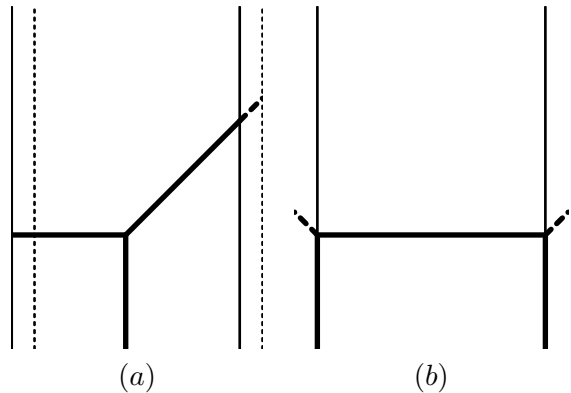

\begin{expl}
On Figure \ref{figure section curve TF1 2}, we draw a genus $1$ bidegree $(1,2)$ and a genus $2$ bidegree $(2,2)$ curve in $\TT F_1$, and a genus $1$ bidegree $(1,0)$ curve inside $\TT F_2$. On can still observe the change of slope when an edge crosses the boundary. In the case of $\TT F_1$, an edge of slope $1$ which crosses the boundary on the right comes back from the left with slope $0$. In the case of $\TT F_2$, and edge of slope $1$ which crosses on the right comes back on the left with new slope $-1$.
\end{expl}

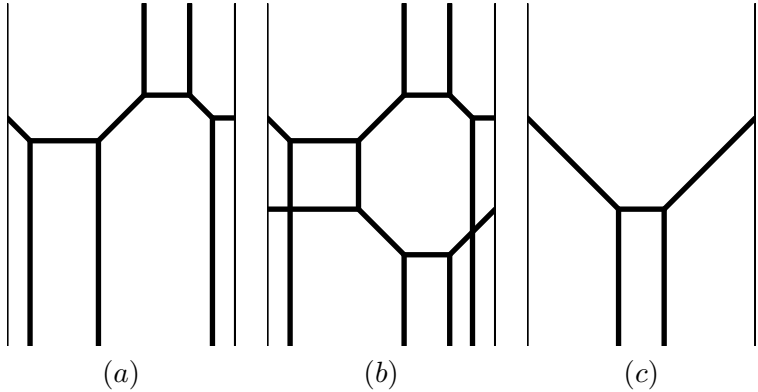
\begin{figure}
\begin{center}
\begin{tabular}{ccc}
\begin{tikzpicture}[line cap=round,line join=round,>=triangle 45,x=0.3cm,y=0.3cm]
\clip(0,0) rectangle (10,15);
\draw [line width=1pt] (0,0)-- (0,15);
\draw [line width=1pt] (10,0)-- (10,15);

\draw [line width=2pt] (0,10)-- (1,9);
\draw [line width=2pt] (1,9)-- (4,9);
\draw [line width=2pt] (4,9)-- (6,11);
\draw [line width=2pt] (6,11)-- (8,11);
\draw [line width=2pt] (8,11)-- (9,10);
\draw [line width=2pt] (9,10)-- (10,10);

\draw [line width=2pt] (1,9)-- (1,0);
\draw [line width=2pt] (4,9)-- (4,0);
\draw [line width=2pt] (6,11)-- (6,15);
\draw [line width=2pt] (8,11)-- (8,15);
\draw [line width=2pt] (9,10)-- (9,0);
\begin{scriptsize}

\end{scriptsize}
\end{tikzpicture}
&
\begin{tikzpicture}[line cap=round,line join=round,>=triangle 45,x=0.3cm,y=0.3cm]
\clip(0,0) rectangle (10,15);
\draw [line width=1pt] (0,0)-- (0,15);
\draw [line width=1pt] (10,0)-- (10,15);

\draw [line width=2pt] (0,10)-- (1,9);
\draw [line width=2pt] (1,9)-- (4,9);
\draw [line width=2pt] (4,9)-- (6,11);
\draw [line width=2pt] (6,11)-- (8,11);
\draw [line width=2pt] (8,11)-- (9,10);
\draw [line width=2pt] (9,10)-- (10,10);

\draw [line width=2pt] (0,6)-- (4,6);

\draw [line width=2pt] (4,6)-- (6,4);
\draw [line width=2pt] (8,4)-- (10,6);
\draw [line width=2pt] (6,4)-- (8,4);

\draw [line width=2pt] (1,9)-- (1,0);
\draw [line width=2pt] (4,9)-- (4,6);
\draw [line width=2pt] (6,11)-- (6,15);
\draw [line width=2pt] (8,11)-- (8,15);
\draw [line width=2pt] (9,10)-- (9,0);

\draw [line width=2pt] (6,4)-- (6,0);
\draw [line width=2pt] (8,4)-- (8,0);
\begin{scriptsize}

\end{scriptsize}
\end{tikzpicture}
&
\begin{tikzpicture}[line cap=round,line join=round,>=triangle 45,x=0.3cm,y=0.3cm]
\clip(0,0) rectangle (10,15);
\draw [line width=1pt] (0,0)-- (0,15);
\draw [line width=1pt] (10,0)-- (10,15);

\draw [line width=2pt] (0,10)-- (4,6);
\draw [line width=2pt] (4,6)-- (6,6);
\draw [line width=2pt] (6,6)-- (10,10);
\draw [line width=2pt] (4,6)-- (4,0);
\draw [line width=2pt] (6,6)-- (6,0);

\begin{scriptsize}

\end{scriptsize}
\end{tikzpicture}\\
$(a)$ & $(b)$ & $(c)$ \\
\end{tabular}
\caption{\label{figure section curve TF1 2}Examples of tropical curves inside $\TT F_1$ ($(a)$ and $(b)$) and in $\TT F_2$ $(c)$.}
\end{center}
\end{figure}

\begin{expl}\label{example superabundant curves}
Notice that in the case of $\TT F_{0,\alpha}$, if $\alpha\in\QQ l$, there exist curves that do not intersect $E_0$ nor $E_\infty$. They are elliptic curves that go around the cylinder direction at least once with the right slope. In Figure \ref{figure superabundant curve}, the unique edge has slope $\frac{1}{2}$, and the two sides of the strip are identified with a translation by $l/2$. Those curves are superabundant. Such curves are called \textit{superabundant loops}. They are characterized in Proposition \ref{proposition superabundant loops}.
\end{expl}

%\begin{defi}
%The curves considered in example \ref{example superabundant curves} are called \textit{superabundant loops}.
%\end{defi}

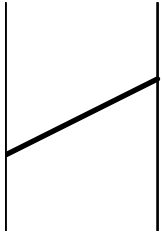
\begin{figure}
\begin{center}
\begin{tikzpicture}[line cap=round,line join=round,>=triangle 45,x=0.2cm,y=0.2cm]
\clip(0,0) rectangle (11,15);
\draw [line width=1pt] (0,0)-- (0,15);
\draw [line width=1pt] (10,0)-- (10,15);
\draw [line width=2pt] (0,5)-- (10,10);
\begin{scriptsize}

\end{scriptsize}
\end{tikzpicture}
\end{center}
\caption{\label{figure superabundant curve} Example of a superabundant loop inside $\TT F_{0,l/2}$.}
\end{figure}

\subsection{Cutting procedure}
\label{section cutting procedure}

\subsubsection{Relation to tropical curve in $\RR^2$.} Each cylinder is obtained as a quotient of $\RR^2$ by a $\ZZ$-action. The action is generated by $\varphi_{0,\alpha}:(x,y)\mapsto(x+l,y+\alpha)$ in the case of $\TT F_{0,\alpha}$, and by $\varphi_\delta:(x,y)\mapsto(x+l,y+\delta x)$ in the case of $\TT F_\delta$. We just denote it by $\varphi$. The projection $\RR^2\to\TT F=\RR^2/\langle\varphi\rangle$ is a covering map. It is not possible to lift right away tropical curves inside $\TT F$ to tropical curves inside $\RR^2$. We explain how to do so by ``cutting" the curve inside $\TT F$.

\smallskip

Let $h:\Gamma\to\TT F$ be a parametrized tropical curve. Let $\Q\subset\Gamma$ be a set of points located on the edges such that $\Gamma-\Q$ is connected and without cycles. In particular, the image of $\pi_1(\Gamma-\Q)$ inside $\pi_1(\TT F)\simeq\ZZ$ is trivial. For each point $q\in\Q$, $(\Gamma-\Q)\cup\{q\}$ contains a unique loop. Let $\lambda_q\in\pi_1(\TT F)\simeq\ZZ$ be the homotopy class realized by the loop inside $\pi_1(\TT F)\simeq\ZZ$. Let $\Q'=\{q\in\Q:\lambda_q\neq 0\}$. The set $\Q'$ is called \textit{admissible}: it is a minimal set of points such that $\Gamma-\Q'$ is connected, and its image contractible inside $\TT F$.

\smallskip

Now, let $\Gamma'$ be the abstract tropical curve where the points from $\Q'$ have been removed and the edges containing a point from $\Q'$ have been replaced by a pair of unbounded ends. As the image of $\pi_1(\Gamma-\Q')$ is trivial, it is possible to lift $h|_{\Gamma-\Q'}$ into a map $\Gamma-\Q'\to\RR^2$. Extending to infinity the edges obtained by removing points from $\Q'$, we get a parametrized tropical curve $h':\Gamma'\to\RR^2$ of genus $g(\Gamma)-|\Q'|$. Moreover, from $\Q'$, each non-vertical unbounded end gets a marked point. Furthermore, unbounded ends are paired together according to the point $q\in\Q'$ they contain. Inside each pair $\{e,e'\}$ of ends having slopes $u_e,u_{e'}$, containing marked points $q_e,q_{e'}\in\RR^2$, small neighbourhoods of $q_e$ and $q_{e'}$ have the same image under $\RR^2\to\TT F$: if $q\in\Q'$ is their common image, up to relabeling one has $q_{e'}=\varphi^{\lambda_q}(q_e)$ and $u_{e'}=\pm(\mathrm{d}\varphi)^{\lambda_q}(u_e)$.

\smallskip

Conversely, take a parametrized tropical curve $h':\Gamma'\to\RR^2$ with this additional data and assumption: each non-vertical unbounded end contains a marked point, they are paired together, and have the same projection under $\RR^2\to\TT F$ near their marked point. Then we get a parametrized tropical curve $h:\Gamma\to\TT F$: it suffices to remove the non-vertical unbounded ends until their marked point, and project into $\TT F$. One can merge the marked points two by two, and the assumption ensures that the balancing condition is satisfied. The obtained tropical curve is of genus $g(\Gamma')$ plus the number of pairs.

\begin{rem}
In concrete terms, the cutting procedure just consists in cutting the curve so that a lift is possible, and then unfold the curve to the universal cover, extending the cut edges to unbounded ends.
\end{rem}

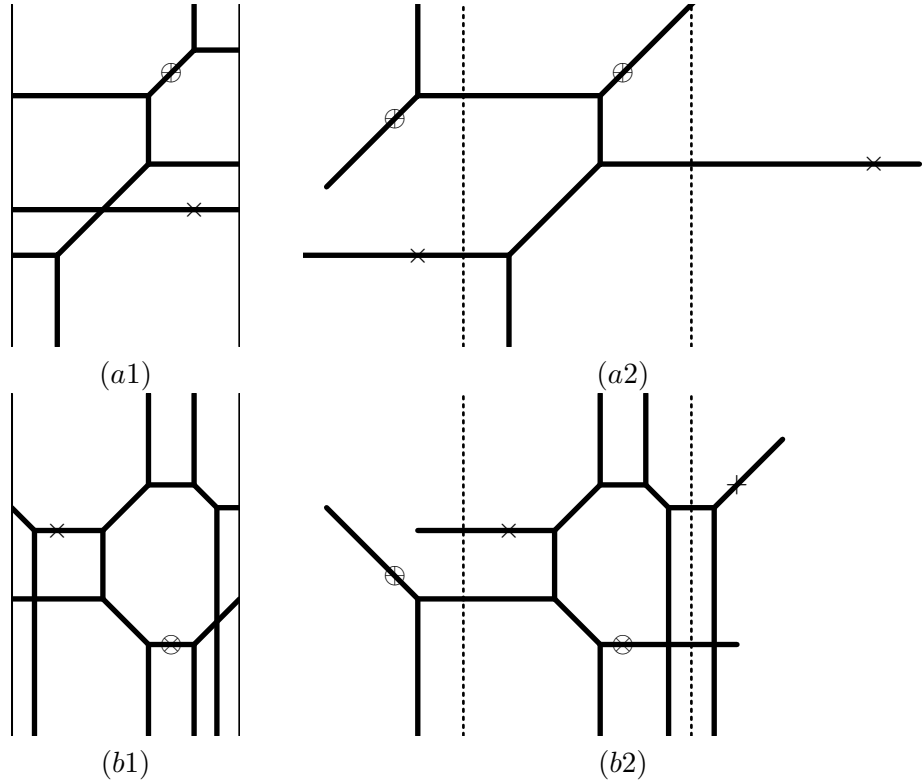
\begin{figure}[h]
\begin{center}
\begin{tabular}{ccc}
\begin{tikzpicture}[line cap=round,line join=round,>=triangle 45,x=0.3cm,y=0.3cm]
\clip(0,0) rectangle (10,15);
\draw [line width=1pt] (0,0)-- (0,15);
\draw [line width=1pt] (10,0)-- (10,15);

\draw [line width=2pt] (0,4)--++(2,0)--++(4,4)--++(4,0);
\draw [line width=2pt] (0,6)--++(10,0);
\draw [line width=2pt] (0,11)--++(6,0)--++(2,2)--++(2,0);
\draw [line width=2pt] (2,0)-- (2,4);
\draw [line width=2pt] (6,8)-- (6,11);
\draw [line width=2pt] (8,13)-- (8,15);

\draw (8,6) node {$\times$};
\draw (7,12) node {$\oplus$};

\begin{scriptsize}

\end{scriptsize}
\end{tikzpicture}
& &
\begin{tikzpicture}[line cap=round,line join=round,>=triangle 45,x=0.3cm,y=0.3cm]
\clip(-7,0) rectangle (21,15);
%\draw [line width=1pt,dotted] (-10,0)--++(0,15);
\draw [line width=1pt,dotted] (0,0)--++(0,15);
\draw [line width=1pt,dotted] (10,0)--++(0,15);
%\draw [line width=1pt,dotted] (20,0)--++(0,15);

\draw [line width=2pt] (-11,4)--(2,4)--++(4,4)--++(14,0);
\draw [line width=2pt] (0,11)--++(6,0)--++(6,6);
\draw [line width=2pt] (0,11)--++(-2,0)--++(-4,-4);
\draw [line width=2pt] (2,0)-- (2,4);
\draw [line width=2pt] (6,8)-- (6,11);
\draw [line width=2pt] (-2,11)--++ (0,5);

\draw (18,8) node {$\times$};
\draw (-2,4) node {$\times$};
\draw (7,12) node {$\oplus$};
\draw (-3,10) node {$\oplus$};

\begin{scriptsize}

\end{scriptsize}
\end{tikzpicture} \\
$(a1)$ & & $(a2)$ \\
\begin{tikzpicture}[line cap=round,line join=round,>=triangle 45,x=0.3cm,y=0.3cm]
\clip(0,0) rectangle (10,15);
\draw [line width=1pt] (0,0)-- (0,15);
\draw [line width=1pt] (10,0)-- (10,15);

\draw [line width=2pt] (0,10)-- (1,9);
\draw [line width=2pt] (1,9)-- (4,9);
\draw [line width=2pt] (4,9)-- (6,11);
\draw [line width=2pt] (6,11)-- (8,11);
\draw [line width=2pt] (8,11)-- (9,10);
\draw [line width=2pt] (9,10)-- (10,10);

\draw [line width=2pt] (0,6)-- (4,6);

\draw [line width=2pt] (4,6)-- (6,4);
\draw [line width=2pt] (8,4)-- (10,6);
\draw [line width=2pt] (6,4)-- (8,4);

\draw [line width=2pt] (1,9)-- (1,0);
\draw [line width=2pt] (4,9)-- (4,6);
\draw [line width=2pt] (6,11)-- (6,15);
\draw [line width=2pt] (8,11)-- (8,15);
\draw [line width=2pt] (9,10)-- (9,0);

\draw [line width=2pt] (6,4)-- (6,0);
\draw [line width=2pt] (8,4)-- (8,0);

\draw (2,9) node {$\times$};
\draw (7,4) node {$\otimes$};

\begin{scriptsize}

\end{scriptsize}
\end{tikzpicture}
& &\begin{tikzpicture}[line cap=round,line join=round,>=triangle 45,x=0.3cm,y=0.3cm]
\clip(-7,0) rectangle (21,15);
\draw [line width=1pt,dotted] (0,0)-- (0,15);
\draw [line width=1pt,dotted] (10,0)-- (10,15);

\draw [line width=2pt] (10,10)-- (11,10);
\draw [line width=2pt] (-2,9)-- (4,9);
\draw [line width=2pt] (4,9)-- (6,11);
\draw [line width=2pt] (6,11)-- (8,11);
\draw [line width=2pt] (8,11)-- (9,10);
\draw [line width=2pt] (9,10)-- (10,10);
\draw [line width=2pt] (11,10)-- (14,13);

\draw [line width=2pt] (0,6)-- (4,6);

\draw [line width=2pt] (4,6)-- (6,4);
\draw [line width=2pt] (-2,6)-- (0,6);
\draw [line width=2pt] (6,4)-- (12,4);
\draw [line width=2pt] (-6,10)-- (-2,6);

\draw [line width=2pt] (11,10)-- (11,0);
\draw [line width=2pt] (4,9)-- (4,6);
\draw [line width=2pt] (6,11)-- (6,15);
\draw [line width=2pt] (8,11)-- (8,15);
\draw [line width=2pt] (9,10)-- (9,0);

\draw [line width=2pt] (6,4)-- (6,0);
\draw [line width=2pt] (-2,6)-- (-2,0);

\draw (2,9) node {$\times$};
\draw (12,11) node {$+$};
\draw (7,4) node {$\otimes$};
\draw (-3,7) node {$\oplus$};

\begin{scriptsize}

\end{scriptsize}
\end{tikzpicture}\\
$(b1)$ & & $(b2)$ \\
\end{tabular}
\caption{\label{figure cutting procedure}Examples of application of the cutting process for a curve inside $\TT F_{0,\alpha}$ in $(a)$ and in $\TT F_1$ for $(b)$.}
\end{center}
\end{figure}

\begin{expl}
On Figure \ref{figure cutting procedure} are two examples of tropical curve inside respectively some $\TT F_{0,\alpha}$ and $\TT F_1$ with an admissible set of points, and the corresponding tropical curve inside $\RR^2$. On $(a)$, the map $\varphi$ is just a translation. The points $\oplus$ are identified by $\varphi$, and the points $\times$ are identified by $\varphi^2$. On example $(b)$, notice that the map $\varphi$ is now a shear transformation that changes non-vertical slopes.
\end{expl}

Finally, we notice that a deformation of $h:\Gamma\to\TT F$ with a choice of admissible set $\Q'$ leads to a deformation of the cut curve $h':\Gamma'\to\RR^2$ provided the points of $\Q'$ never collide with a vertex. Conversely, a deformation of $h':\Gamma'\to\RR^2$ with the marked points on non-vertical ends that continue to match the conditions induce a deformation of $h:\Gamma\to\TT F$ provided the marked points never collide with vertices.

\subsubsection{Menelaus relation.} We have the following result that is the tropical analogue of Proposition \ref{theorem complex menelaus}.

\begin{theo}\label{theorem menelaus tropical bundle}
Let $h:\Gamma\to\TT F$ be a parametrized tropical curve of bidegree $(d_1,d_2)$ inside the total space of a line bundle $\L_D$ for some divisor $D$. Let $x_i\in E$ be the projection of the unbounded ends of $\Gamma$, with weight $w_i$, which is chosen with a sign such that it is positive or negative according to whether it is oriented up or down. Then the divisor $\sum w_i x_i$ is equivalent to $d_1 D$.
\end{theo}

\begin{proof}
Assume $\delta\neq 0$ and that $D=(0,\delta)\in E\times\ZZ$, so that the description of $\TT F$ is given in the preceding section. Apply the cutting procedure to get a tropical curve $\widetilde{h}:\widetilde{\Gamma}\to\RR^2$. According to the tropical Menelaus theorem \cite{mikhalkin2017quantum}, we know that the sum of the moments of the coordinates of the unbounded ends is $0$, where the moment of an unbounded end of slope $(u,v)$ is the evaluation of $\det((u,v),-)$ at any of its points.

Let $x_i$ be the abscissa of the unbounded ends of $\Gamma$, with $w_i$ their respective weight, which is positive or negative according to whether the end points up or down. Let $(l,y_j)$ be the intersection points between $\Gamma$ and the component of the boundary of the strip which is on the right, and $(u_j,v_j)$ the slope of $h$ at the point. The slope at the corresponding point $(0,y_j)$ on the left is $(-u_j,\delta u_j-v_j)$. The moment of $\widetilde{\Gamma}$ at these points are respectively $u_jy_j-v_jl$ and $-u_jy_j$. First, the balancing condition guarantees that $\sum w_i=\delta d_1$, which is the degree of $d_1 D$. By tropical Menelaus theorem in the plane, we get that
$$\sum w_i x_i=-\sum_j (u_jy_j-v_jl-u_jy_j)\equiv 0\mod l.$$
If $\delta=0$, we are inside $\TT F_{0,\alpha}$ and the points which are identified are $(l,y_j)$ and $(0,y_j+\alpha)$. Thus, the moments become respectively $u_jy_j-v_jl$ and $-u_jy_j-u_j\alpha$. We get
$$\sum w_i x_i\equiv \alpha\sum u_j\mod l,$$
which yields the result since $\sum u_j=d_1$. For other line bundles $\L_D$ and other divisors $D$, the result follows by translation.
\end{proof}

\begin{rem}
Similarly to the tropical Menelaus in the plane, Theorem \ref{theorem menelaus tropical bundle} tells us that there is a relation between the position of the unbounded ends of a curve, and they cannot be chosen freely. The equality is only mod $l$ because the coordinates are numbers in $\RR/l\ZZ$.
\end{rem}

\section{Enumerative problems and tropical invariants}

\subsection{Enumerative problems}

Let $\M_{g,n}(\TT F,(d_1,d_2),\mu_0,\mu_\infty)$ be the moduli space of irreducible parametrized tropical curves of genus $g$, bidegree $(d_1,d_2)$, tangency profile $(\mu_0,\mu_\infty)$, with $n$ marked points. A curve in $\M_{g,n}(\TT F,(d_1,d_2),\mu_0,\mu_\infty)$ has exactly $|\mu_0|+|\mu_\infty|$ ends. Assuming the curve is trivalent, it has exactly $3g-3+|\mu_0|+|\mu_\infty|+n$ bounded edges. The $g$ cycles of the graph give $2g$ conditions on the lengths of the bounded edges. Adding the translation, we get an expected dimension of
$$\dim\M_{g,n}(\TT F,(d_1,d_2),\mu_0,\mu_\infty)= |\mu_0|+|\mu_\infty|+g-1+n.$$

It is possible that some combinatorial types vary in a space whose dimension is strictly bigger than the expected dimension. This happens when the cycles do not impose independent conditions on the edge lengths. Such combinatorial types are called \textit{superabundant}. However, as in the planar case handled in \cite{mikhalkin2005enumerative}, it can be shown that they do not contribute any solution to the enumerative problems, and one can thus restrict to the space of \textit{simple tropical curves} defined below, which vary in a space of the expected dimension. In our problem, the only superabundant curves that may appear are superabundant loops from example \ref{example superabundant curves}. Those are handled using Proposition \ref{proposition superabundant loops}.

\begin{defi}
We say that the combinatorial type of a parametrized tropical curve $h:\Gamma\to\TT F$ is \textit{simple} if $\Gamma$ is trivalent and $h$ is an immersion.
\end{defi}

The fact that $h$ is an immersion means that the parametrized curve $h:\Gamma\to\TT F$ cannot have any edge mapped to a point, nor a a flat vertex, \textit{i.e.} a vertex whose adjacent edges are all mapped to a line.

\begin{prop}\label{proposition superabundant loops}
Let $\TT F$ be a cylinder that contains a tropical curve $h:\Gamma\to\TT F$ that has no unbounded end and for which $h$ is an immersion. Then, $\TT F$ is some $\TT F_{0,\alpha}$ for $\alpha\in\QQ l$, and $\Gamma$ is an elliptic curve mapped to $\TT F_{0,\alpha}$ having constant slope proportional to to $(l,-\alpha)$.
\end{prop}

\begin{proof}
Assume $h:\Gamma\to\TT F$ of bidegree $(d_1,d_2)\neq (0,0)$ has no unbounded end. We thus have $d_2=0$ since it has no top ends, and $\delta=0$ since the the curve has $\delta d_1$ lower ends and $d_1\neq 0$. Thus, $\TT F$ is some $\TT F_{0,\alpha}$.

\smallskip

Then, apply the cutting procedure for some admissible set $\Q'$ to get a curve inside $\RR^2$. The pairs of ends have slopes $\pm v_e$ for some $v_e\in\ZZ_{>0}\times\ZZ$, and the marked points on it differ by some integer multiple of $(l,-\alpha)$. For each pair of ends, let $(x_e,y_e)$ be the marked point from $\Q'$ on the end directed by $-v_e$, and $(x_e+k_el,y_e-k_e\alpha)$ the point on the end directed by $v_e$, for some $k_e\in\ZZ$. Menelaus relation for $\Gamma'$ ensures that
$$\sum_e \det(-v_e,(x_e,y_e))+\det(v_e,(x_e+k_el,y_e+k_e\alpha))=0.$$
Thus, $\det\left(\sum_e k_e v_e,(l,-\alpha)\right)=0$. As $\sum_e k_e v_e\in\ZZ_{>0}\times\ZZ$, we have that $(l,-\alpha)$ has rational slope and $\alpha\in\QQ l$. In particular, $\TT F_{0,\alpha}$ contains tropical elliptic curves going once around the cylinder direction: the superabundant loops of slope $(l,-\alpha)$ as in example \ref{example superabundant curves}. These are parametrized by the coordinate of their intersection point with a fixed fiber.

\smallskip

To conclude, assume $\Gamma$ has a at least one vertex. Consider a superabundant loop $\Upsilon$ of slope $(l,-\alpha)$ of highest coordinate that also passes through a vertex $V$ of $\Gamma$. It is possible since there are only a finite number of vertices in $\Gamma$. Not all the adjacent edges of $\Gamma$ to $V$ can be contained in the same line since $h$ is an immersion. But if some edge does not have slope $(l,-\alpha)$, by balancing condition, there is at least one edge that goes above $\Upsilon$, and as $\Gamma$ has no unbounded ends, it has to meet another vertex $W$. A superabundant loop passing through $W$ would be higher than $\Upsilon$, contradicting the assumption on the choice of $\Upsilon$. Thus, all the edges of $\Gamma$ are contained in $\Upsilon$. As $h$ is an immersion, $\Gamma$ has no vertices. It is thus of genus $1$ and is a superabundant loop.
\end{proof}

\begin{rem}
Assuming that $\alpha\notin\QQ l$ ensures that there are no superabundant loops. If $(d_1,d_2)$ is fixed, it is only sufficient to assume that $\alpha\notin \frac{1}{d_1!}\ZZ l$ to ensure that curves of bidegree $(d_1,d_2)$ cannot have any irreducible component being a superabundant loop.
\end{rem}

We have the following evaluation map:
$$\begin{array}{rccl}
\mathrm{ev} & : \M_{g,n}(\TT F,(d_1,d_2),\mu_0,\mu_\infty) & \longrightarrow & \TT F^n \\
 & (\Gamma,h,x_1,\dots,x_n) & \longmapsto & (h(x_1),\dots,h(x_n)) \\
\end{array},$$
that sends a parametrized curve to the position of its marked points. It is also possible to evaluate the position of some unbounded ends, and the evaluation map then takes values inside $\TT F^n\times\TT E^m$, where $m$ is the number of evaluated ends.

\medskip

The following proposition computes the dimension of the moduli space of curves, throwing away the case where superabundant loops occur.

\begin{prop}\label{proposition dimension count}
Assume $\TT F$ is some $\TT F_\delta$ or some $\TT F_{0,\alpha}$ for $\alpha\notin\QQ l$. The dimension of the subspace of $\M_{g,n}(\TT F,(d_1,d_2),\mu_0,\mu_\infty)$ parametrizing the curves having a simple combinatorial type is  $|\mu_0|+|\mu_\infty|+g-1$. For combinatorial types where $\Gamma$ is not trivalent anymore but $h$ is still an immersion, the dimension is strictly less.
\end{prop}

\begin{proof}
Assume $h:\Gamma\to\TT F$ is simple. Using the cutting procedure with an admissible set $\Q'$, we get a curve $h':\Gamma'\to\RR^2$ with a pairing between the non-vertical ends. The deformations of $h:\Gamma\to\TT F$ correspond to deformations of $h':\Gamma'\to\RR^2$ that keep satisfying the \textit{gluing condition} from the cutting procedure: paired ends coincide under the projection $\RR^2\to\TT F$. Recall that for an end with outing slope $u_e\in\ZZ^2$, its \textit{moment} $\mu_e$ is $\det(u_e,P)$, where $P$ is any point on $e$, \textit{e.g.} $q_e$ from the cutting process. The gluing condition amounts to find deformations of $h:\Gamma'\to\RR^2$ such that for each pair $\{e,e'\}$ the quantity $\mu_e+\mu_{e'}$ remains constant.

\smallskip

The moments satisfy the Menelaus relation: $\sum_e\mu_e=0$. It is in fact the only relation between them: if we had another relation, we could use both relations to get a relation not involving all the moments. Then, let $e_0$ be one of the end not involved in the relation. For each other end $e\neq e_0$, as $\Gamma'$ is trivalent, it is possible to deform a path from $e_0$ to $e$. This changes the value of $\mu_e$ and $\mu_{e_0}$ but not for any other end, preventing a relation not involving all the ends to exist.

\smallskip

To conclude, according to \cite{mikhalkin2005enumerative}, $h':\Gamma'\to\RR^2$ varies in a space of dimension $|\mu_0|+|\mu_\infty|+2|\Q'|+g(\Gamma')-1$. Thanks to Proposition \ref{proposition superabundant loops}, there is at least one vertical unbounded end. The gluing condition thus impose $|\Q'|$ independent conditions. As $g(\Gamma)=g(\Gamma')+|\Q'|$, we get the result for simple combinatorial types.

\medskip

We now assume that $h$ is still an immersion but $\Gamma$ is not trivalent anymore. Furthermore, we can assume that no vertex can be split into two vertices adding an edge of slope $0$: if it were the case, we could split the vertex $V$ into $V_1$ and $V_2$ adding an artificial bounded edge of slope $0$ and then delete it. The vertices $V_1$ and $V_2$ are still vertices of the new curve after the splitting, or might become marked points, if they were not at least trivalent. We would get into one of the following situations:
\begin{itemize}[label=-]
\item If the edge was disconnecting, we now have two simple tropical curves $h_j:\Gamma_j\to\TT F$ and genus $g_j$ of smaller degree with tangency profile $(\mu_0^j,\mu_\infty^j)$, for $j=1,2$. By induction on the degree, we can assume they vary in a space of respective dimension at most $|\mu_0^j|+|\mu_\infty^j|+g_j-1$. Deformations of $\Gamma$ are obtained by taking deformations of $\Gamma_1$ and $\Gamma_2$ mapping $V_1$ and $V_2$ to the same image. Thus, as $g=g_1+g_2$, $\Gamma$ varies in a space of dimension at most
$$|\mu_0|+|\mu_\infty|+g-2<|\mu_0|+|\mu_\infty|+g-1.$$
\item If the edge is not disconnecting, we get a new parametrization of the image of $\Gamma$ by a simple tropical curve $\widetilde{h}:\widetilde{\Gamma}\to\TT F$ which is of genus $g-1$. Deformations of $\Gamma$ are the deformations of $\widetilde{\Gamma}$ that map $V_1$ and $V_2$ to the same image. By induction, we can still assume that it varies in a space of dimension at most
$$|\mu_0|+|\mu_\infty|+(g-1)-1<|\mu_0|+|\mu_\infty|+g-2.$$
\end{itemize}

Therefore, we can assume no vertex can be split. Then, it is possible to find a one-parameter family deforming $\Gamma$ into a nearby simple tropical curve: for any non-trivalent vertex, choose a loop or a string between two unbounded ends and deform it, and continue for each remaining non-trivalent vertex. Thanks to the assumption on vertices, no cycle or string can arrive at a vertex and leave with the same slope, so each deformation deletes at least one non-trivalent vertex. If the dimension of the deformation space of $h:\Gamma\to\TT F$ was at least $|\mu_0|+|\mu_\infty|+g-1$, this nearby simple tropical curve would vary in a space of dimension at least one more (using the deformation provided by the one parameter family), which contradicts the dimension count for simple tropical curves.
\end{proof}

\begin{rem}
The proof of the statement concerning non-trivalent tropical curves can be seen as follows: the moduli space $\M_{g,n}(\TT F,(d_1,d_2),\mu_0,\mu_\infty)$ is a polyhedral complex, and cells corresponding to non-trivalent curves are faces of cells corresponding to simple combinatorial types, thus of strictly smaller dimension.
\end{rem}

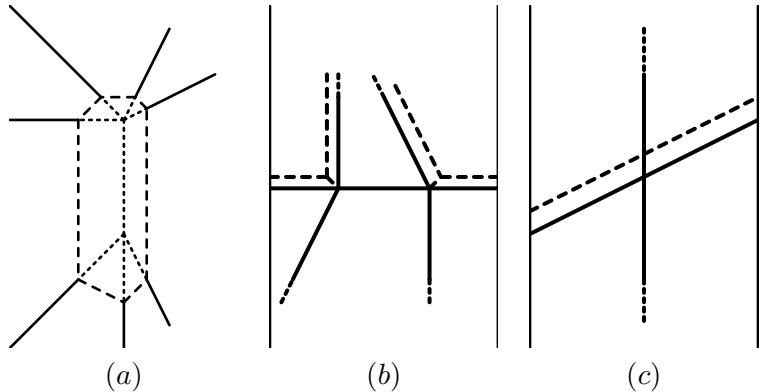
\begin{figure}
\begin{center}
\begin{tabular}{ccc}
\begin{tikzpicture}[line cap=round,line join=round,>=triangle 45,x=0.3cm,y=0.3cm]
\clip(0,0) rectangle (10,15);

\draw [line width=1pt,dotted] (5,10)-- (5,5);

\draw [line width=1pt,dotted] (5,10)-- (3,10);
\draw [line width=1pt,dotted] (5,10)-- (4,11);
\draw [line width=1pt,dotted] (5,10)-- (5.5,11);
\draw [line width=1pt,dotted] (5,10)-- (6,10.5);

\draw [line width=1pt,dotted] (5,5)-- (3,3);
\draw [line width=1pt,dotted] (5,5)-- (5,2);
\draw [line width=1pt,dotted] (5,5)-- (6,3);

\draw [line width=1pt] (3,10)-- (0,10);
\draw [line width=1pt] (4,11)-- (0,15);
\draw [line width=1pt] (5.5,11)-- (7,14);
\draw [line width=1pt] (6,10.5)-- (9,12);

\draw [line width=1pt] (3,3)-- (0,0);
\draw [line width=1pt] (5,2)-- (5,0);
\draw [line width=1pt] (6,3)-- (7,1);

\draw [line width=1pt, dashed] (3,10)-- (3,3);
\draw [line width=1pt, dashed] (3,3)-- (5,2);
\draw [line width=1pt, dashed] (5,2)-- (6,3);
\draw [line width=1pt, dashed] (6,3)-- (6,10.5);
\draw [line width=1pt, dashed] (6,10.5)-- (5.5,11);
\draw [line width=1pt, dashed] (5.5,11)-- (4,11);
\draw [line width=1pt, dashed] (4,11)-- (3,10);
v

\begin{scriptsize}

\end{scriptsize}
\end{tikzpicture}
&
\begin{tikzpicture}[line cap=round,line join=round,>=triangle 45,x=0.3cm,y=0.3cm]
\clip(0,0) rectangle (10,15);
\draw [line width=1pt] (0,0)-- (0,15);
\draw [line width=1pt] (10,0)-- (10,15);

\draw [line width=1.5pt] (0,7)-- (10,7);
\draw [line width=1.5pt] (3,7)-- (3,11);
\draw [line width=1.5pt,dotted] (3,12)-- (3,11);
\draw [line width=1.5pt] (3,7)-- (1,3);
\draw [line width=1.5pt,dotted] (0.5,2)-- (1,3);
\draw [line width=1.5pt] (7,7)-- (5,11);
\draw [line width=1.5pt,dotted] (5,11)-- (4.5,12);
\draw [line width=1.5pt] (7,7)-- (7,3);
\draw [line width=1.5pt,dotted] (7,3)-- (7,2);

\draw [line width=1.5pt,dashed] (0,7.5)-- (2.5,7.5);
\draw [line width=1.5pt,dashed] (2.5,7.5)-- (2.5,12);
\draw [line width=1.5pt,dashed] (2.5,7.5)-- (3,7);
\draw [line width=1.5pt,dashed] (7,7)-- (7.5,7.5);
\draw [line width=1.5pt,dashed] (10,7.5)-- (7.5,7.5);
\draw [line width=1.5pt,dashed] (5.5,11.5)-- (7.5,7.5);

\begin{scriptsize}

\end{scriptsize}
\end{tikzpicture}
&
\begin{tikzpicture}[line cap=round,line join=round,>=triangle 45,x=0.3cm,y=0.3cm]
\clip(0,0) rectangle (10,15);
\draw [line width=1pt] (0,0)-- (0,15);
\draw [line width=1pt] (10,0)-- (10,15);

\draw [line width=1.5pt] (5,3)-- (5,12);
\draw [line width=1.5pt,dotted] (5,3)-- (5,1);
\draw [line width=1.5pt,dotted] (5,14)-- (5,12);
\draw [line width=1.5pt] (0,5)-- (10,10);
\draw [line width=1.5pt, dashed] (0,6)-- (10,11);

\begin{scriptsize}

\end{scriptsize}
\end{tikzpicture}\\
$(a)$ & $(b)$ & $(c)$ \\
\end{tabular}
\caption{\label{figure opening cycle} On $(a)$, the deformation of a flat contractible cycle into a an open cycle. On $(b)$, the deformation of a flat non-contractible cycle containing at least to vertices. On $(c)$, the deformation of a loop.}
\end{center}
\end{figure}

\begin{rem}\label{remark superabundancy degree zero}
Using Proposition \ref{proposition superabundant loops}, the restriction to $\TT F$ being $\TT F_\delta$ or $\TT F_{0,\alpha}$ for $\alpha\notin\QQ l$ prevents the appearance of superabundant loops. The latter are genus $1$ curve without unbounded ends that yet vary in a $1$-dimensional space: they are superabundant. The proof and statement of Proposition \ref{proposition dimension count} can be adapted to work in the setting of $\TT F_{0,\alpha}$ when $\alpha\in\QQ l$. The dimension count for simple combinatorial types is still valid, except when there are no ends: Proposition \ref{proposition superabundant loops} then ensures that the curve is a superabundant loop that varies in a $1$-dimensional space.

\smallskip

Now, the only difference is that when $\Gamma$ is not trivalent anymore, we need to take the following case into account: when splitting the vertices of $\Gamma$ that can be split in a way that disconnects $\Gamma$ into $\Gamma_1$ and $\Gamma_2$, one of the component might have no ends, and thus become a superabundant loop according to \ref{proposition superabundant loops}. Assume $\Gamma_2$ is the superabundant loop. Then the dimension of the space in which $\Gamma$ varies is equal to the dimension in which $\Gamma_1$ varies if $V_1$ is a vertex of $\Gamma_1$, or plus one if $V_1$ is not a vertex of $\Gamma_1$ (meaning $V$ was a quadrivalent vertex). Thus, the combinatorial types that vary in a space of dimension $|\mu_0|+|\mu_\infty|+g-1$ are the simple ones, that may additionally have superabundant loops attached to edges.
\end{rem}

Unless stated otherwise, we consider $\TT F$ as in the assumption of Proposition \ref{proposition dimension count} so that there are no superabundant loops. We then consider the following enumerative problems.

\begin{prob}
How many parametrized tropical curves of bidegree $(d_1,d_2)$ and genus $g$ pass through $\delta d_1+2d_2+g-1$ points in generic position ?
\end{prob}

%The number of irreducible curves is denoted by $N_{g,(d_1,d_2)}^{\TT F,\mathrm{trop}}(\P)$ and the count including reducible curves is denoted by $N_{g,(d_1,d_2)}^{\TT F,\mathrm{trop},\bullet}(\P)$. A priori, it depends on the choice of the point configuration $\P$ as well as on the cylinder. The correspondence theorem ensures that it does not.
Let $\|\mu_0\|+\|\nu_0\|=\delta d_1+d_2$ and $\|\mu_\infty\|+\|\nu_\infty\| =d_2$ be partitions.

\begin{prob}
How many parametrized tropical curves of bidegree $(d_1,d_2)$, ramification profile $(\mu_0+\nu_0,\mu_\infty+\nu_\infty)$ and genus $g$ have $\mu_{0i}$ unbounded ends of slope $(0,-i)$ and $\mu_{\infty i}$ unbounded ends of slope $(0,i)$ which are fixed, and pass through an additional set of $|\nu_0|+|\nu_\infty|+g-1$ points in generic position ?
\end{prob}

%Similarly, the counts are denoted by $N_{g,(d_1,d_2)}^{\TT F,\mathrm{trop}}(\mu_0,\mu_\infty,\nu_0,\nu_\infty)(\P)$ and $N_{g,(d_1,d_2)}^{\TT F,\mathrm{trop},\bullet}(\mu_0,\mu_\infty,\nu_0,\nu_\infty)(\P)$.
Both enumerative problem amount to find the preimages of a generic point configuration by the evaluation map.

\subsection{Correspondence theorem}

Unless stated otherwise, we consider $\TT F$ as in the assumption of \ref{proposition dimension count} so that there are no superabundant loops. First, we prove that the solutions to the preceding enumerative problems need to be simple tropical curves.

\begin{lem}
If the point configuration $\P$ is generic, the tropical curves passing through $\P$ are simple and their number is finite.
\end{lem}

\begin{proof}
We consider the moduli space $\M_{g,n}(\TT F,(d_1,d_2))$, with $n=\delta d_1+2d_2+g-1$ marked points. The enumerative problem consists in finding the antecedent of a generic $\P\in\TT F^n$ by the evaluation map. The moduli space is the union of a finite number of polyedra corresponding to the possible combinatorial types. We want to prove that only the linear maps associated to simple combinatorial types are surjective.

If a combinatorial type does not correspond to an immersion, then, as in the planar case, it is possible to reparametrize the image by an immersion which is of strictly smaller genus or fewer ends. This is done by gluing the parallel adjacent edges, and deleting the contracted edges. Thus, we restrict to the immersed case.

If the combinatorial type corresponds to an immersion, we have seen that the dimension is as expected for a trivalent curve, and strictly less for a non-trivalent one, unless we are in presence of superabundant loops, but the assumption ensures that there are none. In particular, for the immersions of smaller genus or with fewer ends obtained from non-immersed curves, the dimension is also strictly less. Then the evaluation map cannot be surjective for these combinatorial types. The result follows.
\end{proof}

We now count the solutions of the preceding enumerative problems with suitable multiplicities. For a simple  tropical curve $h:\Gamma\to\TT F$, the multiplicity of a vertex $V$ is $|\det(a_V,b_V)|$, where $a_V$ and $b_V$ are the slopes of two out of the three edges adjacent to $V$. The determinant can be taken in any affine chart of the manifold, since its value does not depend on this choice.

\begin{defi}
The multiplicity of a connected simple tropical curve $h:\Gamma\to\TT F$ is
$$m^\CC_\Gamma=\prod_V m_V,$$
where the product runs over the trivalent vertices of $\Gamma$.
\end{defi}

\begin{rem}
In the presence of superabundant loops, \textit{i.e.} in $\TT F_{0,\alpha}$ for $\alpha\in\QQ l$, it is possible to define a multiplicity to a superabundant loop attached to a curve so that the count matches the count when $\alpha\notin\QQ l$. It shall be obtained as follows: slightly deforming $\alpha$ deletes the superabundant loops and it becomes an edge that goes around the cylinder direction several times. The multiplicity with which it needs to be counted encompasses the various possibilities, as shown on Figure \ref{figure deforming superabundant loop}.
\end{rem}

Let $\P\subset\TT F$ be a generic configuration of $\delta d_1+2d_2+g-1$ points inside $\TT F$. Let
$$N^{\delta,\mathrm{trop}}_{g,(d_1,d_2)}(\P)=\sum_{h(\Gamma)\supset\P} m_\Gamma^\CC,$$
where the sum runs over the irreducible parametrized tropical curves $h:\Gamma\to\TT F$ of bidegree $(d_1,d_2)$ and genus $g$ that pass through $\P$. The corresponding sum over the curves including the reducible ones is denoted by $N^{\delta,\mathrm{trop},\bullet}_{g,(d_1,d_2)}(\P)$.

For the second enumerative problem, fix partitions $\mu_0$, $\mu_\infty$, $\nu_0$ and $\nu_\infty$ such that $\|\mu_0\|+\|\nu_0\|=\delta d_1+d_2$ and $\|\mu_\infty\|+\|\nu_\infty\|=d_2$. Let $\P$ be this time a configuration of $|\nu_0|+|\nu_\infty|+g-1$ points, for each $i$, $|\mu_{0i}|$ points on $E_0$ and $|\mu_{\infty i}|$ points on $E_\infty$. Let
$$N^{\delta,\mathrm{trop}}_{g,(d_1,d_2)}(\mu_0,\mu_\infty,\nu_0,\nu_\infty)(\P)=\frac{1}{I^{\mu_0}I^{\mu_\infty}}\sum_{h(\Gamma)\supset\P} m_\Gamma^\CC,$$
where the sum runs over the irreducible parametrized tropical curves $h:\Gamma\to\TT F$ satisfying the following:
\begin{itemize}[label=-]
\item They are of bidegree $(d_1,d_2)$, genus $g$ and tangency profile $(\mu_0+\nu_0,\mu_\infty+\nu_\infty)$.
\item They pass through $\P$.
\item For any $i$, the curve pass through each of the $\mu_{0i}$ (resp. $\mu_{\infty i}$) marked points on $E_0$ (resp. $E_\infty$) with an end of weight $i$.
\end{itemize}
The count including reducible curves is denoted with a $\bullet$.

\begin{rem}
The counts a priori depend on the choice of the cylinder $\TT F$. In fact they only depend on the cylinder through the degree $\delta$, and that is why only the latter is emphasized in the notation.
\end{rem}

We have the following correspondence theorem that relates the enumerative problems over $\CC F$, and $\TT F$.

\begin{theo}\label{theorem correspondence}
\begin{itemize}[label=$\circ$]
\item The count $N^{\delta,\mathrm{trop}}_{g,(d_1,d_2)}(\P)$ does not depend on the choice of $\P$ as long as it is generic, and it is equal to $N^{\delta,\mathrm{cpx}}_{g,(d_1,d_2)}$ (resp. with $\bullet$).
\item The count $N^{\delta,\mathrm{trop}}_{g,(d_1,d_2)}(\mu_0,\mu_\infty,\nu_0,\nu_\infty)(\P)$ does not depend on the choice of $\P$ as long as it is generic, and it is equal to $N^{\delta,\mathrm{cpx}}_{g,(d_1,d_2)}(\mu_0,\mu_\infty,\nu_0,\nu_\infty)$ (resp. with $\bullet$).
\end{itemize}
\end{theo}

\begin{proof}
For the curves which are not superabundant, the proof of this correspondence theorem follows exactly the same steps as other versions of the correspondence theorem that take place inside toric surfaces. This follows from the fact that it is possible to degenerate the surfaces $\CC F$ into a union of toric surfaces glued over their toric divisors, and endow it with a smooth log structure. We explain the main steps and refer the reader to \cite{nishinou2006toric} for more details. %The case of the superabundant loop is treated separately.

\medskip

The union of the images of all tropical curves passing through $\P$ induces a polyedral subdivision of the strip $[0;l]\times\RR$. Up to adding fibers to this union, we can assume that for each vertical edge in the subdivision, the whole vertical line containing it is in the subdivision. Using the $\ZZ$ action generated by the map $(x,y)\mapsto(x+l,y+\delta x)$, it extends periodically to a polyhedral subdivision $\Xi_1$ of $\RR^2$. If $\delta=0$, the action is rather generated by $(x,y)\mapsto(x+l,y-\alpha)$. For any value of $(d_1,d_2)$, we can assume that $\alpha,l\in\QQ$ but such that there are no superabundant loop of bidegree $(k,0)$ for any $k\leqslant d_1$. By periodically extending the subdivision of $[0;l]$, we also get a subdivision of $\RR$. We consider the cones over both subdivisions:
\begin{itemize}[label=$\circ$]
\item The cone over the subdivision in $\RR$ is a fan in $\RR^2$ whose rays are generated by $(1,p_i)$, where $p_i$ is a point of the above subdivision of $\RR$. It is denoted by $\Sigma$, and is depicted on Figure \ref{figure fan of sigma}. The fan $\Sigma$ defines an almost toric variety $\T$: it is a partial compactification of $(\CC^*)^2$ by a chain of $\CC P^1$, glued along their toric boundary. The ``almost toric" is because the variety is not constructed from a finite fan. This non-finite fan projects to $\RR_+$, giving a map from $\T$ to $\CC$. The variety $\T$ is then fibered over $\CC$, the special fiber being the chain of $\CC P^1$.

\begin{figure}[h]
\begin{center}
\begin{tikzpicture}[line cap=round,line join=round,>=triangle 45,x=0.2cm,y=0.2cm]
\clip(-10,-1) rectangle (10,15);
\draw [line width=1pt,->] (-10,0)-- (10,0);
\draw [line width=1pt,->] (0,-1)-- (0,10);

\draw [line width=1.5pt] (0,0)-- (-10,10);
\draw [line width=1.5pt] (0,0)-- (-8,10);
\draw [line width=1.5pt] (0,0)-- (-6,10);
\draw [line width=1.5pt] (0,0)-- (-4,10);
\draw [line width=1.5pt] (0,0)-- (-2,10);
\draw [line width=1.5pt] (0,0)-- (0,10);
\draw [line width=1.5pt] (0,0)-- (2,10);
\draw [line width=1.5pt] (0,0)-- (4,10);
\draw [line width=1.5pt] (0,0)-- (6,10);
\draw [line width=1.5pt] (0,0)-- (8,10);
\draw [line width=1.5pt] (0,0)-- (10,10);

\draw [line width=1.5pt,dotted] (6,3)-- (9,3);
\draw [line width=1.5pt,dotted] (8,7)-- (9,7);
\draw [line width=1.5pt,dotted] (-6,3)-- (-9,3);
\draw [line width=1.5pt,dotted] (-8,7)-- (-9,7);

\begin{scriptsize}

\end{scriptsize}
\end{tikzpicture}
\caption{\label{figure fan of sigma} The fan $\Sigma$ defining the variety $\T$.}
\end{center}
\end{figure}
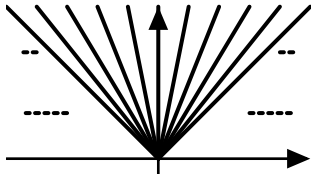

\medskip

The action of $\ZZ$ on $(\CC^*)^2$ generated by the map $(t,z)\mapsto(t,a t^l z)$ descends to an action on the fan. The action on $\RR^2$ is given by the following matrix:
$$\begin{pmatrix}
1 & 0 \\
l & 1 \\
\end{pmatrix},$$
\textit{i.e.} the shear transformation of size $l$ parallel to the horizontal direction. Taking the quotient by the action, the quotient of $\T$ is a family of elliptic curves $\E_t=\CC^*/\langle a t^l\rangle$ that degenerates to the central fiber $\E_0$, which is a chain of $\CC P^1$ glued along their toric boundary. The difference is that the chain makes a cycle.

\item Similarly, we consider the cone over $\Xi_1$ gives a fan inside $\RR^3$. A cone of $\Xi$ is generated by $(1,\sigma)$, where $\sigma$ is some cell of $\Xi_1$. The fan $\Xi$ defines an almost toric variety $\T'$ of dimension three. It is a partial compactification of $(\CC^*)^3$ whose boundary consists of toric surfaces glued along their toric divisors.

The fan $\Xi_1$ is mapped to $\Sigma$ by $(t,x,y)\mapsto(t,x)$. We consider the $\ZZ$-action on $(\CC^*)^3$ generated by the map
$$(t,z,w)\longmapsto(t,a t^l z,z^\delta w).$$
This action is compatible with the previous $\ZZ$-action on $(\CC^*)^2$, and the projection $(t,z,w)\mapsto t$. The quotient of $\T'$ by the $\ZZ$-action is then a family $\X_t$ that admits a factorization $\X_t\to\E_t$. This maps realizes $\X_t$ as a line bundle of degree $\delta$ over the elliptic curve $\E_t$.

The central fiber $\X_0$ is a union of toric surfaces glued along their toric boundary.
\end{itemize}

The previous description extends the setting of a degeneration of toric surfaces to a more general setting. It is possible to be more general by considering the quotient of $(\CC^*)^3$ by a map $(t,z,w)\mapsto(t,a t^l z,b t^a z^\delta w)$, but up to a change of basis, we can assume that the map is of the assumed form. In the case of a degree $0$ line bundle, it is of the form
$$(t,z,w)\mapsto(t,a t^l z,b t^{\alpha} w).$$

We have obtained a family of elliptic curves $\E_t$ that degenerates to a union of $\CC P^1$ glued over their toric divisors, and a degeneration of the degree $\delta$ line bundles over $\E_t$ to a union of toric varieties glued over their toric divisors. We can now apply the method from \cite{nishinou2006toric} to get a correspondence theorem. We refer the reader to it for more details. We summarize the main steps:
\begin{itemize}[label=-]
\item Endow the special fiber $\X_0$, which is a union of toric surfaces glued along their toric divisors, with a log-structure obtained from the toric construction.
\item Find the \textit{prelog-curves} in the special fiber $\X_0$, \textit{i.e.} the reducible nodal curves that matches the point constraints. These curves are the limit curves of the families of curves inside $\X_t$ that match the constraints.
\item Given a prelog curve, find the possible log-structures on it so that they can be deformed to the general fiber.
\end{itemize}
The count of the solutions yields the expected multiplicity $m^\CC_\Gamma=\prod m_V$. It is obtained as follows:
\begin{itemize}[label=-]
\item The complement of the marked points on a curve $\Gamma$ passing through $\P$ is a union of trees, each one possessing a unique unbounded end. Indeed, there are no cycles, otherwise the curve would vary in a one parameter family. Analogously, there are no components with at least two unbounded ends. Thus, there is exactly one unbounded end in each component.
\item When looking for the prelog-curves, their search can be done by pruning the trees in the forest from the previous point. Each vertex provides $\frac{m_V}{w_{V1}w_{V2}}$ solutions, where $w_{VA}$ and $w_{V2}$ are the weights of the adjacent edges not in direction of the unbounded end of the tree.
\item The upgrading from prelog-curves to log-curves gives a $w$ factor for each edge of each tree. In particular edges which are not marked contribute $w$, and edges marked $w^2$.
\end{itemize}
The product yields the expected multiplicity. The case of relative invariants is treated similarly.

\end{proof}

\begin{rem}
In the case of superabundant loops, the element $\alpha$ is torsion in $\mathrm{Pic}(E)$. It is also possible to find a correspondence theorem by choosing a family of line bundles that are also torsion in $\mathrm{Pic}(\E)$. The correspondence works the same way for simple tropical curves. Superabundant loops might be dealt with separately, lifting them as elliptic curve inside $\CC F_{0,\alpha}$.
\end{rem}

\begin{rem}
In some sense, the proof of the correspondence theorem consists in observing that the proof from Nishinou-Siebert \cite{nishinou2006toric} generalizes for non-superabundant tropical curves inside a family of varieties that degenerates to a union of toric varieties glued along their toric divisors, and not only toric ones. This approach is also used in \cite{nishinou2020realization} for curves in abelian surfaces, with the crucial difference that in this case the curves are now all superabundant.
\end{rem}

\subsection{Refined enumeration}

Following the correspondence theorem and the definition of the classical complex multiplicity, we set the following definition.

\begin{defi}
Let $h:\Gamma\to\TT F$ be a simple irreducible parametrized tropical curve. Its \textit{refined multiplicity} is defined as follows:
$$m^q_\Gamma=\prod_V [m_V]_q\in\ZZ[q^{\pm 1/2}],$$
where $[a]_q=\frac{q^{a/2}-q^{-a/2}}{q^{1/2}-q^{-1/2}}$ denotes the $q$-analog of $a$.
\end{defi}

For a generic choice of conditions $\P$ inside $\TT F$, we can now make the count of tropical solutions using instead the refined multiplicity, getting Laurent polynomials
$$BG^{\delta,\mathrm{trop}}_{g,(d_1,d_2)}(\P)=\sum_{h(\Gamma)\supset\P} m^q_\Gamma,$$
$$\text{and }BG^{\delta,\mathrm{trop}}_{g,(d_1,d_2)}(\mu_0,\mu_\infty,\nu_0,\nu_\infty)(\P)=\frac{1}{I_q^{\mu_0}I_q^{\mu_\infty}}\sum_{h(\Gamma)\supset\P} m^q_\Gamma, $$
where $I_q^\mu=\prod [i]_q^{\mu_i}$. As usual, one can count reducible curves, and their count is denoted with a $\bullet$. We have the following invariance theorem.

\begin{theo}\label{theorem refined invariance}
The refined count $BG^{\delta,\mathrm{trop}}_{g,(d_1,d_2)}(\P)$ (resp. $BG^{\delta,\mathrm{trop}}_{g,(d_1,d_2)}(\mu_0,\mu_\infty,\nu_0,\nu_\infty)$) does not depend on the choice of $\P$ as long as it is generic. (resp. with $\bullet$)
\end{theo}

\begin{proof}
The proof reduces to the proof of invariance in \cite{itenberg2013block}. We consider a generic path $\P_s$ between two generic choices of point configuration $\P_0$ and $\P_1$. Given the assumption, there are no superabundant curves that can appear. Moreover, assume only one of the points moves: $p_s\in\P_s$

We know that locally, a deformation of $p_s$ induces a deformation of the solutions, which have the same multiplicity since the latter is constant on the combinatorial types. We know that the complement of the marked points in any of the curves solution to the problem is a union of trees that each possess a unique unbounded end. Otherwise it would be possible to deform the curve in a one parameter family, contradicting the generic choice of the configuration. As the marked point $p_s\in\P_s$ moves, it is possible to deform any curve passing through $\P_s$. We have two possibilities:
\begin{itemize}[label=-]
\item both sides of the edge that $p_s$ splits belong to the same tree,
\item they belong to different components of $\Gamma\backslash h^{-1}(\P_s)$.
\end{itemize}
In the first case, only the unique path between both sides of the cut edge move. This is a cycle in $\Gamma$. Meanwhile in the second case, it is the path from the marked points to both unbounded ends in each tree that move. According to \cite{itenberg2013block}, which uses Proposition 3.9 in \cite{gathmann2007numbers}, the only possibilities that can occur when moving a point are:
\begin{enumerate}[label=(\alph*)]
\item the curve degenerates to a curve that has a quadrivalent vertex
\item the curve degenerates to a curve with a vertex that coincides with a marked point,
\item the curve degenerates to a curve that has two quadrivalent vertices related by a pair of parallel edges.
\end{enumerate}
In each case, the invariance of counts with refined multiplicity comes from \cite{itenberg2013block}.
\end{proof}

\begin{rem}
The only difference with the planar case is that if one of the marked points is the only marked point belonging to a cycle that is not homologically trivial in $\TT F$, then the event (c) cannot occur. In fact, one can also use the cutting process to prove the invariance locally by reducing to the planar case from \cite{itenberg2013block}.
\end{rem}

Only the degree $\delta$ is emphasized in the notation of the invariants. If $\delta\neq 0$, it is because for each tropical elliptic curve $\TT E$ there is a unique choice of surface $\TT F_\delta$, and a scaling argument shows that the invariant does not depend on the choice of $\TT E$ either. If $\delta=0$, the computation using floor diagrams proves that the count does not depend on the choice of $\TT E$ nor $\alpha$ as long as the latter is chosen generically. This can also be done as follows: apply the cutting procedure to get inside $\RR^2$, fix the configuration $\P$, and move the parameter of gluing $\alpha$ instead. The refined count does not depend on the choice of the constraints. The scaling argument takes care of the invariance with respect to $\TT E$.

\begin{rem}
It is possible to extend the result to the case of $\TT F_{0,\alpha}$ where there are superabundant loops, meaning that $\alpha$ is torsion. We have two possibilities:
\begin{itemize}
\item either we treat curves with a superabundant loop as reducible curves, and we need to define a refined multiplicity whose value does not really matter since these curves are deformed separately,
\item or we define the multiplicity by slightly deforming the gluing parameter $\alpha$ and deforming the curves along with it.
\end{itemize}
In the first case we obtain an invariant for each surface $\TT F_{0,\alpha}$ but that does depend on whether $\alpha$ is torsion or not. The second case corresponds to a way of counting curves so that the count does not depend on the choice of the parameter $\alpha$, even if it is torsion.

Concretely, the presence of superabundant loops makes appear reducible curves that would not appear otherwise. It is possible to make them into irreducible tropical curves as follows. Choose a parametrization of the reducible tropical curve by a disconnected graph. Each superabundant loop now corresponds to a circle. Choose an intersection point between the superabundant loop and the rest of the curve. Then, add an edge between the circle and the corresponding point on some other component of the parametrizing graph. The curve is irreducible but has a contracted edge. It is now possible to deform the curve when changing the gluing parameter $\alpha$. The deformation appears on Figure \ref{figure deforming superabundant loop}. The number of possibilities of deformations is exactly described by floor multiplicities.
\end{rem}

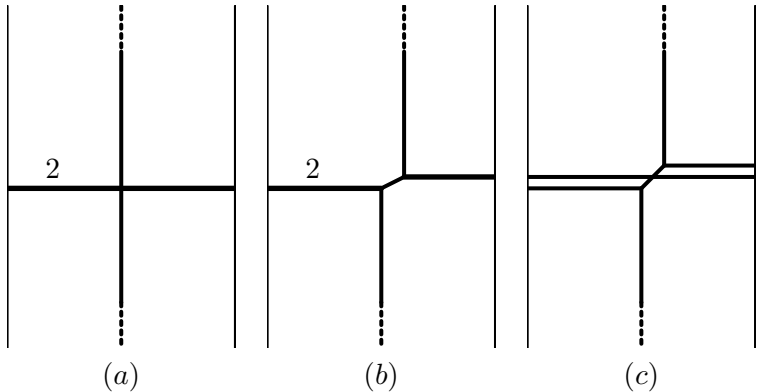
\begin{figure}
\begin{center}
\begin{tabular}{ccc}
\begin{tikzpicture}[line cap=round,line join=round,>=triangle 45,x=0.3cm,y=0.3cm]
\clip(0,0) rectangle (10,15);
\draw [line width=1pt] (0,0)-- (0,15);
\draw [line width=1pt] (10,0)-- (10,15);

\draw [line width=2pt] (0,7)-- (10,7);
\draw [line width=1.5pt] (5,2)-- (5,13);
\draw [line width=1.5pt,dotted] (5,2)-- (5,0);
\draw [line width=1.5pt,dotted] (5,15)-- (5,13);
\draw (2,7) node[above] {$2$};

\begin{scriptsize}

\end{scriptsize}
\end{tikzpicture}
&
\begin{tikzpicture}[line cap=round,line join=round,>=triangle 45,x=0.3cm,y=0.3cm]
\clip(0,0) rectangle (10,15);
\draw [line width=1pt] (0,0)-- (0,15);
\draw [line width=1pt] (10,0)-- (10,15);

\draw [line width=2pt] (0,7)-- (5,7);
\draw [line width=1.5pt] (5,7)-- (6,7.5);
\draw [line width=2pt] (6,7.5)-- (10,7.5);
\draw [line width=1.5pt] (5,2)-- (5,7);
\draw [line width=1.5pt] (6,7.5)-- (6,13);
\draw [line width=1.5pt,dotted] (5,2)-- (5,0);
\draw [line width=1.5pt,dotted] (6,15)-- (6,13);
\draw (2,7) node[above] {$2$};

\begin{scriptsize}

\end{scriptsize}
\end{tikzpicture}
&
\begin{tikzpicture}[line cap=round,line join=round,>=triangle 45,x=0.3cm,y=0.3cm]
\clip(0,0) rectangle (10,15);
\draw [line width=1pt] (0,0)-- (0,15);
\draw [line width=1pt] (10,0)-- (10,15);

\draw [line width=1.5pt] (0,7)-- (5,7);
\draw [line width=1.5pt] (5,7)-- (6,8);
\draw [line width=1.5pt] (6,8)-- (10,8);
\draw [line width=1.5pt] (0,7.5)-- (10,7.5);
\draw [line width=1.5pt] (5,2)-- (5,7);
\draw [line width=1.5pt] (6,8)-- (6,13);
\draw [line width=1.5pt,dotted] (5,2)-- (5,0);
\draw [line width=1.5pt,dotted] (6,15)-- (6,13);

\begin{scriptsize}

\end{scriptsize}
\end{tikzpicture}\\
$(a)$ & $(b)$ & $(c)$ \\
\end{tabular}
\caption{\label{figure deforming superabundant loop} A superabundant loop of weight $2$ and its deformations when changing the gluing parameter $\alpha$.}
\end{center}
\end{figure}

\section{Vertical floor diagrams}

We now intend to provide an algorithm to compute the count of parametrized tropical curves of bidegree $(d_1,d_2)$ in $\TT F$. This gives us the value of the invariants
$$N^{\delta}_{g,(d_1,d_2)} \text{ and }N^{\delta}_{g,(d_1,d_2)}(\mu_0,\mu_\infty,\nu_0,\nu_\infty),$$
as well as
$$BG^{\delta,\mathrm{trop}}_{g,(d_1,d_2)} \text{ and }BG^{\delta,\mathrm{trop}}_{g,(d_1,d_2)}(\mu_0,\mu_\infty,\nu_0,\nu_\infty).$$
We now allow ourselves to remove the ``cpx" and ``trop" from the notations, since Theorem \ref{theorem correspondence} states their equality. We start with the definition of a floor diagram, then we discuss how to construct a floor diagram from a tropical curve before giving the way to count the floor diagrams so that their count matches the previous invariants.

\subsection{Floor diagrams}
\label{subsection floor diagrams}

\begin{defi}
A \textit{floor diagram} $\Dfk$ of bidegree $(d_1,d_2)$ in $\TT F$ is an oriented weighted graph satisfying the following requirements:
\begin{enumerate}[label=(\alph*)]
\item the edges, called \textit{elevators} have a weight $w_e\in\NN$,
\item counted with weight, there are $\delta d_1+d_2$ unbounded elevators oriented inward, and $d_2$ oriented outward,
\item each vertex $\F$, called \textit{floor} has a weight $w_\F\in\NN$,
\item the sum of the weights of the floors is $d_1$,
\item there are no oriented cycles,
\item for a floor $\F$, its divergence $\div\F=\sum_{e\ni\F}\pm w_e$, with the sign of $\pm w_e$ depending on whether $e$ starts or ends at $\F$, is equal to $\delta w_\F$.
\end{enumerate}
The \textit{genus} of $\Dfk$ is by definition its genus as an unoriented graph plus its number of floors. The partitions $\mu_0\vdash \delta d_1+d_2$ and $\mu_\infty\vdash d_2$ induced by the unbounded elevators are called the tangency profile of the floor diagram.
\end{defi}

For a floor $\F$, its valency is equal to the number of elevators adjacent to it and is denoted by $\val(\F)$.

\begin{defi}
Let $\Dfk$ be a floor diagram of bidegree $(d_1,d_2)$ and tangency profile $(\mu_0,\mu_\infty)$. A \textit{marking} of $\Dfk$ is an increasing injection $\mathfrak{m}:[\![ 1;|\mu_0|+|\mu_\infty|+g(\Dfk)-1]\!] \to \Dfk$, where points are sent either to a floor or to an elevator of $\Dfk$, endowed with the partial order induced by the orientation, and such that
\begin{itemize}[label=-]
\item each floor $\F$ is marked by exactly one point,
\item each component of the complement of the markings on the elevators contains a unique unbounded elevator.
\end{itemize}
\end{defi}

\begin{rem}\label{remark marked ends}
To handle solutions of the enumerative problem where we do not only have point constraints but conditions on the ends of the curves, in the previous definition, one should allow the markings of a diagram to be located at ``infinity", \textit{i.e.} the end of the unbounded elevators. We then have to distinguish between \textit{fixed elevators}, which are elevators with a marking at their infinite extremity (corresponding to the $\mu_0$ and $\mu_\infty$ constraints), and \textit{marked elevators}, which are just elevators containing a marked point but not at their extremity.

In order to keep things simple, we keep the main framework of diagrams avoiding constraints at infinity, and put a remark to explain what changes in their presence.
\end{rem}

\begin{expl}
In Figure \ref{figure example marked floor diagram} we can see two examples of marked floor diagrams. As in \cite{brugalle2007enumeration}, we only mark the image of the markings, main difference being that it is not unique anymore since not all edges bear a marked point. The number of markings is then equal to the number of increasing maps to the diagram.

The first two marked diagrams are for $\delta=0$. They have the same underlying diagram but different markings. They are both of genus $2$ since there are two floors and no apparent cycle. The bidegree is $(5,3)$ as the sum of floors weights is $5$ while the sum of weights at top infinity is $3$.

The last two diagrams are for $\delta=1$. We can check that the balancing condition is satisfied at each floor, and that the complement of marked elevators contains a unique unbounded elevator in each of the marking. The genus is still $2$ while the bidegree is $(3,1)$.
\end{expl}

\begin{figure}[h]
\begin{center}
\includegraphics[scale=0.7]{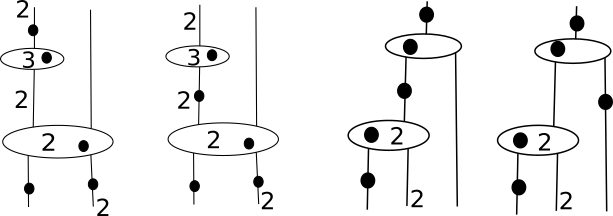}
\caption{\label{figure example marked floor diagram}Four floor diagrams with marking}
\end{center}
\end{figure}

We refer to section \ref{section example computations} for more examples of floor diagrams and markings on it.

\subsection{Floor diagrams from a tropical curve}

We now explain how to get a marked floor diagram $(\Dfk,\mathfrak{m})$ from a parametrized tropical curve $h:\Gamma\to\TT F$ passing through an vertically \textit{stretched} configuration of points. In all what follows, we consider a parametrized tropical curve $h:\Gamma\to\TT F$ of bidegree $(d_1,d_2)$ and fixed tangency profile $(\mu_0,\mu_\infty)$.

\begin{rem}
We only consider tropical curves only passing through point constraints, but everything also works when we add constraints on the position of some unbounded ends of the curves. The difference is that we get marked floor diagrams with fixed elevators instead of only marked elevators.
\end{rem}

\begin{defi}
A generic point configuration $\P$ on $\TT F$ is said to be \textit{stretched} if the difference between the $y$-coordinate of the points taken in the standard chart of $\TT F$ is very large when compared to the length $l$ of the elliptic curve $\TT E$.
\end{defi}

Before getting to the process, we give a few lemmas about the edges of parametrized tropical curves inside $\TT F$.

\begin{lem}\label{lemma bounded slope}
Let $h:\Gamma\to\TT F$ be a parametrized tropical curve of bidegree $(d_1,d_2)$. The slope of an edge can only take a finite number of values.
\end{lem}

\begin{proof}
Let $e$ be a bounded edge of $\Gamma$ with slope $(u,v)$. Assume $u\neq 0$. The intersection index with a fiber $F$ gives intersection number $|u|$, and it cannot be more than the intersection number between $\Gamma$ and $F$. So one has $|u|\leqslant \Gamma\cdot F=d_1$. Assuming $v\neq 0$, we can similarly intersect with a section and get this time $|v|\leqslant \Gamma\cdot E_0=\delta d_1+d_2$.
\end{proof}

\begin{lem}\label{lemma length bounded}
Let $h:\Gamma\to\TT F$ be a parametrized tropical curve of bidegree $(d_1,d_2)$ and let $e$ be a non-vertical edge of $\Gamma$ with slope $(u,v)$. Then there exists $M$ only depending on the bidegree $(d_1,d_2)$ and the length $l$ of $\TT E$ such that one has
$$l(e)\leqslant M.$$
\end{lem}

In other words, Lemma \ref{lemma length bounded} asserts that there is a uniform bound on the length of non-vertical edges for curves of a fixed bidegree in $\TT F$.

\begin{proof}
Let $e$ be a bounded edge of $\Gamma$. If the length of $e$ is too big, it rolls several times around the cylinder and provides too many intersection points with a fiber $F$. In fact, it provides at least $\left\lfloor \frac{l(e)}{l} \right\rfloor$ intersection points, which each contribute $|u|$ to the intersection index. Then, one gets that $\left\lfloor \frac{l(e)}{l} \right\rfloor|u|\leqslant d_1$. Hence, as $|u|\geqslant 1$, it suffices to take $M=l\left(d_1+1\right)$.
\end{proof}

\begin{defi}
A vertical edge of $\Gamma$ is called an \textit{elevator}. A connected component of the complement of vertical edges is called a \textit{floor}.
\end{defi}

We now assume that our configuration of points $\P$ is stretched in the vertical direction, meaning the difference between the heights of two points in the configuration are very large when compared to $l,d_1,d_2$ and $\delta$. We also assume that $\Gamma$ passes through $\P$. We define a diagram $\Dfk(\Gamma)$ from $\Gamma$ as follows, and then show that it is actually a floor diagram:
\begin{itemize}[label=-]
\item The vertices of $\Dfk$ are the floors of $\Gamma$. The weight of a floor $\F$ is its intersection number with a fiber $F$.
\item The edges of $\Dfk$ are the elevators of $\Gamma$, oriented with an increasing $y$-coordinate. The weight of an elevator is its weight as an edge of $\Gamma$, \textit{i.e.} the integral length of the slope of $h$ on it.
\item The floor diagram $\Dfk$ is marked by the marked points of $\P$.
\end{itemize}
If $\Gamma$ has tangency profile $(\mu_0,\mu_\infty)$, so has $\Dfk$. We now have to check the various requirements to ensure that with the marking $\mathfrak{m}$ provided by the marked points $\P$, $(\Dfk,\mathfrak{m})$ is indeed a floor diagram of bidegree $(d_1,d_2)$, tangency profile $(\mu_0,\mu_\infty)$, and of genus $g(\Dfk)=g(\Gamma)$. This is the content of the following proposition.

\begin{prop}
Let $h:\Gamma\rightarrow\TT F$ be a tropical curve of bidegree $(d_1,d_2)$ passing by the stretched configuration $\P$, and let $\Dfk$ be the graph constructed by the preceding construction. Then, one has
\begin{enumerate}[label=$(\roman*)$]
\item the floors weights satisfy $\sum_\F w_\F=d_1$,
\item for each floor $\F$ of $\Gamma$ one has $\div\F=\sum_{e\ni\F}\pm w_e$,
\item each floor contains exactly one marked point,
\item each elevator contains at most one marked point,
\item the floors have genus $1$, and in particular $g(\Dfk)=g(\Gamma)$,
\item each floor is adjacent to at least one unmarked elevator, and components of the complement of the marked elevators contain a unique unbounded elevator.
\end{enumerate}
In particular, $\Dfk$ is a floor diagram, and $\mathfrak{m}$ is a marking on $\Dfk$.
\end{prop}

\begin{proof}
\begin{enumerate}[label=$(\roman*)$]
\item The total intersection number between $\Gamma$ and a generic fiber $F$ is $d_1$, and only the floors contribute to the intersection, since the elevators have also a vertical slope. The result follows.
\item We cut the curve along the strip to get an honest tropical curve inside $\RR^2$. The sum of all the slopes of the outgoing edges is $0$ by the balancing condition. By definition, the elevators adjacent to the floor contribute exactly to $(\div\F)\times (0,-1)$. Meanwhile, let $P$ be one of intersection points with right side of the strip. Let $(u_P,v_P)$ be its slope. The slope of the edge coming on the other side of the strip is $(u_P,v_P-\delta u_P)$ with ingoing orientation. Thus, taking the outgoing orientation, the two slopes add up to $(0,\delta u_P)$. Adding all the contribution, we get $(0,\delta \sum_P u_P)$. Moreover, the intersection index $w_\F$ between $\F$ and a fiber is precisely $\sum_P u_P$. Hence, by the balancing condition, one gets
$$\div\F=\delta w_\F.$$
\item By lemma \ref{lemma bounded slope}, the slope of a curve of bidegree $(d_1,d_2)$ can only take a finite set of values. By lemma \ref{lemma length bounded}, the length of non-vertical edges is bounded by the degree. As there are a finite number of edges, it is not possible to find a path from two points in the stretched configuration without having to pass through an elevator. Thus, a floor contains at most one marked point. It contains at least one otherwise it is possible to translate vertically the floor and get a $1$-parameter family of curves passing through $\P$, contradicting the genericity of $\P$.
\item As the configuration $\P$ is generic, no two marked points have the same projection to $\TT E$, and thus an elevator cannot contain more than one marked point.
\item As in the planar case and proved in \cite{mikhalkin2005enumerative}, the complement of the marked points $\Gamma\backslash h^{-1}(\P)$ is without cycle, and each component contains a unique unbounded end:
	\begin{itemize}[label=-]
	\item As solutions are trivalent, it is possible to deform cycles and paths between two unbounded ends, thus, there are no cycles, and at most one unbounded end per component.
	\item An Euler characteristic computation ensures that there is at least one end per component. Alternatively, a bounded component would give a relation between the marked points at its boundary, resulting in a non generic choice of constraints. Thus, there are at least one end per component.
	\end{itemize}
Thus, a floor cannot be of genus greater than $2$, otherwise, as it contains a unique marked point, the complement of the marked points would still contain some cycle.
\item Following the preceding point, not all the elevators adjacent to a floor can be marked, since it would lead to a bounded component in $\Gamma-h^{-1}(\P)$. The rest follows from the same observation.
\end{enumerate}
\end{proof}

\begin{rem}
The difference between the markings for diagrams on a cylinder and the planar case from \cite{brugalle2007enumeration}\cite{brugalle2008floor} is that the floors from a diagram in a cylinder do not contain any unbounded end from the tropical curve $\Gamma$. Thus, the unbounded end (which is vertical) is reached by a path of elevators.
\end{rem}

In particular, the point $(v)$ can be refined in the following proposition: each floor is a graph whose total space is homeomorphic to a circle. Then, the projection $\pi:\TT F\to \TT E$ restricts on the floor to a cover of $\TT E$.

\begin{prop}
Let $h:\Gamma\to\TT F$ be a genus $g$ parametrized tropical curve of bidegree $(d_1,d_2)$ passing through a stretched configuration of points $\P$. Each floor of $\Gamma$ is homeomorphic to a circle. The degree $w_\F$ of the floor splits as
$$w_\F=d_\F k_\F,$$
where $d_\F$ is the degree of the cover $\pi_\F:\F\to \TT E$, and $k_\F$ is the common horizontal coordinate of slope of $h$ on the edges of $\F$.
\end{prop}

\begin{proof}
As the curve is trivalent, and the floor $\F$ contains a unique cycle, if an edge of $\F$ was not contained in the cycle, it would be separating. In particular, on one side would be a graph without cycle whose vertices are at some point adjacent to elevators. Thus, their slopes would be vertical by the balancing condition. Hence, $\F$ is homeomorphic to a circle and every vertex of $\F$ must be adjacent to some elevator. In particular, the horizontal slope of the edges on $\F$ is constant using to the balancing condition. Let us denote this number by $k_\F >0$. Furthermore, there are $d_\F$ points in $\F$ resulting from the intersection of $\F$ with a fiber $F$. Each of them contributes $k_\F$ to the intersection between $\F$ and $F$. Hence,
$$w_\F=d_\F k_\F.$$
\end{proof}

\begin{rem}
The possible splittings of $w_\F$ as a product $dk$ are partially depicted on Figure \ref{figure deforming superabundant loop} $(b)$ and $(c)$. The integer $d$ is the number of turns that the floor makes, and $k$ the ``speed" at which it makes them.
\end{rem}

\subsection{Multiplicities of floor diagrams}

\subsubsection{Classical multiplicity.} We now define the multiplicity of a marked floor diagram $(\Dfk,\mathfrak{m})$ so that it matches the sum of the multiplicities of the tropical curves $h:\Gamma\to\TT F$ it encodes. The fact that the multiplicity is the right one is proven in Theorem \ref{theorem count floor diagrams}.

\begin{defi}\label{definition multiplicity floor diagram}
Let $(\Dfk,\mathfrak{m})$ be a marked floor diagram. The multiplicity $m(\F)$ of a floor $\F$ is
$$m(\F)=w_\F^{\val(\F)-1}\sigma_1(w_\F),$$
where $\sigma_1(n)=\sum_{d|n}d$ is the sum of divisors. Let $\Dfk^\infty_{m}$ (resp. $\Dfk^\infty_{um}$, $\Dfk^b_{m}$, $\Dfk^b_{um}$) be the set of unbounded elevators which are marked (resp. unbounded unmarked, bounded marked, bounded unmarked). The multiplicity of $(\Dfk,\mathfrak{m})$ is defined as follows:
$$m(\Dfk,\mathfrak{m})=\prod_\F m(\F) \prod_{e\in\Dfk^\infty_{m}}w_e \prod_{e\in\Dfk^b_{m}}w_e^2\prod_{e\in\Dfk^\infty_{um}}w_e^2\prod_{e\in\Dfk^b_{um}}w_e^3,$$
where the first product is over the floors of $\Dfk$.
\end{defi}

\begin{rem}\label{remark multiplicity fixed ends}
In the presence of constraints on the unbounded ends, \textit{i.e.} when considering the relative invariants $N^\delta_{g,(d_1,d_2)}(\mu_0,\mu_\infty,\nu_0,\nu_\infty)$, we get fixed elevators, corresponding to elements in the partition $\mu_0$ or $\mu_\infty$. The diagram multiplicities are handled as in \cite{gathmann2007caporaso}, and we take as diagram multiplicity $\frac{1}{I^{\mu_0+\mu_\infty}}m(\widetilde{\Dfk},\widetilde{\mfk})$, where $(\widetilde{\Dfk},\widetilde{\mfk})$ is the marked diagram where the infinite markings have been replaced by finite marking on their respective elevators, diagram for which the multiplicity is defined above.
\end{rem}

We refer to the proof of Theorem \ref{theorem count floor diagrams} for details but briefly explain the meaning of the different terms in the multiplicity. The floor multiplicity corresponds the various possible shapes of the floor and the product of the vertices on it, as partially depicted on Figure \ref{figure deforming superabundant loop} $(b)$ and $(c)$. In a sense, it corresponds to the computation of some $N^\delta_{1,(w_\F,-)}(\mu_0,\mu_\infty,\nu_0,\nu_\infty)$, where $|\nu_0|+|\nu_\infty|=1$. Choosing a shape of floor amounts to choose a splitting $w_\F=d_\F  k_\F$, and the sum over these possibilities makes appear the $\sigma_1(w_\F)$. Concerning elevators, as each elevator is adjacent to two vertices, it contributes to a factor $w_e^2$ to the multiplicity. The contribution $w_e^3$ of the unmarked elevators comes from the fact that contrarily to the planar case, fixing the elevators adjacent to a vertex and the position of a point is not enough to ensure a unique solution. In fact, there are precisely $w_e$ possible positions, where $w_e$ is the weight of the unmarked elevator in the direction of the unique unmarked unbounded elevator.

\begin{expl}
The multiplicities of the marked diagrams on Figure \ref{figure example marked floor diagram} are respectively:
\begin{itemize}[label=-]
\item $m(\Dfk_1,\mfk_1)=2^3\sigma_1(2)\cdot 3^1\sigma_1(3)\cdot 2\cdot 2^3\cdot 2,$
\item $m(\Dfk_2,\mfk_2)=2^3\sigma_1(2)\cdot 3^1\sigma_1(3)\cdot 2\cdot 2^2\cdot 2^2,$
\item $m(\Dfk_3,\mfk_3)=2^2\sigma_1(2)\cdot 1^2\sigma_1(1)\cdot 2^2,$
\item $m(\Dfk_4,\mfk_4)=2^2\sigma_1(2)\cdot 1^2\sigma_1(1)\cdot 2^2.$
\end{itemize}
The first and second diagram have each two possible markings since we need to choose which of the marking is on the bottom end of weight $2$. The third diagram has only one marking but the last one has three since we need to choose which of the three markings lies on the right unbounded elevator.
\end{expl}

\subsubsection{Refined multiplicity.} The multiplicity giving the refined invariants is more complicated to handle, since its value does not only depend on the valency and weight of a floor, but also on the weights of the adjacent elevators. With the same notations as in Definition \ref{definition multiplicity floor diagram}, we set the following.

\begin{defi}\label{definition refined multiplicity floor diagram}
Let $(\Dfk,\mathfrak{m})$ be a marked floor diagram. Let $\F$ be a floor, and $(u_i)_{1\leqslant i\leqslant N}$ the weights of the adjacent elevators. The refined multiplicity $m^q(\F)$ of a floor $\F$ is
$$m^q(\F)=\sum_{kd=w_\F} d^{N-1}\prod_{i=1}^N \frac{[ku_i]_q}{[u_i]_q}= \sum_{kd=w_\F} d^{N-1}\prod_{i=1}^N \frac{q^{ku_i/2}-q^{-ku_i/2}}{q^{u_i/2}-q^{-u_i/2}}.$$
The refined multiplicity of $(\Dfk,\mathfrak{m})$ is defined as follows:
$$m^q(\Dfk,\mathfrak{m})=\prod_\F m^q(\F) \prod_{e\in\Dfk^\infty_{m}}[w_e]_q \prod_{e\in\Dfk^b_{m}}[w_e]_q^2\prod_{e\in\Dfk^\infty_{um}}w_e[w_e]_q\prod_{e\in\Dfk^b_{um}}w_e[w_e]_q^2,$$
where the first product is over the floors of $\Dfk$.
\end{defi}

\begin{rem}
The definition is adapted to the setting of fixed elevators by dividing by $I_q^{\mu_0+\mu_\infty}$.
\end{rem}

\subsection{Statement}

We now have the main result of the section, asserting that the count of marked floor diagrams coincides with the number of tropical curves.

\begin{theo}\label{theorem count floor diagrams}
One has
$$\sum_{(\Dfk,\mathfrak{m})}m(\Dfk,\mathfrak{m})=N^{\delta}_{g,(d_1,d_2)} \text{ and } \sum_{(\Dfk,\mathfrak{m})}m(\Dfk,\mathfrak{m})=N^\delta_{g,(d_1,d_2)}(\mu_0,\mu_\infty,\nu_0,\nu_\infty) ,$$
$$\sum_{(\Dfk,\mathfrak{m})}m^q(\Dfk,\mathfrak{m})=BG^{\delta,\mathrm{trop}}_{g,(d_1,d_2)} \text{ and } \sum_{(\Dfk,\mathfrak{m})}m^q(\Dfk,\mathfrak{m})=BG^{\delta,\mathrm{trop}}_{g,(d_1,d_2)}(\mu_0,\mu_\infty,\nu_0,\nu_\infty) ,$$
where the first sum is over the marked floor diagrams of genus $g$ and bidegree $(d_1,d_2)$, and the second sum over the marked floor diagrams of genus $g$, bidegree $(d_1,d_2)$  and ramification profile $(\mu_0+\nu_0,\mu_\infty+\nu_\infty)$, with elevators corresponding to $\mu_0$ and $\mu_\infty$ fixed.
\end{theo}

\begin{proof}
We already proved the following: if the point configuration $\P$ is stretched, all the sought parametrized tropical curves passing through $\P$ admit a floor decomposition. Thus, we need to prove that the multiplicity of a marked floor diagram $(\Dfk,\mathfrak{m})$ matches the sum of the multiplicities of the tropical curves it encodes.

\medskip

To recover a tropical curve from a floor diagram, one just needs to recover the floors and the position of the unmarked elevators. This is done as follows. By assumption, the complement of marked elevators is a forest of trees that each contain a unique unbounded elevator. For a floor $\F$, assuming all the adjacent elevators are marked, we show below that this is indeed sufficient to impose a finite number of shape of the floor $\F$, and fix the position of the remaining elevator $e$ up to $w_e$-torsion. We then proceed by induction, pruning the trees of the forest.

\medskip

Let $\F$ be a floor all of whose adjacent elevators but one bear a marking. Let $(u_i)$ be the collection of weights of the adjacent elevators (with a sign according to whether the elevator goes up or down), $x_i$ their position in $\TT E$, with the first elevator being the unmarked one. The floor $\F$ is a circle with some adjacent elevators, and it might go around the cylinder several times. Thus, one has to consider all the possible splittings of $w_\F$ as a product $w_\F=d\cdot k$, where $d$ is the degree of the topological cover $\pi|_\F:\F\to \TT E$, and $k$ the common horizontal coordinate of the slopes of the edges on $\F$.
\begin{itemize}[label=-]
\item First, assume that $d=1$. We consider the chart given by the strip $[0;l]\times\RR$. Then the floor that we are looking for is the tropical graph (\textit{i.e.} the usual graph inside $\TT F$ plus vertical ends to make it balanced at corners of the graph) of some piecewise affine function on $[0;l]$ satisfying the gluing condition on the boundary of the strip leading to $\TT F$.

\begin{rem}
The problem amounts to find a meromorphic sections of a line bundle that have prescribed poles and zeros, and one non-fixed pole/zero $x_0$ which should have multiplicity $u_1$.
\end{rem}

The condition for the existence of such a piecewise affine function is that the divisor $\sum_1^N u_ix_i$ is equivalent to the divisor of the line bundle defining $\TT F$ inside $\mathrm{Pic}(\TT E)\simeq \ZZ\times \TT E$: the degree ensures that the change of slope between $0$ and $l$ is the right one, and the component in $\TT E$ ensures that points on both sides of the strip are glued together. The degree condition is already ensured by the divergence condition in the diagram. The remaining constraint determines uniquely $u_1x_1$:
$$u_1x_1+\sum_2^N u_k x_k=\left\{ \begin{array}{l}
w_\F\alpha \text{ if }\delta=0 \\
0 \text{ if }\delta\neq 0 \\
\end{array}\right. \in \TT E\simeq\RR/l\ZZ.$$
Thus, it gives $u_1$ possibilities for the position $x_1$ of the unmarked elevator, translating by the $u_1$-torsion elements. This accounts for one of the $u_1$ factor in the formula. (The other ones come from the vertex multiplicity.) For each possible choice of $x_1$, the choice of piecewise affine function is unique up to a constant shift because the difference of two such functions gives an affine function on $\TT E$, which is necessarily constant. The point condition on the floor fixes the constant.

\item If $d\neq 1$, we unfold the floor to the $d$-fold cover of $\TT E$. We thus take a chart given by the strip $[0;dl]\times\RR$. For each elevator, one has to choose a lift (no pun intended !), there are $d^{\val(\F)-1}$ possibilities for the lifts of the marked elevators. We now look for a piecewise affine function on $[0;dl]$. The same reasoning gives $u_1$ solutions for the position of $x_1$ inside $[0;dl]$. We then project the solutions back to $\TT E$ and get floors that make $d$ rounds around the cylinder. However, one has to divide by the $d$ deck transformations of the cover. However, as there are now $d$ possible shifts to make the floor pass through the fixed point, we have to multiply again by $d$, cancelling the previous division. Notice that the position of the unmarked elevator is fixed in the cover, we do not and cannot choose its lift under the covering map since the relation $\sum u_i x_i\equiv\ast$ has to be satisfied in the $d$-fold cover of $\TT E$.
\end{itemize}
The vertex multiplicity of every vertex on $\F$ is equal to $u_i k$, where $u_i$ is the weight of the adjacent elevator. Thus, we get the following contribution:
\begin{equation}\label{equation floor multiplicity}
u_1\sum_{dk=w_\F} d^{\val(\F)-1}\left(\prod_1^N ku_i\right)=u_1\left(\prod_1^N u_i\right)w_\F^{\val(\F)-1}\sigma_1(w_\F),
\end{equation}
where $\sigma_1(n)=\sum_{d|n}d$. We then make the product over all the floors to reconstruct all the tropical curves. Finally, we get:
\begin{itemize}[label=-]
\item the product of all the $m(\F)$,
\item a marked bounded edge $e$ is adjacent to two floors, thus we multiply by $w_e^2$.
\item an unmarked bounded edge is adjacent to two floors, hence a factor $w_e^2$ coming from the multiplicities of the adjacent vertices, and as its position is only determined up to $w_e$-torsion, an additional factor $w_e$, as present in equation (\ref{equation floor multiplicity}).
\item a marked unbounded end is adjacent to only one floor, hence a contribution $w_e$.
\item an unmarked unbounded end is adjacent to only one floor, and its position is fixed up to $w_e$-torsion, hence a contribution of $w_e^2$.
\end{itemize}
We thus get the announced multiplicity. In the refined case, equation (\ref{equation floor multiplicity}) is replaced by
$$u_1\sum_{kd=w_\F} d^{\val(\F)-1}\prod_i[ku_i]_q= u_1\prod_i [u_i]_q\sum_{dk=w_\F} d^{\val(\F)-1}\prod_i \frac{[ku_i]_q}{[u_i]_q}=u_1m^q(\F)\prod_i [u_i]_q .$$
Making the product of all contribution, we get the multiplicity from Definition \ref{definition refined multiplicity floor diagram}.
\end{proof}

\begin{rem}
The multiplicity of a floor $\F$ is derived from the solution to the following special instance of our enumerative problem: the genus is $1$, all but  one ends are fixed, and we have a unique additional point condition.
\end{rem}

As in \cite{arroyo2011recursive}, the floor diagram decomposition can be seen as the iterated application of some Caporaso-Harris type formula for the relative invariants $N^\bullet_{g,(d_1,d_2)}(\mu_0,\mu_\infty,\nu_0,\nu_\infty)$. In our case, the formulas for classical and refined invariants are as follows. Despite their massive expression, their proof is not that difficult and relies on the existence of floor diagrams. Their shape becomes clearer with the Fock space approach which is the content of section \ref{section fock}. If $\mu\geqslant\nu$ are two partitions, then $\bino{\mu}{\nu}$ is by definition $\prod_i \bino{\mu_i}{\nu_i}$. Recall that $I^\mu=\prod_i i^{\mu_i}$, and $I_q^\mu=\prod_i [i]_q^{\mu_i}$.

\begin{theo}\label{theorem caporaso harris formula}
We have the following recursive formulas:
\begin{align*}
 & N^{\delta,\bullet}_{g,(d_1,d_2)}(\mu_0,\mu_\infty,\nu_0,\nu_\infty) \\
= & \sum_{k:\nu_{0k}>0} k N^{\delta,\bullet}_{g,(d_1,d_2)}(\mu_0+e_k,\mu_\infty,\nu_0-e_k,\nu_\infty)  \\
 & + \sum \bino{\mu_0}{\mu'_0} \bino{\nu'_0}{\nu_0-e_k} I^{\nu'_0-\nu_0}s^{|\mu'_0-\mu_0|+|\nu'_0-\nu_0|}\sigma_1(s)k^2  N^{\delta,\bullet}_{g',(d_1-s,d_2)}(\mu'_0,\mu_\infty,\nu'_0,\nu_\infty) \\
 & + \sum \bino{\mu_0}{\mu'_0}\bino{\nu'_0}{\nu_0} I^{\nu'_0-\nu_0} s^{|\mu_0-\mu'_0|+|\nu'_0-\nu_0|}\sigma_1(s) k^2 N^{\delta,\bullet}_{g',(d_1-s,d_2)}(\mu'_0+e_k,\mu_\infty,\nu'_0,\nu_\infty), \\
\end{align*}
and
\begin{align*}
& BG^{\delta,\bullet}_{g,(d_1,d_2)}(\mu_0,\mu_\infty,\nu_0,\nu_\infty) \\
= & \sum_{k:\nu_{0k}>0} [k]_q BG^{\delta,\bullet}_{g,(d_1,d_2)}(\mu_0+e_k,\mu_\infty,\nu_0-e_k,\nu_\infty)  \\
 & + \sum  \bino{\mu_0}{\mu'_0}\bino{\nu'_0}{\nu_0-e_k} I_q^{\nu'_0-\nu_0}\left(\sum_{dl=s}d^{|\mu'_0-\mu_0|+|\nu'_0-\nu_0|}\frac{[lk]_q}{[k]_q}\right. \\
 & \left.\prod_i \left(\frac{[il]_q}{[i]_q}\right)^{\mu_{0i}-\mu'_{0i}+\nu'_{0i}-\nu_{0i}}\right)k[k]_q BG^{\delta,\bullet}_{g',(d_1-s,d_2)}(\mu'_0,\mu_\infty,\nu'_0,\nu_\infty) \\
 & + \sum \bino{\mu_0}{\mu'_0}\bino{\nu'_0}{\nu_0} I_q^{\nu'_0-\nu_0}\left(\sum_{dl=s}d^{|\mu'_0-\mu_0|+|\nu'_0-\nu_0|}\frac{[lk]_q}{[k]_q}\right. \\
 & \left.\prod_i \left(\frac{[il]_q}{[i]_q}\right)^{\mu_{0i}-\mu'_{0i}+\nu'_{0i}-\nu_{0i}}\right)k[k]_q BG^{\delta,\bullet}_{g',(d_1-s,d_2)}(\mu'_0+e_k,\mu_\infty,\nu'_0,\nu_\infty), \\
\end{align*}
where the second and third sums are over $g',s,k,\mu'_0,\nu'_0$ such that respectively
$$\begin{array}{lcl}
\text{second sum:} & \text{ and } & \text{third sum:} \\
\mu'_0\leqslant \mu_0 & & \mu'_0\leqslant \mu_0 \\
\nu'_0\geqslant \nu_0-e_k & & \nu'_0\geqslant \nu_0 \\
\|\nu'_0\|+\|\mu'_0\|=\delta(d_1-s)+d_2 & & \|\nu'_0\|+\|\mu'_0\|+k=\delta(d_1-s)+d_2 \\
g-g'=|\nu'_0-\nu_0| & & g-g'=|\nu'_0-\nu_0|+1 \\
\end{array}$$
\end{theo}

\begin{proof}
The proof is similar to the tropical proof of the Caporaso-Harris formula in \cite{gathmann2007caporaso}. The three sums emphasize the three possible events when taking one of the $|\nu_0|+|\nu_\infty|+g-1$ marked points and moving it far down:
\begin{itemize}[label=$\circ$]
\item The marked point goes on an elevator, and we then count the same curves but with a bottom elevator that is now fixed. The factor $k$ is present due to the fact that for relative invariant, tropical curves are counted with a multiplicity $\frac{1}{I^{\mu_0+\mu_\infty}}m_\Gamma$.
\item The marked points goes on a floor, of degree denoted by $s$, and exactly one unfixed bottom elevator arrives on the floor. We get terms in the second sum:
	\begin{itemize}[label=-]
	\item The $\mu_0$ fixed ends are split between those that do not arrive on the floor and those who do, respectively $\mu'_0$ and $\mu_0-\mu'_0$.
	\item One unfixed elevator of weight $k$ arrives on the floor.
	\item Including the unfixed ends leaving the floor, there are $\nu'_0\geqslant \nu_0-e_k$ unfixed ends. 
	\end{itemize}
The other terms are here to account for the different shapes of the floors, and the multiplicities of the vertices on it. The binomial coefficients are here to account for the symmetries between the different unbounded ends.
\item Last, the marked point still goes on a floor, but only fixed bottom elevators arrive on it. Then, one unmarked elevator needs to leave the floor, and the other are free. The position of the free elevators is determined by the other. We get the last sum with a similar bookkeeping.
\end{itemize}
\end{proof}

%figure des notations ?

\begin{rem}
The recursive formula from Theorem \ref{theorem caporaso harris formula} along with some initialization for $d_1=1$ and $g=1$, or $g=0$ and $d_1=0$ allows one to compute all relative invariants.
\end{rem}

\begin{rem}
It is also possible to get a Caporaso-Harris like formula by taking a point far up. Concretely, it means starting the floor diagram by the top floor rather than by the first floor.
\end{rem}

\section{Examples and applications}

We now give several examples of computation of invariants $N^\delta_{g,(d_1,d_2)}(\mu_0,\mu_\infty,\nu_0,\nu_\infty)$. We also give a polynomiality result about some of their generating functions, and quasi-modularity for another.

\subsection{Concrete computations}
\label{section example computations}

\begin{expl}
We compute $N^\delta_{1,(d_1,d_2)}$ for any bidegree. As the genus is $1$, there is only one floor diagram that contributes: a floor diagram with more floors would be of genus at least $2$. However, if $d_2\geqslant 1$, there are two possible markings, since the non-marked end can be above or under the floor. Their multiplicities are both $d_1^{\delta d_1+2d_2-1}\sigma_1(d_1)$, giving
$$N^\delta_{1,(d_1,d_2)}=2d_1^{\delta d_1+2d_2-1}\sigma_1(d_1).$$
If $d_2=0$, there is no elevators above the floor, and one thus has
$$N^\delta_{1,(d_1,d_2)}=d_1^{\delta d_1-1}\sigma_1(d_1).$$
\end{expl}

\begin{figure}[h]
\begin{center}
\includegraphics[scale=0.5]{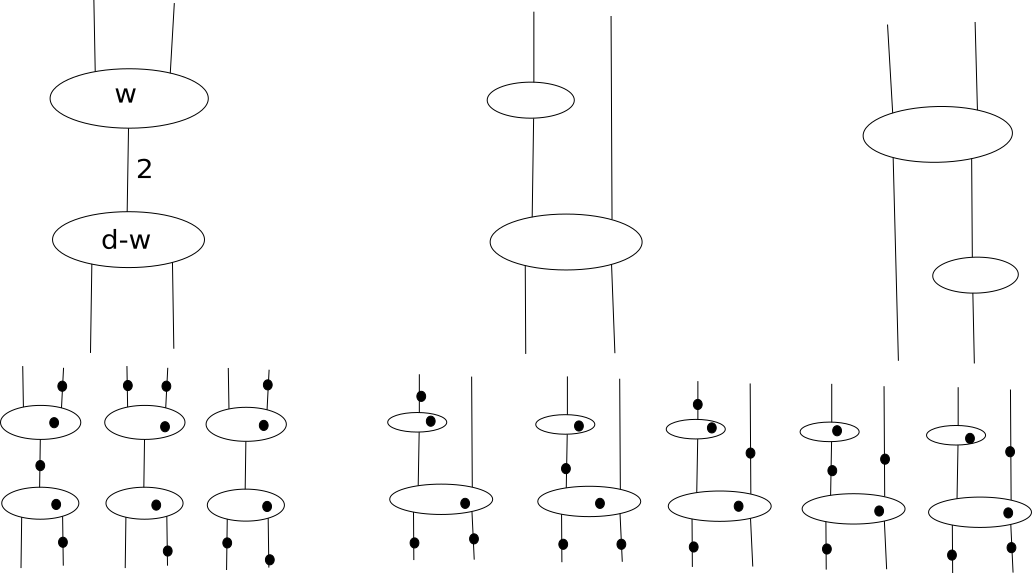}
\caption{\label{figure fd n2d2} Floor diagrams of genus $2$ and bidegree $(d,2)$ in $\TT F_0$ with their markings.}
\end{center}
\end{figure}

\begin{expl}
We take $\delta=0$, $g=2$ and compute $N^0_{2,(d,2)}$ for any $d$. The floor diagrams $\Dfk$ we are looking for are depicted on Figure \ref{figure fd n2d2}. There are three possible shapes. For each shape, one has to look for all the possible splittings of the degree $d$ among the two floors of the curve: $w$ and $d-w$. Finally, one has to account for the different labellings $\mathfrak{m}$ of the floor diagrams. There are $n=5$ points for marking $\Dfk$. As both floors are marked, exactly two elevators have no marked point. Taking into account the fact that each component of the complement of the marked elevators possesses a unique unbounded end, this allows one to make the list of all possibilities depicted on Figure \ref{figure fd n2d2}.
\begin{itemize}[label=$\circ$]
\item The first diagram possesses three possible markings. In exactly two of them, the unique bounded elevator of weight $2$ is unmarked. They give the following contribution:
$$w^2\sigma_1(w)(d-w)^2\sigma_1(d-w)\left( 2^2+2^3+2^3\right).$$
\item For the second diagram on Figure \ref{figure fd n2d2}, the possible images of the markings are drawn under it. There are five possibilities. Then, one has to count the increasing functions to the each of the image, as for usual floor diagrams in \cite{brugalle2007enumeration}. This yields the following contribution:
$$w\sigma_1(w)(d-w)^3\sigma_1(w)\left( 1+1+3+3+2 \right).$$
\item The last floor diagram is the symmetric of the second one and yields the same contribution.
\end{itemize}
Finally, we get that
$$N^0_{2,(d,2)}=\sum_{w=1}^{d-1} 20 w^2\sigma_1(w)(d-w)^2\sigma_1(d-w) + 20w\sigma_1(w)(d-w)^3\sigma_1(w).$$
\end{expl}

\begin{figure}[h]
\begin{center}
\includegraphics[scale=0.5]{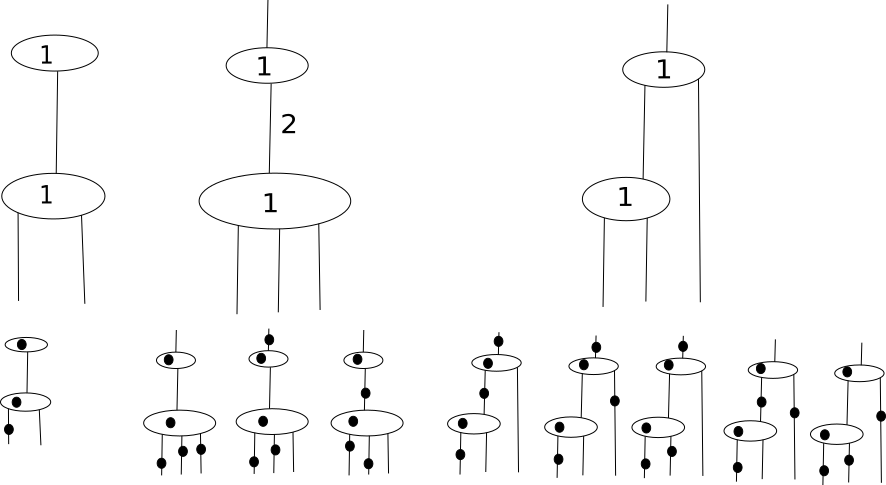}
\caption{\label{figure fd 1N220N221} Floor diagrams of genus $2$ and bidegree $(2,0)$ and $(2,1)$ in $\TT F_1$ with their markings.}
\end{center}
\end{figure}

\begin{expl}
We now assume $\delta=1$, $g=2$ and compute $N^1_{2,(2,0)}$ and $N^1_{2,(2,1)}$. For $N^1_{2,(2,0)}$, there is just one floor diagram of genus $2$, on the left of Figure \ref{figure fd 1N220N221}, and it possesses exactly one marking, depicted under it. It yields
$$N^1_{2,(2,0)}=1.$$
For $N^1_{2,(2,1)}$, there are two floor diagrams of genus $2$ and bidegree $(2,1)$. They are drawn on the right of Figure \ref{figure fd 1N220N221}. There are $n=5$ points to mark the diagram. Thus, exactly two elevators are unmarked.
\begin{itemize}[label=$\circ$]
\item The first diagram (middle of Figure \ref{figure fd 1N220N221}) has three possibles markings. The edge of weight $2$ is unmarked for the first two, and marked for the last. They give the following contribution:
$$2^3+2^3+2^2.$$
\item Similarly to the previous example, we depict the possible markings of the second floor diagram (right of Figure \ref{figure fd 1N220N221}) according to their image. When counting them with the number of increasing functions to the fixed image, they yield the following contribution:
$$1+3+1+4+4.$$
\end{itemize}
We thus get
$$N^1_{2,(2,1)}=33.$$
\end{expl}

\begin{figure}[h]
\begin{center}
\includegraphics[scale=0.5]{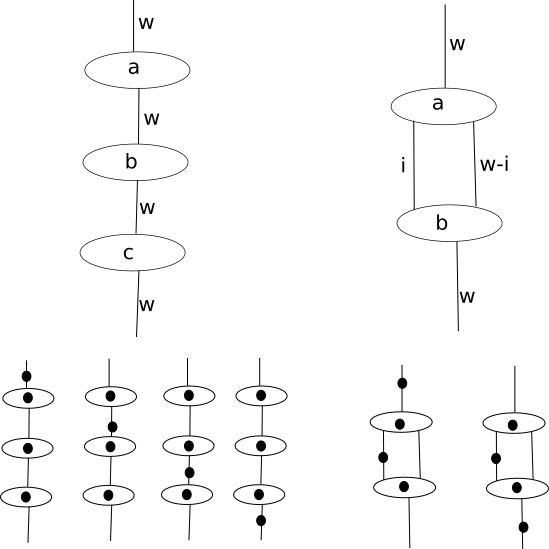}
\caption{\label{figure fd genus 3} Floor diagrams used for the computation of $N^0_{3,(d,w)}(0,0,w^1,w^1)$.}
\end{center}
\end{figure}

\begin{expl}
Back with $\delta=0$, we now take $g=3$ and compute the relative invariant
$$N_{3,(3,w)}(0,0,w^1,w^1).$$
There are now two floor diagrams of genus $3$. One has three floors, and one has two, and they are depicted on Figure \ref{figure fd genus 3}.
\begin{itemize}[label=$\circ$]
\item The first one possesses four different markings, which are each of multiplicity $w^{9}$. In this case, we have on Figure \ref{figure fd genus 3} $a=b=c=1$. If the bidegree was $(d,w)$ rather than $(3,w)$, one would have to multiply by the floors multiplicities $abc\sigma_1(a)\sigma_1(b)\sigma_1(c)$ for all the splittings $a+b+c=d$.
\item The second floor diagram possesses two possible markings, and additionally to the splitting of $3=a+b$ among the two floors, one has to account for the splitting of $w$ among the two bounded edges between the two floors. They yield the following contribution:
$$12w^3\sum_{i=1}^{w-1} i^2(w-i)^3  ,$$
as the floors contribute $2^2 1^2\sigma_1(2)\sigma_1(1)+2^2 1^2\sigma_1(2)\sigma_1(1)=12$. If the bidegree was rather $(d,w)$, one would make the sum over the splittings $d=a+b$, with floors multiplicities $a^2b^2\sigma_1(a)\sigma_1(b)$.
\end{itemize}
In total, we get
$$N^0_{3,(3,w)}(0,0,w^1,w^1)=4w^9+12w^3\sum_{i=1}^{w-1} i^2(w-i)^3.$$
Notice that as for any polynomial $P$, $\sum_{k=1}^nP(k)$ is still a polynomial in $n$, the result is in fact a polynomial in $w$. This fact is generalized in the next subsection.
\end{expl}

\begin{rem}
Notice that in each case, the number of floors is bounded by the genus, since each floor contributes $1$ to the genus, and by the first coordinate of the bidegree, since each floor is of degree at least $1$.
\end{rem}

\subsection{Polynomiality of relative invariants}

We now consider partitions $\nu_0$ and $\nu_\infty$ of fixed length $l_0$ and $l_\infty$, but of arbitrary size, provided that $\|\nu_0\|-\|\nu_\infty\|=\delta d_1$ for some fixed $d_1$. Let $d_2=\|\nu_\infty\|$. We look for curves of genus $g$ of bidegree $(d_1,d_2)$ and tangency profile $(\nu_0,\nu_\infty)$ passing through $l_0+l_\infty+g-1$ points. For the following theorem, we consider a partition $\nu$ of fixed length $l$ as a tuple $(\nu_1,\dots,\nu_l)$.

\begin{theo}\label{theorem polynomiality}
For any fixed $g$ and $d_1$, the function $(\nu_0,\nu_\infty)\mapsto N^\delta_{g,(d_1,d_2)}(0,0,\nu_0,\nu_\infty)$ is a piecewise polynomial function in the variables $\nu_{0,1},\dots,\nu_{0,l_0},\nu_{\infty, 1},\dots,\nu_{\infty, l_\infty-1}$.
\end{theo}

\begin{rem}
The last variable $\nu_{\infty, l_\infty}$ does not appear since the variables are constrained by the relation $\|\nu_0\|-\|\nu_\infty\|=\delta d_1$.
\end{rem}

\begin{proof}
As the genus $g$, number of marked points $l_0+l_\infty+g-1$ and total weight of the floors $d_1$ are fixed, up to the choice of the weight of the edges, there are a finite number of floor diagrams. Thus, we only need to show that spreading the weight among the edges leads to a polynomial for each marked floor diagram $(\Dfk,\mathfrak{m})$.

Let $\Dfk_1$ be a minimal subset of the set of bounded edges of a marked floor diagram $(\Dfk,\mathfrak{m})$ such that the complement is simply connected. Thus, if one fixes the weight $w_e$ of an edge $e\in\Dfk_1$, as the weights of the unbounded ends are prescribed by $\nu_0$ and $\nu_\infty$, there is a unique way to complete the weights of the other edges with numbers (eventually negative) so that the condition $\div\F=\delta w_\F$ is satisfied.

For each choice of weights $w_e$ on the marked edges, the multiplicity given by the formula of Definition \ref{definition multiplicity floor diagram} is a polynomial in the variables. Then, we need to make a sum over all the possible values of the weights $w_e$ of the marked edges. As for any polynomial $P$ the function $\sum_0^n P(k)$ is still a polynomial in $n$, the result follows by induction.

Finally, the fact that the weights need to be positive integers impose some linear conditions on the variables, and if some weights coincide, we may have to divide by the automorphism group of the floor diagram. Thus, each floor diagram does not always contribute solutions, and we only get a piecewise polynomial function.
\end{proof}

\begin{rem}
It is straightforward to extend the result to the case where some of the ends of the tropical curves are fixed but we restrict to the case where they are not to keep relatively easy notations.
\end{rem}

\subsection{Quasi-modularity}

We finish this section by proving the quasi-modularity of the following generating series: let $g$, $d_2$ be fixed integers and $\mu_0$, $\mu_\infty$, $\nu_0$, $\nu_\infty$ be fixed partitions with $\|\mu_0+\nu_0\|=\|\mu_\infty+\nu_\infty\|=d_2$, we consider the generating series for invariants in $\TT F_{0,\alpha}$:
$$F_{g,d_2}(\mu_0,\mu_\infty,\nu_0,\nu_\infty)=\sum_1^\infty N_{g,(d_1,d_2)}^{0}(\mu_0,\mu_\infty,\nu_0,\nu_\infty)y^{d_1}.$$
Quasi-modular functions are some generating functions that generalize modular functions, which are holomorphic functions on the Poincar\'e half-plane that have a nice behavior under the action of the modular group $PSL_2(\ZZ)$ acting by homography. We use the following characterization \cite{kaneko1995generalized}: a function is quasi-modular if and only if it expresses as a polynomial in the Eisenstein series $G_2,G_4,G_6$, where up to a linear transformation, $G_{2k}(y)=\sum_{n=0}^\infty \left(\sum_{d|n}d^k\right) y^n$. We also use the fact that quasi-modular functions are stable by the differential operator $D=y\frac{\mathrm{d}}{\mathrm{d}y}$.

\begin{theo}\label{theorem quasimodularity}
The functions $F_{g,d_2}(\mu_0,\mu_\infty,\nu_0,\nu_\infty)(y)$ are quasi-modular forms.
\end{theo}

\begin{rem}
Theorem \ref{theorem quasimodularity} extends the result from \cite{bohm2020counts} for generating series of relative invariants.
\end{rem}

\begin{proof}
As $d_2$, $g$ and the ramification profile are fixed, up to the choice of the weights of the floors, there are a finite number of floor diagrams. Thus, to show that the generating series are quasi-modular, we only need to show that the generating series for the multiplicities of each floor diagram is a quasi-modular form. For a floor diagram $\Dfk$, the generating series for the various weights that we can put on the floors is as follows:
$$W\sum_{a_1,\cdots,a_p} \left(\prod_1^p a_i^{\val\F_i-1}\sigma_1(a_i)\right)y^{\sum a_i},$$
where $W=\prod_{e\in\Dfk^\infty_{m}}w_e \prod_{e\in\Dfk^b_{m}}w_e^2\prod_{e\in\Dfk^\infty_{um}}w_e^2\prod_{e\in\Dfk^b_{um}}w_e^3$ is the contribution of the elevators, and the floors of $\Dfk$ are labeled $\F_1,\dots,\F_p$, having respective weights $a_1,\cdots,a_p$, so that $\sum a_i=d_1$. This factors as
$$W\prod_{i=1}^p \left( \sum_1^\infty a^{\val\F_i-1}\sigma_1(a)y^a\right).$$
Each sum under the product can be recognized as $D^{\val\F_i-1}G_2(y)$, where $G_2(y)=\sum_1^\infty \sigma_1(n)y^n$ and $D=y\frac{\mathrm{d}}{\mathrm{d}y}$. Hence, the generating series for each floor diagram is a quasi-modular form. Thus, so are the generating series for invariants.
\end{proof}

Quasi-modularity is a desirable property because it allows for a control of the asymptotics of the coefficients. Moreover, as they express as a polynomial in the Eisensetin series, it means that the knowledge of the degree of this polynomial and a finite number of values is sufficient to recover all values of the invariants.

\section{The Fock space approach}
\label{section fock}

The goal of this section is to use the floor diagram approach to relate the invariants considered in this paper to some coefficients of operators on a Fock space. This approach was pioneered by Y. Cooper and R. Pandharipande \cite{cooper2017fock} to express Severi degrees of $\PP^1\times\PP^1$ and some double Hurwitz number. The approach was later generalized by F. Block and L. G\"ottsche \cite{block2016fock} to both fit the setting of their refined invariants, and to apply for $h$-transverse polygons, which are basically the polygons for which the floor diagram approach works for the Severi degrees. In \cite{cavalieri2021counting}, the authors use floor diagrams to express descendant invariants of Hirzebruch surfaces in terms of operators on a Fock space, and one point relative descendant invariants.

\medskip

We describe two settings: one for the classical invariants $N^{\delta,\bullet}_{g,(d_1,d_2)}$, and one for the refined invariants $BG^{\delta,\bullet}_{g,(d_1,d_2)}$.

\medskip

We consider the Heisenberg algebra $\H$, modeled on the hyperbolic lattice, generated over $\QQ$ by $a_n,b_n$ for $n\in\ZZ$ and the following relations:
$$[a_n,a_m]=[b_n,b_m]=0 \text{ and }[a_n,b_m]=n\delta_{n+m,0}.$$
Its deformation $\H^q$ is the $\QQ[q^{\pm 1/2}]$-algebra generated by the same generators, but with the relations
$$[a_n,a_m]=[b_n,b_m]=0 \text{ and }[a_n,b_m]=[n]_q\delta_{n+m,0}.$$
The operators $a_{-n},b_{-n}$ are called \textit{creation operators} and $a_n,b_n$ are called \textit{annihilation operators}. By definition, $a_0=b_0=0$. For a partition $\mu$, let
$$a_\mu=\prod_i \frac{(a_i)^{\mu_i}}{\mu_i!},  a_{-\mu}=\prod_i \frac{(a_{-i})^{\mu_i}}{\mu_i!}, b_\mu=\prod_i \frac{(b_i)^{\mu_i}}{\mu_i!}, b_{-\mu}=\prod_i \frac{(b_{-i})^{\mu_i}}{\mu_i!}\in\H (\text{resp. }\H^q).$$
The \textit{Fock space} $\F$ (resp. $\F^q$) is the free $\QQ$-vector space (resp. $\QQ[q^{\pm 1/2}]$-vector space) generated by the creation operators acting on the vacuum $v_\emptyset$. In other words, by definition for any $n>0$ one has $a_nv_\emptyset=b_nv_\emptyset=0$, and $\F$ has a basis indexed by double partitions: let $|\mu,\nu\rangle =a_{-\mu}b_{-\nu}v_\emptyset$. The vectors $|\mu,\nu\rangle$ form a basis of $\F$. The condition $\langle v_\emptyset\mid v_\emptyset\rangle=1$, and $a_n$ (resp. $b_n$) being adjoint to $a_{-n}$ (resp. $b_{-n}$) defines a scalar product for which
$$\langle \mu,\nu \mid \mu',\nu'\rangle = \left\{ \begin{array}{l}
 \frac{I^\mu}{\mu!}\frac{I^\nu}{\nu!}\delta_{\nu,\mu'} \text{ in }\F, \\
  \frac{I_q^\mu}{\mu!}\frac{I_q^\nu}{\nu!}\delta_{\nu,\mu'} \text{ in }\F^q.\\
  \end{array} \right.$$
The Fock space $\F$ (resp. $\F^q$) is graded in the following way: $\bigoplus_n \F_n$, where $\F_n$ is generated by $|\mu,\nu\rangle$ for $\|\mu\|+\|\nu\|=n$. Then $\H$ (resp. $\H^q$) becomes a graded algebra: $\H=\bigoplus_{n\in\ZZ}\H_n$, and elements $a_{-n},b_{-n}$ belong to $\H_n$.

\medskip

We now introduce a family of operators on the Fock space. Let $\delta\in\NN$, we set:
\begin{align*}
H_\delta(t) = & \sum_{k>0}b_{-k}b_{k} + \sum_{s,\mu,\nu,k} t^s  s^{|\mu|+|\nu|}\sigma_1(s) k^2 a_{-\mu}a_\nu b_k \\
& + \sum_{s,\mu,\nu,k} t^s s^{|\mu|+|\nu|}\sigma_1(s) k^2 b_{-k}a_{-\mu}a_\nu , \\
\end{align*}
and its deformation
\begin{align*}
H^q_\delta(t) = & \sum_{k>0}b_{-k}b_{k} + \sum_{s,\mu,\nu,k} t^s  \left[\sum_{dl=s} d^{|\mu|+|\nu|}\prod \left(\frac{[li]_q}{[i]_q}\right)^{\mu_i+\nu_i}\right] k[k]_q a_{-\mu}a_\nu b_k \\
& + \sum_{s,\mu,\nu,k} t^s \left[\sum_{dl=s} d^{|\mu|+|\nu|}\prod \left(\frac{[li]_q}{[i]_q}\right)^{\mu_i+\nu_i}\right] k[k]_q b_{-k}a_{-\mu}a_\nu , \\
\end{align*}
where in each case, the first sum over the tuples $(s,\mu,\nu,k)$ is with $s,k>0$ and such that
$$\|\mu\|+s\delta=\|\nu\|+k,$$
while the second is over those such that
$$\|\mu\|+s\delta+k=\|\nu\|.$$
The various monomials in the function $H_\delta$ correspond to the various vertices that can occur in a floor diagram:
\begin{itemize}[label=-]
\item The term $b_{-k}b_k$ corresponds to a marked point on an elevator of weight $k$.
\item The term $a_{-\mu}a_\nu b_k$ corresponds to a floor with $|\mu|$ elevators leaving the floor, $|\nu|$ arriving on it, and a remaining elevator of weight $k$. In the setting of the floor diagrams, it is the unique elevator that goes in the direction of the unique unbounded end of the complement of marked points on elevators. As its position is fixed by the constraints, it behaves like a fixed elevator for the next floor that it meets.
\item Similarly for the term $b_{-k}a_{-\mu}a_\nu$ but with the specific elevator on the other side.
\item The positive integer $s$ corresponds to the weight of the floor. This way, the previous relation on the indices of the sum is just the condition on the divergence of a floor, which comes from the balancing condition. The $s^{|\mu|+|\nu|}\sigma_1(s)$ is the floor multiplicity appearing in Definition \ref{definition multiplicity floor diagram}, and the $k^2$ part of the multiplicity of the specific elevator. Their refined counterpart appears for the deformed operator $H_\delta^q(t)$.
\end{itemize}

\begin{rem}\label{rem caporaso operator}
Notice how the monomials in the expressions of the operators $H_\delta(t)$ and $H_\delta^q(t)$ correspond to the terms in the Caporaso-Harris formula.
\end{rem}

\begin{rem}
Sometimes, one includes an additional variable $u$, for instance in \cite{cooper2017fock}. It is used to recover the Euler characteristic of the graph encoded by the floor diagram. In the classical case, one then sets
\begin{align*}
\widetilde{H}_\delta(t,u) = & \sum_{k>0}b_{-k}b_{k} + \sum_{s,\mu,\nu,k} t^s u^{|\mu|} s^{|\mu|+|\nu|}\sigma_1(s) k^2 a_{-\mu}a_\nu b_k \\
& + \sum_{s,\mu,\nu,k} t^s u^{|\mu|+1} s^{|\mu|+|\nu|}\sigma_1(s) k^2 b_{-k}a_{-\mu}a_\nu , \\
\end{align*}
Precisely, the $b_{-k}b_k$ correspond to a half-closed interval, which is of Euler characteristic $0$, and each monomial in the other sum corresponds to a floor which encodes a graph of genus $1$, including the leaves of the elevators ending on the floor, but not the one that starts from the floor. The exponent of $u$ is then precisely the opposite of the Euler characteristic. However, this additional variable is not really needed due to the necessary relation between the number of marked points $n$, the genus $g$ and the bidegree $(d_1,d_2)$: $n=\delta d_1+2d_2+g-1$.
\end{rem}

\begin{rem}
Since the operators do not all commute, one needs to care about the order of the operators in the monomials. In particular, this prevents to right the two sums as a unique sum with $k\in\ZZ\backslash\{0\}$ since the position of the $b$ operator is important.
\end{rem}

We now state the main theorem of the section. Concretely, it means that the invariants, relative invariants, and their refined counterparts can be expressed as some coefficients of powers of the operators $H_\delta(t,u)$ (resp. $H^q_\delta(t,u)$) .

\begin{theo}\label{theorem coefficient fock operator}
We have
$$N^{\delta,\bullet}_{g,(d_1,d_2)} = \left\langle 0,1^{d_2} \right| \mathrm{Coeff}_{t^{d_1}} H_\delta(t)^{\delta d_1+2d_2+g-1} \left| 0,1^{\delta d_1+d_2}  \right\rangle, $$
$$BG^{\delta,\bullet}_{g,(d_1,d_2)} = \left\langle 0,1^{d_2} \right| \mathrm{Coeff}_{t^{d_1}} H^q_\delta(t)^{\delta d_1+2d_2+g-1} \left| 0,1^{\delta d_1+d_2}  \right\rangle, $$
and
$$N^{\delta,\bullet}_{g,(d_1,d_2)}(\mu_0,\mu_\infty,\nu_0,\nu_\infty) = \frac{\mu_\infty !}{I^{\mu_\infty+\nu_\infty}}\frac{\mu_0!}{I^{\mu_0+\nu_0}} \left\langle \mu_\infty,\nu_\infty \right| \mathrm{Coeff}_{t^{d_1}} H_\delta(t)^{|\nu_0|+|\nu_\infty|+g-1} \left| \mu_0,\nu_0  \right\rangle, $$
$$BG^{\delta,\bullet}_{g,(d_1,d_2)}(\mu_0,\mu_\infty,\nu_0,\nu_\infty) = \frac{\mu_\infty !}{I_q^{\mu_\infty+\nu_\infty}}\frac{\mu_0!}{I_q^{\mu_0+\nu_0}} \left\langle \mu_\infty,\nu_\infty \right| \mathrm{Coeff}_{t^{d_1}} H^q_\delta(t)^{|\nu_0|+|\nu_\infty|+g-1} \left| \mu_0,\nu_0  \right\rangle. $$
\end{theo}

The operator $H_\delta(t)$ is a power series in $t$, and $\mathrm{Coeff}_{t^{d_1} }$ means that we just take the corresponding coefficient. Similarly for the deformed version.

\begin{rem}
In view of Remark \ref{rem caporaso operator}, the Caporaso-Harris formula corresponds just to expanding one term of the product $H_\delta(t)^{\delta d_1+2d_2+g-1}$ in Theorem \ref{theorem coefficient fock operator}.
\end{rem}

\begin{rem}
With the $u$ variable to recover the genus, the first equality in the theorem becomes for instance
$$N^{\delta,\bullet}_{g,(d_1,d_2)} = \left\langle 0,1^{d_2} \right| \mathrm{Coeff}_{t^{d_1}u^{g+d_2-1}} H_\delta(u,t)^{\delta d_1+2d_2+g-1} \left| 0,1^{\delta d_1+d_2}  \right\rangle. $$

%$$N^{\delta,\bullet}_{g,(d_1,d_2)}(\mu_0,\mu_\infty,\nu_0,\nu_\infty) = \frac{\mu_\infty !}{I^{\mu_\infty+\nu_\infty}}\frac{I^{\mu_0}}{I^{\mu_0+\nu_0}} \left\langle \mu_\infty,\nu_\infty \right| \mathrm{Coeff}_{t^{d_1}u^{g+l(\mu_0)+l(\nu_0)-1}} H_\delta(u,t)^{l(\nu_0)+l(\nu_\infty)+g-1} \left| \mu_0,\nu_0  \right\rangle. $$
\end{rem}

\begin{rem}
Notice that in the expression of $H_\delta$, the assumption on the indices of each sum ensures degree of the monomial $a_{-\mu}a_\nu b_k$ (resp. $b_{-k}a_{-\mu}a_\nu$) is equal to $\delta s$. Thus, $H_\delta$ belongs to $\sum \H_{-\delta s}t^s$, and so do its powers $H_\delta^n$. In particular,
$$\left\langle 0,1^{d_2} \right| H_\delta(t)^{\delta d_1+2d_2+g-1} \left| 0,1^{\delta d_1+d_2}  \right\rangle
= t^{d_1}\left\langle 0,1^{d_2} \right| \mathrm{Coeff}_{t^{d_1}} H_\delta(t)^{\delta d_1+2d_2+g-1} \left| 0,1^{\delta d_1+d_2}  \right\rangle,$$
since every other term in the series is $0$ by degree consideration.
\end{rem}

As in \cite{cooper2017fock} and \cite{block2016fock}, these relations allow us to find some expressions of the generating series of the invariants.

\begin{coro}
One has the following generating series:
$$\sum_{d_1,d_2,g} N_{g,(d_1,d_2)}^{\delta,\bullet} \frac{Q^{\delta d_1+2d_2+g-1}}{(\delta d_1+2d_2+g-1)!} E_0^{d_1}F^{d_2}
= \left\langle e^{a_1 F} e^{QH_\delta(E_0)}e^{a_{-1}} \right\rangle ,$$

$$\sum_{d_1,d_2,g} BG_{g,(d_1,d_2)}^{\delta,\bullet} \frac{Q^{\delta d_1+2d_2+g-1}}{(\delta d_1+2d_2+g-1)!} E_0^{d_1}F^{d_2}
= \left\langle e^{a_1 F} e^{QH^q_\delta(E_0)}e^{a_{-1}} \right\rangle .$$
\end{coro}

\begin{proof}
We have the following. On the first line, we use Theorem \ref{theorem coefficient fock operator}, on the second we replace $|0,1^{\delta d_1+d_2}\rangle$ by $e^{a_{-1}}v_\emptyset$ since only one term of the exponential series contributes to the coefficient for degree reasons, and we set $n=\delta d_1+2d_2+g-1>0$.
\begin{align*}
 & \sum_{d_1,d_2,g} N_{g,(d_1,d_2)}^{\delta,\bullet} \frac{Q^{\delta d_1+2d_2+g-1}}{(\delta d_1+2d_2+g-1)!} E_0^{d_1}F^{d_2}  \\
 = & \sum_{d_1,d_2,g} \frac{Q^{\delta d_1+2d_2+g-1}}{(\delta d_1+2d_2+g-1)!} E_0^{d_1}F^{d_2} \left\langle 0,1^{d_2} \right| \mathrm{Coeff}_{t^{d_1}} H_\delta(t)^{\delta d_1+2d_2+g-1} \left| 0,1^{\delta d_1+d_2}  \right\rangle \\
  = & \sum_{d_1,d_2,n} \frac{1}{n!} E_0^{d_1}F^{d_2} \left\langle 0,1^{d_2} \right| \mathrm{Coeff}_{t^{d_1}} (QH_\delta(t))^{n} e^{a_{-1}}\left| 0,0  \right\rangle \\
  = & \sum_{d_2,n} F^{d_2} \left\langle 0,1^{d_2} \right| \frac{(QH_\delta(E_0))^{n}}{n!} e^{a_{-1}}\left| 0,0  \right\rangle \\
  = & \sum_{d_2} F^{d_2} \left\langle 0,1^{d_2} \right| e^{QH_\delta(E_0)} e^{a_{-1}}\left| 0,0  \right\rangle \\
  = & \left\langle e^{a_1 F} e^{QH_\delta(E_0)}e^{a_{-1}} \right\rangle .\\
\end{align*}
The computation is verbatim for refined invariants.
\end{proof}

Similarly, one can compute generating series for relative invariants. For $\mu$ a partition, let $U_0^\mu=\prod_i U_{0i}^{\mu_i}$ and similarly for $V_0$, $U_\infty$ and $V_\infty$. Then, as noticed in \cite{block2016fock}, for indeterminates $U$ and $V$, one has
$$\exp\left( \frac{1}{n}(b_{-n}U + a_{-n}V)\right)v_\emptyset = \sum_{\mu,\nu}\frac{1}{I^{\mu+\nu}}U^\mu V^\nu |\mu,\nu\rangle.$$

\begin{coro}
One has the following generating series:
\begin{align*}
 & \sum_{d_1,d_2,g,\mu,\nu} N_{g,(d_1,d_2)}^{\delta,\bullet}(\mu_0,\mu_\infty,\nu_0,\nu_\infty) \frac{Q^{|\nu_0|+|\nu_\infty|+g-1}}{(|\nu_0|+|\nu_\infty|+g-1)!} E_0^{d_1}F^{d_2} \\
 = & \left\langle e^{F\sum \frac{b_{n}U_{\infty n}+a_{n}V_{\infty n}}{n}} e^{QH_\delta(E_0)}e^{\sum \frac{b_{-n}U_{0n}+a_{-n}V_{0n}}{n}} \right\rangle . \\
\end{align*}
\end{coro}

To get the generating series of invariants for irreducible curves, as stated in \cite{block2016fock}, one only needs to take the logarithm. This comes from the fact that the number of reducible curves is obtained by taking union of different irreducible components.

\begin{coro}
One has the following relations:
$$\sum_{d_1,d_2,g} N^\delta_{g,(d_1,d_2)} \frac{Q^{\delta d_1+2d_2+g-1}}{(\delta d_1+2d_2+g-1)!} E_0^{d_1}F^{d_2}
= \log\left\langle e^{a_1 F} e^{QH_\delta(E_0)}e^{a_{-1}} \right\rangle ,$$

$$\sum_{d_1,d_2,g} BG^\delta_{g,(d_1,d_2)} \frac{Q^{\delta d_1+2d_2+g-1}}{(\delta d_1+2d_2+g-1)!} E_0^{d_1}F^{d_2}
= \log\left\langle e^{a_1 F} e^{QH^q_\delta(E_0)}e^{a_{-1}} \right\rangle .$$
\end{coro}

\begin{proof}
If $S$ is the generating series for irreducible invariants, then the generating series for reducible invariants with $r$ components is exactly $\frac{1}{r!}S^r$. Taking the sum, we get that the generating series for reducible invariants is exactly $e^S$. The result follows.
\end{proof}

\begin{rem}
Notice that in this setting, the superabundant loops are not considered as curves since they do not appear in any $\TT F_{0,\alpha}$. To get the corresponding generating series, one would have to multiply by their generating series. This corresponds to compute the invariants $N_{g,(d,0)}^{\delta,\bullet}$ for $\CC F_{0,\alpha}$, where $\alpha$ is a torsion element, and include them as possible components. These numbers are $0$ for generic $\alpha$, but might be non-zero if $\alpha$ is torsion.
\end{rem}

To prove Theorem \ref{theorem coefficient fock operator}, we introduce Feynman graphs. Their main interest is that they are related to floor diagrams relatively easily, and that they provide a combinatorial interpretation of the vaccum expectation of monomials in the Heisenberg algebra, thus relating floor diagrams counts to vacuum expectations of operators. We follow what happens in \cite{cavalieri2021counting}.

Let $P=m_\infty m_n \cdots m_1 m_0$ be a product of monomials in the Heisenberg algebra such that:
\begin{itemize}[label=$\bullet$]
\item all the $m_i$ are monomials in the $a_k$ and $b_k$,
\item $m_0$ contains only $a_k$ and $b_k$ with negative indices,
\item $m_\infty$ contains only $a_k$ and $b_k$ with positive indices,
\item for all monomials $m_i$, every operator with a negative index stands left of an operator with a positive index.
\end{itemize}
The goal is to compute the vacuum expectation $\langle P\rangle$. One associates a family of graphs called \textit{Feynman graphs} to any product $P$ in the following way:
\begin{itemize}[label=$\circ$]
\item To any monomial $m_i$, associate a germ of graph with exactly one vertex, a thick edge of weight $k$ leaving the vertex on the left (resp. right) for each $b_k$ (resp. $b_{-k}$), and a thin edge of weight $k$ to the left (resp. right) for each $a_k$ (resp. $a_{-k}$).
\item For each $b_{-k}$ (resp. $a_{-k}$) in $m_0$, we associate a germ with a thick (resp. thin) edge of weight $k$ pointing to the right.
\item For each $b_k$ (resp. $a_k$) in $m_\infty$, we associate a germ with a unique vertex and a thick (resp. thin) edge of weight $k$ pointing to the left.
\item The germs are ordered by the order on the indices of the $m_i$.
\item A Feynman graph is any (marked, weighted, ordered) graph obtained by gluing pairs of half-edges such that no half-edge is left alone. A gluing between two germs of edges can occur if
	\begin{itemize}[label=-]
	\item it connects a right half-edge to a left half-edge of a vertex on its right,
	\item one of the half-edge is thickened and the other is thin,
	\item they share the same weight.
	\end{itemize}
\end{itemize}

The following theorem is Proposition 5.1 in \cite{block2016fock}. It relates the vacuum expectation of a product $P$ to the Feynman graphs obtained from $P$.

\begin{prop}{(Wick's Theorem)} \label{proposition wick theorem}
The vacuum expectation $\langle P\rangle$ for a product of monomials $P$ is equal to the weighted sum of all Feynman graphs for $P$, where each Feynman graph is weighted by the product of weights of all edges (interior edges and ends).
\end{prop}

\begin{rem}
Notice that the weight of an edge is $k\in\N$ in the classical setting, and $[k]_q\in\ZZ[q^{\pm 1/2}]$ in the refined setting.
\end{rem}

\begin{proof}
We only give a sketch of proof and leave the bookkeeping task to the reader. When trying to compute the vacuum expectation, take any operator $a_k$ or $b_k$ with a positive entry. assume it is some $a_k$. We move it to the right part of the product. It commutes to every other operator except if it encounters some $b_{-k}$, where by the commutation relation we then have
$$a_k b_{-k}=k+b_{-k}a_k.$$
Thus, the vacuum expectation that we are trying to compute becomes a sum of two new vacuum expectations:
\begin{itemize}[label=-]
\item one where $a_k$ and $b_{-k}$ are deleted and replaced by a factor $k$,
\item one where we have replaced $a_kb_{-k}$ with $b_{-k}a_k$.
\end{itemize}
In the end, when $a_k$ is moved completely to the right, we get $a_k v_\emptyset=0$. Thus, we exactly get a sum of vacuum expectations for every $b_{-k}$ to the right of $a_k$. This corresponds to drawing an edge between the two germs of edges corresponding to $a_k$ and $b_{-k}$. Notice that the order between the monomials and the germs of edges on the graph is reversed since the graph is drawn from left to right when the monomials are read from right to left. Therefore, moving every monomial with a positive entry to the right, the vacuum expectation expresses as a sum over all the possibilities to connect an operator $a_k$ (resp. $b_k$) to a $b_{-k}$ (resp. $a_{-k}$) on its right. These are exactly the Feynman's graphs. They are counted with a weight equal to the product of the weights of the edges.
\end{proof}

\begin{proof}[Proof of Theorem \ref{theorem coefficient fock operator}]
We only treat the classical case. The refined case is handled the same way. We consider the following vacuum expectation:
$$\left\langle 0,1^{d_2}\right| \mathrm{Coeff}_{t^{d_1}}H_\delta(t)^{\delta d_1+2d_2+g-1} \left| 0,1^{\delta d_1+d_2} \right\rangle .$$
First, replace $|0,1^{\delta d_1+d_2}\rangle$ by $\frac{a_{-1}^{\delta d_1+d_2}}{(\delta d_1+d_2)!}|v_\emptyset\rangle$ and $\langle 0,1^{d_2}|$ by $\langle v_\emptyset |\frac{a_1^{d_2}}{d_2!}$. Then we expand the product $H_\delta(t)$, and remember only the $t^{d_1}$ term. Proposition \ref{proposition wick theorem} asserts that it can be expressed as a sum over suitable Feynman graphs, with some multiplicity coming from the coefficients of the monomials in the expression of $H_\delta$. More precisely, after the expansion of the product, one gets a sum of monomials $P$ of the form $m(P)m_\infty m_{\delta d_1+2d_2+g-1}\cdots m_1 m_0$, where
\begin{itemize}[label=$\circ$]
\item $m_\infty=a_1^{d_2}$,
\item $m_0=a_{-1}^{\delta d_1+d_2}$,
\item every other $m_i$ is either some $b_{-k}b_k$, some $b_{-k}a_{-\mu}a_\nu$ or some $a_{-\mu}a_\nu b_k$.
\item $m(P)$ is $\frac{1}{d_2!(\delta d_1+d_2)!}$ times the product of the coefficients of the chosen monomials $m_i$ in the expansion.
\end{itemize}
Clearly, the Feynman graphs for the product of monomials in the previous expansion correspond exactly to the marked floor diagrams from section \ref{subsection floor diagrams}. One just needs to check that their multiplicities are identical. The multiplicity from Definition \ref{definition multiplicity floor diagram} gives a multiplicity $w^{\val(F)-1}\sigma_1(w)$ for every floor $\F$ of weight $w$ in the diagram, and a product on the elevators of the diagram. On the Feynman graph side, we recover the floors multiplicities in the term $s^{|\mu|+|\nu|}\sigma_1(s)$. Using Proposition \ref{proposition wick theorem}, every edge of weight $k$ brings a contribution $k$. This matches the contribution in the multiplicity from Definition \ref{definition multiplicity floor diagram}, except for the edges that are not marked. The missing factor $k^2$ from these edges comes from the coefficients in $H_\delta(t)$.

The proof of the formula for relative invariants is readily the same but with more bookkeeping, which is thus left to the reader. The only difference is that one needs to add a factor $\frac{\mu_0!}{I^{\mu_0+\nu_0}}\frac{\mu_\infty!}{I^{\mu_\infty+\nu_\infty}}$ to compensate for the weight of the unbounded ends in the formula.
\end{proof}

\bibliographystyle{plain}
\bibliography{biblio}

\end{document}